\documentclass{article}

\usepackage[final, nonatbib]{neurips_2024}

\usepackage[utf8]{inputenc} 
\usepackage[T1]{fontenc}    

\usepackage[enable]{easy-todo}
\usepackage{here}           
\usepackage{wrapfig}        
\usepackage{xcolor}         
\usepackage{url}            
\usepackage{hyperref}       
\usepackage{mathtools}      
\usepackage{amsfonts}       
\usepackage{amsmath, amssymb, amsthm} 
\usepackage{nicefrac}       
\usepackage{microtype}      
\usepackage{booktabs}       
\usepackage{lipsum}         
\usepackage{multirow}
\usepackage{type1cm}        
\usepackage{bm}             
\usepackage{bbm}            
\usepackage{dsfont}         
\usepackage{cases}          
\usepackage{array}
\usepackage{float}
\usepackage{subcaption}     

    \theoremstyle{plain}
    \newtheorem{theorem}{Theorem}[section]

    \newtheorem{lemma}[theorem]{Lemma}
    \newtheorem{corollary}[theorem]{Corollary}
    \theoremstyle{definition}
    \newtheorem{definition}[theorem]{Definition}
    
    \theoremstyle{remark}

\def\nn{\nonumber \\ }

\def\mbr{\mathbb{R}} 
\def\mbq{\mathbb{Q}}
\def\mbz{\mathbb{Z}}
\def\mbn{\mathbb{N}}

\def\mcc{\mathcal{C}}

\def\mcf{\mathcal{F}}

\def\mcn{\mathcal{N}}
\def\mcO{\mathcal{O}}

\def\mcl{\mathcal{L}}

\def\de{\delta}
\def\ep{\epsilon}

\def\th{\theta}

\def\om{\omega}

\title{Physics-informed Neural Networks for Functional Differential Equations: Cylindrical Approximation and Its Convergence Guarantees} 

\author{%
  Taiki Miyagawa* \\
  Independent Researcher \\
  \texttt{chopin.grande.valse.brillatnte@gmail.com} \\
  \And
  Takeru Yokota* \\
  Interdisciplinary Theoretical and Mathematical Sciences Program (iTHEMS), RIKEN \\
  Center for Quantum Computing, RIKEN \\
  \texttt{takeru.yokota@riken.jp} \\ 
}

\begin{document}

\maketitle

\def\thefootnote{*}\footnotetext{These authors contributed equally to this work. {\color{red} Some contents are omitted due to arXiv's storage limit. Please refer to our full paper at OpenReview (NeurIPS 2024) or \url{https://github.com/TaikiMiyagawa/FunctionalPINN}.}}\def\thefootnote{\arabic{footnote}}
\begin{abstract}
We propose the first learning scheme for functional differential equations (FDEs).
FDEs play a fundamental role in physics, mathematics, and optimal control.
However, the numerical analysis of FDEs has faced challenges due to its unrealistic computational costs and has been a long standing problem over decades.
Thus, numerical approximations of FDEs have been developed, but they often oversimplify the solutions. 
To tackle these two issues, we propose a hybrid approach combining physics-informed neural networks (PINNs) with the \textit{cylindrical approximation}. 
The cylindrical approximation expands functions and functional derivatives with an orthonormal basis and transforms FDEs into high-dimensional PDEs. 
To validate the reliability of the cylindrical approximation for FDE applications, we prove the convergence theorems of approximated functional derivatives and solutions.
Then, the derived high-dimensional PDEs are numerically solved with PINNs.
Through the capabilities of PINNs, our approach can handle a broader class of functional derivatives more efficiently than conventional discretization-based methods, improving the scalability of the cylindrical approximation.
As a proof of concept, we conduct experiments on two FDEs and demonstrate that our model can successfully achieve typical $L^1$ relative error orders of PINNs $\sim 10^{-3}$.
Overall, our work provides a strong backbone for physicists, mathematicians, and machine learning experts to analyze previously challenging FDEs, thereby democratizing their numerical analysis, which has received limited attention. Code is available at \url{https://github.com/TaikiMiyagawa/FunctionalPINN}.
\end{abstract}

\section{Introduction} 
\label{sec: Introduction}
Functional differential equations (FDEs) appear in a wide variety of research areas \cite{venturi2018numerical, venturi2021spectral, rodgers2024tensor}. 
FDEs are partial differential equations (PDEs) involving functional derivatives, where a functional $F$ is a function of an input function $\th(x)$ to a real number, i.e., $F: \th \mapsto F([\th]) \in \mbr$, and a functional derivative is defined as the derivative of functional w.r.t. the input function at $x$, denoted by $\delta F([\th])/\delta \th(x)$.
FDEs play a fundamental role in Fokker-Planck systems \cite{fox1986functional}, turbulence theory \cite{monin2013statistical145}, quantum field theory \cite{peskin1995introduction_to_QFT}, mean-field games \cite{carmona2018probabilistic}, mean-field optimal control \cite{chow2019algorithm, ruthotto2020machine}, and unnormalized optimal transport \cite{gangbo2019unnormalized}.  
Major examples of FDEs include the Hopf functional equation in fluid mechanics, the Fokker-Planck functional equation in the theory of stochastic processes, and the functional Hamilton-Jacobi equation in optimal control problems in density spaces. 

Despite their wide applicability, numerical analyses of FDEs are known to suffer from significant computational complexity; therefore, numerical approximation methods have been developed over decades.
They include the functional power series expansion \cite{monin2013statistical145}, the Reynolds number expansion \cite{monin2013statistical145}, finite difference methods \cite{chow2019algorithm}, finite element methods \cite{rodgers2024tensor}, tensor decomposition methods \cite{venturi2018numerical, rodgers2024tensor}, and the cylindrical approximation \cite{baez1997functional, friedrichs1957integration, gikhman2004theory, van2008stochastic}.

\begin{wrapfigure}[15]{r}[0pt]{0.5\columnwidth}
    \centering
    \centerline{\includegraphics[width=0.5\columnwidth]{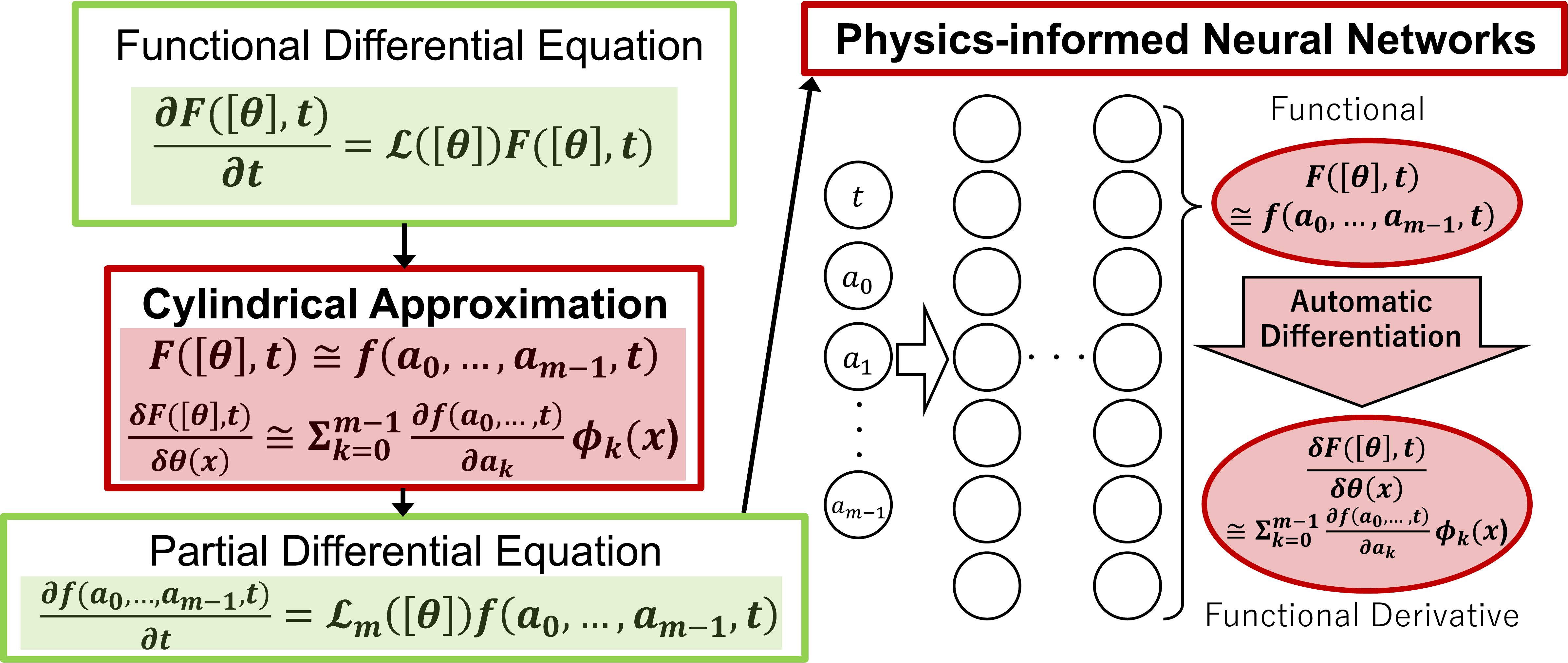}}
    \caption{\textbf{Overall architecture.} An FDE is simplified to a high-dimensional PDE via the cylindrical approximation. The PDE is solved with a PINN. The approximated functional derivative can be efficiently computed with automatic differentiation. 
    }
    \label{fig: figure1}
\end{wrapfigure}
However, they tend to oversimplify the solution of FDEs, prioritizing the reduction of computational costs.
The functional power series expansion is applicable only to input functions close to the expansion center. Moreover, it has no convergence guarantees in general \cite{monin2013statistical145}.
The Reynolds number expansion requires the Reynolds number to be close to zero, severely restricting its applicability, because the Reynolds number for turbulent flow can be $\gtrsim 1000$.
Discretization-based methods such as the finite difference and element methods restrict spacetime resolution and/or the class of input functions and functional derivatives \cite{rodgers2024tensor}.
Existing methods relying on the cylindrical approximation, akin to the spectral method for PDEs, include tensor decomposition \cite{venturi2018numerical, venturi2021spectral} to reduce computational costs; however, they tend to significantly simplify the class of input functions and functional derivatives. For instance, their expressivity is limited to polynomials or Fourier series of a few degrees.

To address the notorious computational complexity and limited approximation ability, we propose a hybrid approach combining physics-informed neural networks (PINNs) and the cylindrical approximation (Fig.~\ref{fig: figure1}).
In the first stage, we expand the input function with orthonormal basis functions, thereby transforming a given FDE into a high-dimensional PDE of the expansion coefficients.
This approximation is referred to as the \textit{cylindrical approximation}. 
We prove the convergence of the approximated functional derivatives and FDE solutions, validating the reliability of the cylindrical approximation for FDE applications, which is our main theoretical contribution.
In the second stage, the derived high-dimensional PDE is numerically solved with a PINN, which is known to be a universal approximator tailored to solve high-dimensional PDEs efficiently in a mesh-free manner.

A notable advantage of our approach is that, with the help of PINNs, it reduces the computational complexity by orders of magnitude, compared with previous discretization-based methods. 
In fact, it requires only $\mcO(m^r)$, where $m$ represents the ``class size'' of input functions and functional derivatives (e.g., the degree of polynomials), and $r$ ($\geq 1$) denotes the order of the functional derivative included in the target FDE (typically $1$ or $2$). 
This is a notable reduction from the state-of-the-art cylindrical approximation algorithm \cite{venturi2018numerical}, which requires as large as $\mcO(m^6)$.
Consequently, our approach substantially extends the class of input functions and functional derivatives that can be represented by the cylindrical approximation.
For instance, our approach extends the degrees of polynomials or Fourier series used for the approximation from $6$ \cite{venturi2018numerical} to $1000$, showing unprecedented expressivity.

As a proof of concept, we conduct experiments on two FDEs: the functional transport equation and the Burgers-Hopf equation. 
The results show that our model accurately approximates not only the solutions but also their functional derivatives, successfully achieving typical $L^1$ relative error orders of PINNs $\sim 10^{-3}$ \cite{cuomo2022scientific, blechschmidt2021three, das2022state,daw2023mitigating,jagtap2020adaptive,raissi2019physics_PINN_org,sharma2023stiff,wang2022is,wang2023learning}.

Our contribution is threefold. (1) We propose the first learning scheme for FDEs to address the significant computational complexity and limited approximation ability. Our model exponentially extends the class of input functions and functional derivatives that can be handled accurately and efficiently. (2) We prove the convergence of approximated functional derivatives and FDE solutions, ensuring the cylindrical approximation to be safely applied to FDEs. (3) Our experimental results show that our model accurately approximates not only the FDE solutions but also their functional derivatives.

\section{Related Work} 
\label{sec: Related Work}
FDEs are prevalent across numerous research fields \cite{fox1986functional, easther2003review, gangbo2019unnormalized, carmona2018probabilistic, chow2019algorithm, ruthotto2020machine, monin2013statistical145, peskin1995introduction_to_QFT}.
Research on FDEs has mainly focused on their theoretical aspects and formal solutions, with very few algorithms available to numerically solve general FDEs \cite{chow2019algorithm, rodgers2024tensor, venturi2018numerical, yokota2023physics}.
In \cite{chow2019algorithm}, a numerical method specialized for the Hamilton-Jacobi functional equation for optimal control problems in density space is proposed, based on spacetime discretization.
Similarly, \cite{rodgers2024tensor} employs spacetime discretization with tensor decomposition.
The state-of-the-art algorithm proposed in \cite{venturi2018numerical}, the CP-ALS (Canonical Polyadic tensor expansion \& Alternating Least Squares) algorithm, uses the cylindrical approximation along with the finite difference method and tensor decomposition, requiring $\mcO(m^6)$ (see App.~\ref{app: Computational Complexity of CP-ALS} for the derivation), whereas our model requires only $\mcO(m^r)$ ($r$ is $1$ or $2$ in most FDEs).
Furthermore, our model does not require such discretization, making it mesh-free.
See App.~\ref{app: Extra Introduction to FDEs} for an additional introduction to FDEs and their approximations.

The cylindrical approximation originates from the theory of stochastic processes \cite{gikhman2004theory, van2008stochastic}.
It is reminiscent of the spectral method for PDEs \cite{meuris2023machine} and is a generalization to FDEs.
Convergence theorems of the cylindrical approximation are summarized in a recent seminal paper \cite{venturi2021spectral}.
Note that the cylindrical approximation in this paper (Eq.~\eqref{eq: cylindrical approx of functional derivatives with omitting second term}) is different from the one in \cite{venturi2021spectral}, tailored for practical use. Consequently, our convergence theorems also differ from those in \cite{venturi2021spectral}. See App.~\ref{app: Functional Derivatives} for technical details.
See App.~\ref{app: Supplementary Related Work} for more comparisons with other studies.

\section{Proposed Approach} 
\label{sec: Theory: Cylindrical Approximation}
\subsection{Step 1: Cylindrical Approximation}
We first introduce the \textit{cylindrical approximation of functionals, functional derivatives, and FDEs}, beginning with the expansion of input functions and culminating in the transformation of FDEs into high-dimensional PDEs. Additionally, we prove the convergence theorems for this approximation. 
The rigorous mathematical background is reviewed in App.~\ref{app: Theorems Related to Cylindrical Approximation and Convergence} for interested readers.

Firstly, we define the \textit{cylindrical approximation of functionals} \cite{baez1997functional, friedrichs1957integration, gikhman2004theory, van2008stochastic}.
Any function $\th$ in a real separable Hilbert space $H$ can be represented uniquely in terms of an orthonormal basis $\{ \phi_k\}_{k=0}^\infty$ as $\th(x) = \sum_{k=0}^{\infty} a_k \phi_k(x)$, where $a_k := (\th, \phi_k)_H$ are the \textit{coefficients} (or spectrum) of $\th$ in terms of $\{\phi_k\}_{k \geq 0}$, and $(\cdot, \cdot)_H$ denotes the inner product of $H$. 
Substituting this expansion to functional $F([\th])$, we can define a multivariable function $f (\{a_k\}_{k=0}^\infty) := F([ \sum_{k=0}^{\infty} a_k \phi_k ])$ for any functional $F: H \rightarrow \mbr$. 
Truncating $k$ at $m-1 \in \mbz_{\geq 0}$ gives the cylindrical approximation of functionals: 
\begin{equation} \label{eq:1}
    f(\{a_k\}_{k=0}^{m-1}) := F([P_m \th]),    
\end{equation}
where $P_m$ is the projection operator s.t.~$P_m \th(x) := \sum_{k=0}^{m-1} a_k \phi_k (x)$, and $m$ is referred to as the \textit{degree} of approximation. 
See Thm.~\ref{thm: Uniform convergence of cylindrical approximation of functionals} and Thm.~\ref{thm: Convergence rate of cylindrical approximation of functionals} for the uniform convergence and convergence rate of this approximation, originally given by \cite{prenter1970weierstrass, venturi2021spectral}.

Secondly, we define the \textit{cylindrical approximation of functional derivatives}.
The functional derivative of $F$ w.r.t.~$\th$ at $x$ is defined as $\frac{\delta F ([\th])}{\delta \th (x)} := \underset{\ep \rightarrow 0}{\lim} \frac{ F([\th(y) + \ep \de(x - y)]) - F([\th(y)]) }{\ep}$, where $\de(x)$ denotes the Dirac delta function. This definition is impractical to simulate on computers with spacetime discretization; thus, we employ the expansion $\frac{\delta F([\th])}{\delta \theta(x)} = \sum_{k=0}^\infty (\frac{\delta F([\th])}{\delta \th}, \phi_k)_H \phi_k(x)$.
The expansion is possible because $\frac{\delta F ([\th])}{\delta \th (x)}$ itself is a function of $x$ in $H$ and thus can be represented as an orthonormal basis expansion.
Note that the expansion coefficients $(\frac{\delta F([\th])}{\delta \th}, \phi_k)_H$ are known to be equal to $\frac{\partial f}{\partial a_k}$ (see App.~\ref{app: Functional Derivatives} for the proof).
Hence, truncating $\th$ at $m-1$ gives the cylindrical approximation of functional derivatives:
\begin{equation} \label{eq: cylindrical approx of functional derivatives with omitting second term}
    P_m\frac{\delta F([P_m \th])}{\delta \theta(x)} = \sum_{k=0}^{m-1} \frac{\partial f}{\partial a_k} \phi_k (x) \,.
\end{equation}
Note that Eq.~\eqref{eq: cylindrical approx of functional derivatives with omitting second term} is different from the cylindrical approximation adopted in \cite{venturi2018numerical, venturi2021spectral}. They do not apply $P_m$ to $\delta F([P_m \theta]) / \delta \theta(x)$, and the emerging ``tail term'' $\Sigma_{k=m}^{\infty} (\delta F([\theta]) / \delta \theta, \phi_k) \phi_k(x)$ is simply ignored without any rationale.

The first main theoretical contribution of our work is the following convergence theorem of Eq.~\eqref{eq: cylindrical approx of functional derivatives with omitting second term}.
\begin{theorem}[Pointwise convergence of approximated functional derivatives (informal)] \label{thm: Pointwise convergence of cylindrical approximation (informal)}
    For arbitrary $\theta \in H$ and orthonormal basis $\lbrace \phi_0,\phi_1,\ldots \rbrace$, Eq.~\eqref{eq: cylindrical approx of functional derivatives with omitting second term} converges to $\frac{\delta F([\theta])}{\delta \theta}$ as $m \rightarrow \infty$.
\end{theorem}
The formal statement and proof are given in App.~\ref{app: Theorem: Pointwise Convergence of Functional Derivatives under Cylindrical Approximation}.
The convergence rate is the same as $\|\theta - P_m \theta \|$ if $\delta F([\theta]) / \delta \theta(x) \in \mathrm{span}\{\phi_0, \dots, \phi_{m-1}\}$.
A technical discussion when this is not the case is provided in App.~\ref{app: Proofs and Derivations}.

\begin{wrapfigure}[22]{r}[0pt]{0.45\columnwidth}
    \centering
    \centerline{\includegraphics[width=0.45\columnwidth]{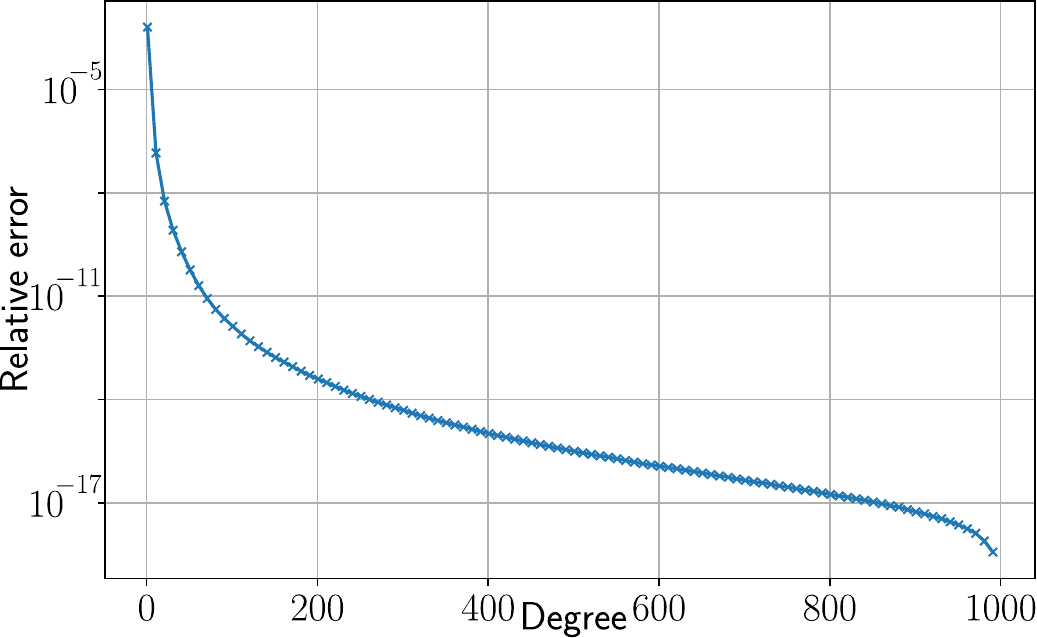}}
    \caption{\textbf{Cylindrical approximation of FDE's solution.} 
    The $L^1$ relative error, defined as $\frac{1}{N} \sum_{i=1}^{N} | F([\theta_i]) - F([P_m \theta_i]) | / | F([\theta_i]) |$, diminishes with increasing $m$. 
    The Burgers-Hopf equation with the delta initial condition (Sec.~\ref{sec: Burgers-Hopf Equation}) is considered. Note that PINN's training is not included.
    }
    \label{fig: fitting}
\end{wrapfigure}
Finally, we define the \textit{cylindrical approximation of FDEs} \cite{venturi2021spectral}.
In this paper, we consider the \textit{abstract evolution equation}, a class of FDEs having the following form: 
\begin{equation}
    {\partial F([\th], t)}/{\partial t} = \mcl([\th]) F([\th], t)
\end{equation}
with $F([\th], 0) = F_0([\th])$, where $\mcl([\th])$ is a linear functional operator, and $F_0$ is a given initial condition. 
The abstract evolution equation is a crucial class of FDEs in physics, mathematics, and engineering, including the Hopf functional equation, Fokker-Planck functional equation, and functional Hamilton-Jacobi equation.
The cylindrical approximation of the abstract evolution equation is given by 
\begin{align}
    {\partial f(\boldsymbol{a}, t)} / {\partial t} = \mcl_m([\th]) f(\boldsymbol{a}, t)     
\end{align}
with $f(\boldsymbol{a}, 0) =  f_0 (\boldsymbol{a})$, where $\boldsymbol{a}:= (a_0, \ldots, a_{m-1})^\top$, $f(\boldsymbol{a}, t) := F([P_m \theta], t)$, and $f_0 (\boldsymbol{a}) := F_0([P_m \th])$.
The operator $\mcl_m([\th])$ is the cylindrical approximation of $\mcl([\th])$. 
Examples are given in Secs.~\ref{sec: Functional Transport Equation} \& \ref{sec: Burgers-Hopf Equation}.

The second main theoretical contribution of our work is the following convergence theorem of solutions:
\begin{theorem}[Convergence of approximated solutions (informal)] \label{thm: main theorem (informal)}
    Under the cylindrical approximation (Eq.~\eqref{eq: cylindrical approx of functional derivatives with omitting second term}), if the FDE depends on functional derivatives only in the form of the inner product $(v, \frac{\delta F([\theta])}{\delta \theta})_H$ ($v \in H$),
    then, the solution of the approximated abstract evolution equation ($F([P_m \th], t)$) converges to the solution of the original one ($F([\th], t)$) as $m \rightarrow \infty$.
\end{theorem}
The proof is given in App.~\ref{app: Details about our FDEs}.
The convergence is visualized in Fig.~\ref{fig: fitting}.
Similar theorems for the FDEs with the second or higher-order functional derivatives can be shown in a similar way.
The inner-product assumption in Thm.~\ref{thm: main theorem (informal)} is satisfied by major FDEs, such as the Hopf functional equation, Fokker-Planck functional equation, and functional Hamilton-Jacobi equation.

In the following, we apply the cylindrical approximation to two FDEs: the functional transport equation (FTE) and the Burgers-Hopf equation (BHE).

\subsubsection{Application 1: Functional Transport Equation} 
\label{sec: Functional Transport Equation}

We first construct a simple FDE, the \textit{functional transport equation} (FTE), which is a generalization of the transport equation (the continuity equation) \cite{landau2013fluid_mechanics}. The FTE is provided by 
\begin{equation} \label{eq:FTE_main_text}
    \frac{\partial}{\partial t}F([\th],t) = -\int dx \, u(x)\frac{\delta F([\th],t)}{\delta \th(x)},
\end{equation}
with the initial condition $F([\th],0) = F_0([\th])$, where $x \in [-1, 1]$, and $u(x)$ is a given function. 
Specifically, we use the initial condition $F([\th],0) = \rho_0\int dx\, u(x)\th(x)$ with $\rho_0$ a constant.
The exact solution is $F([\th],t) = F_0([\th-ut]) = \rho_0\int dx\, u(x) (\th(x)-u(x)t) $, as can be seen by substituting this into Eq.~\eqref{eq:FTE_main_text}.
More details and motivations of the FTE are provided in App.~\ref{app: Functional Transport Equation}.

The cylindrical approximation of the FTE is given by 
\begin{equation} \label{eq: FTE CA}
    \frac{\partial}{\partial t}f(\boldsymbol{a},t) = - \sum_{k=0}^{m-1} u_k \frac{\partial}{\partial a_k}f(\boldsymbol{a},t)
\end{equation}
with the initial condition $f(\boldsymbol{a},0)=f_0(\boldsymbol{a})$, 
where 
$\th(x) = \sum_{k=0}^{m-1} a_k \phi_k(x)$, 
and $f_0(\boldsymbol{a}) = F_0([P_m \th]) = \rho_0 \Sigma_{k=0}^{m-1} u_k a_k$ with $u_k := (u, \phi_k)_{L^2([-1, 1])} = \int dx u(x) \phi_k (x)$.
We use the Legendre polynomials as the orthonormal basis $\{\phi_k\}_{k \geq 0}$.
The exact solution of Eq.~\eqref{eq: FTE CA} is $f(\boldsymbol{a}, t) = \rho_0\sum_{k=0}^{m-1} u_k (a_k - u_k t)$.

In our experiments in Sec.~\ref{sec: Experiment}, we consider two types of FTEs: (i) $u_k := \upsilon_0$ for $k=1$ (0 otherwise) and (ii) $u_k := \upsilon_0$ for $k \leq 14$ (0 otherwise), where $\upsilon_0$ is a constant.
For convenience, we call them the \textit{linear} and \textit{nonlinear initial conditions}, respectively.
The high-dimensional PDE thus obtained (Eq.~\eqref{eq: FTE CA}) is solved with PINNs (Sec.~\ref{sec: PINN}). 

\subsubsection{Application 2: Burgers-Hopf Equation} 
\label{sec: Burgers-Hopf Equation}
The second FDE is the \textit{Burgers-Hopf equation} (BHE), a crucial equation in turbulence theory:
\begin{equation} \label{eq: FTE org}
    \frac{\partial F([\th],t)}{\partial t} = \int dx \th(x)\frac{\partial^2}{\partial x^2}\frac{\delta F([\th],t)}{\delta \th(x)}     
\end{equation}
with the initial condition $F([\th],0)=F_0([\th])$, where $x \in [-1/2, 1/2]$.
Specifically, we use the Gaussian initial condition $F_0([\th]) = - \overline{\mu} \int dx  \th(x) + \frac{1}{2} \int dx \int dx' C(x,x') \th(x) \th(x')$, where $\overline{\mu}$ is a constant, and $C(x, x')$ is the infinite-dimensional covariance matrix. 
The exact solution is $F([\th],t) = F_0([\Theta])$, where $\Theta([\th], x, t) := \frac{1}{\sqrt{4\pi t}}
\int_{-\infty}^{\infty}dx' e^{-\frac{1}{4t} (x-x')^2} \th(x')$.
The derivation is provided in App.~\ref{app: Derivation of Eq. BH exact solution}.
Strictly speaking, Eq.~\eqref{eq: FTE org} is a modification of the original BHE.
The modification includes making the BHE dimensionless and neglecting the advection term.
For more technical details, see App.~\ref{app: Burgers-Hopf Equation}.

The cylindrical approximation of the BHE is given by
\begin{equation} \label{eq:BHE_CA_main_text}
    \frac{\partial}{\partial t} f(\boldsymbol{a}, t) = \sum_{k=0}^{m-1} \sum_{l=0}^{m-1} \int dx \frac{\partial \phi_k (x)}{\partial x^2} \phi_l(x)  a_k \frac{\partial}{\partial a_l}f(\boldsymbol{a}, t)
\end{equation}
with the initial condition $f(\boldsymbol{a},0) = f_0(\boldsymbol{a})$, where $\theta(x) = \sum_{k=0}^{m-1} a_k \phi_k(x)$, 
$f_0(\boldsymbol{a}) = F_0([P_m \th]) = - \bar{\mu} a_0 + \frac{1}{2} \Sigma_{k,l = 0}^{m-1} \tilde{C}_{kl} a_k a_l$ with $\tilde{C}_{kl} := \int dx \int dx' C(x,x') \phi_{k}(x) \phi_{l}(x')$.
We use the Fourier series as the orthonormal basis: $\{\phi_{k}(x)\}_{k \geq 0} = \{1, \sqrt{2}\sin(\pi k x), \sqrt{2}\cos(\pi kx)\}_{k \geq 1}$.
Then, the exact solution under the cylindrical approximation is given by
\begin{align} \label{eq: BHE sol main text}
    &f(\boldsymbol{a},t) =  - \overline{\mu} a_{0} +
    \Sigma_{k, l = 0}^{M-1}
    \big(
    e^{-4\pi^2 (k^2+l^2) t}
    \tilde{C}_{2k,2l}
    a_{2k} a_{2l} \nn
    &+
    e^{-4\pi^2 (k^2+(l+1)^2) t}
    \tilde{C}_{2k,2l+1}
    a_{2k} a_{2l+1} 
    +
    e^{-4\pi^2 ((k+1)^2+l^2)t}
    \tilde{C}_{2k+1,2l}
    a_{2k+1} a_{2l} \nn
    &+
    e^{-4\pi^2 ((k+1)^2+(l+1)^2)t}
    \tilde{C}_{2k+1,2l+1}
    a_{2k+1}
    a_{2l+1}
    \big)/2 \, ,
\end{align}
where $m = 2M$ ($M \in \mbn$). 
The derivation is given in App.~\ref{app: Derivation of BH exact solution in exp}.

In our experiments in Sec.~\ref{sec: Experiment}, we adopt two types of the covariance matrices: 
(i) $\tilde{C}_{ij} = \sigma^2$ for all $i = j \geq 0$ (0 otherwise) and
(ii) $\tilde{C}_{ij} = \sigma^2$ for $i = j = 0$ (0 otherwise), where $\sigma^2$ is a constant. 
Substituting (i) and (ii) into $f_0$, we have two types of initial conditions, which we call the \textit{delta and constant initial conditions}, respectively.
Again, the high-dimensional PDE thus obtained (Eq.~\eqref{eq:BHE_CA_main_text}) is solved with PINNs (Sec.~\ref{sec: PINN}).

\subsection{Step 2: Solving Approximated FDEs with PINNs} 
\label{sec: PINN}
\begin{wrapfigure}[11]{r}[0pt]{0.5\columnwidth}
 \centering
    \centerline{\includegraphics[width=0.5\columnwidth]{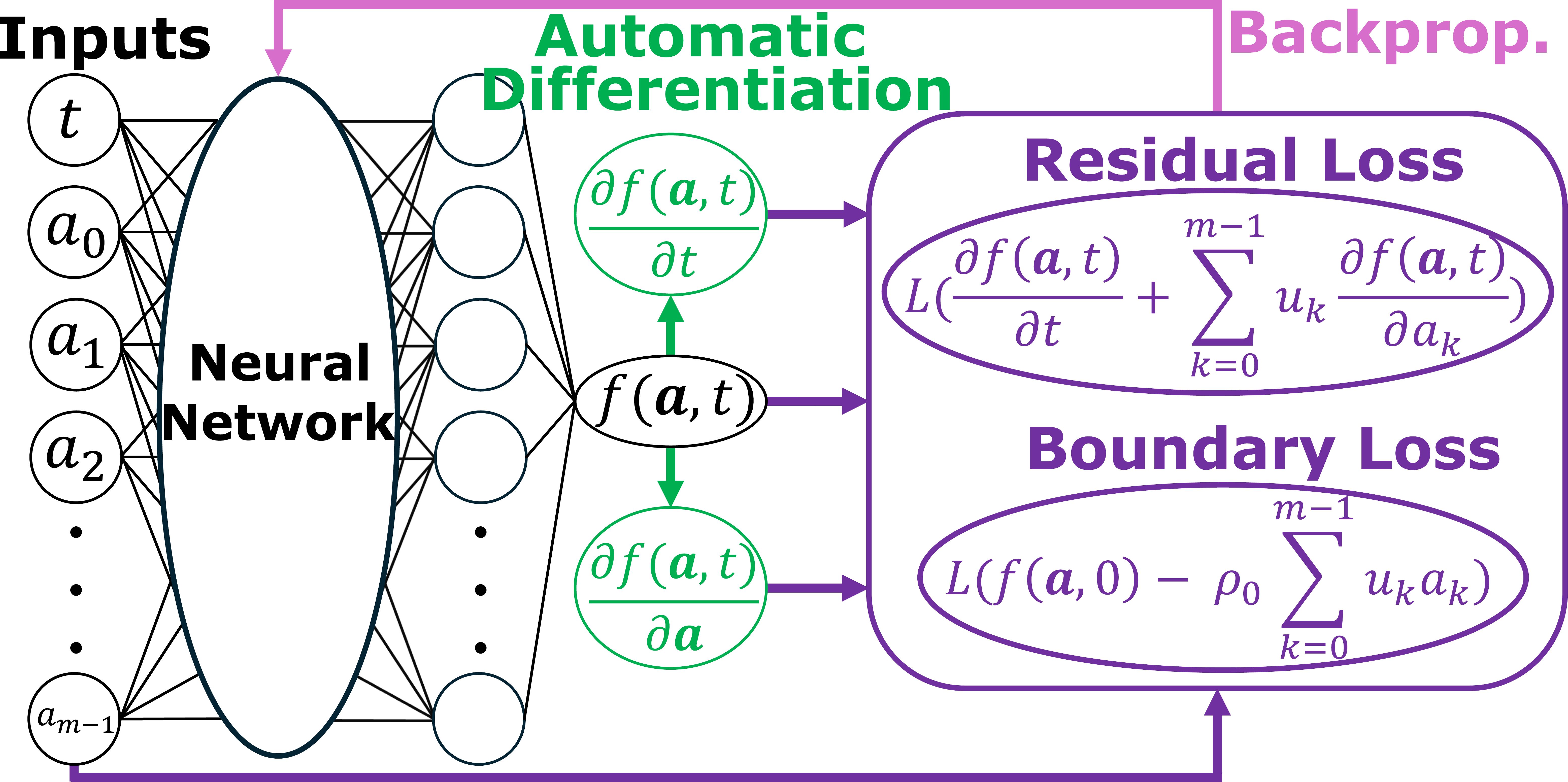}}
    \caption{\textbf{PINN's architecture.}}
    \label{fig: PINN}
\end{wrapfigure}
We briefly introduce the foundation of PINNs \cite{raissi2019physics_PINN_org}.
PINNs are universal approximators and can solve PDEs.
Let us consider a PDE $\partial f(t,x)/\partial t = \mathcal{N}[f]$ with an initial condition $\mathcal{B}[f]|_{t=0} = 0$, where $t \in [0, 1]$, $x \in [-1, 1]$. $\mathcal{N}$ and $\mathcal{B}$ are operators defining the PDE and the initial condition, respectively.
The PINN aims to approximate the solution $f(t,x)$.
Thus, the inputs to the PINN are $t$ and $x$, randomly sampled from $[0, 1]$ and $[-1, 1]$, respectively. 
Note that $(t=0, x)$ with $x \in [-1, 1]$ are also input to the PINN to compute the boundary loss. 
The inputs are transformed through linear layers and activation functions. 
The final output of the PINN is an approximation of $f(t, x)$, denoted by $\hat{f}(t, x)$. 
The loss function is the weighted sum of the residual loss $\| \partial \hat{f}(t,x)/\partial t - \mathcal{N}[\hat{f}] \|$ and the boundary loss $\| B[\hat{f}]|_{t=0} \|$, where $\| \cdot \|$ is a certain norm. 
The partial derivatives in the loss function can be computed via automatic differentiation of the PINN's output $\hat{f}$ w.r.t. the inputs $(t, x)$. 
Finally, the weight parameters of the PINN are minimized through backpropagation.

Next, we explain how to apply PINNs to our high-dimensional PDEs: see Fig.~\ref{fig: PINN}. 
For concreteness, consider the FTE (Eq.~\eqref{eq: FTE CA}) with the linear initial condition. 
The inputs to the PINN are $(t, a_0, a_1, \ldots, a_{m-1})$ and $(0, a_0, a_1, \ldots, a_{m-1})$, randomly sampled from finite intervals, and the outputs are $\hat{f}(\boldsymbol{a}, t)$ and $\hat{f}(\boldsymbol{a}, 0)$, respectively. 
Then, $\partial \hat{f}(\boldsymbol{a}, t) / \partial t$ and $\partial \hat{f}(\boldsymbol{a}, t) / \partial \boldsymbol{a}$ are computed via automatic differentiation, and we obtain the residual loss $\| \partial \hat{f}(\boldsymbol{a}, t) / \partial t + \Sigma_{k=0}^{m-1} u_k \partial \hat{f}(\boldsymbol{a},t) / \partial a_k \|$ and the boundary loss $\| \hat{f}_0(\boldsymbol{a}) - \rho_0 \Sigma_{k=0}^{m-1} u_k a_k \|$, where $\hat{f}_0(\boldsymbol{a}) := \hat{f}(\boldsymbol{a}, 0)$. 
These losses are minimized via mini-batch optimization.

\begin{wraptable}[30]{r}[0pt]{0.5\columnwidth}
    \caption{\textbf{Mean relative and absolute errors of models trained on FTEs under linear (top two) and nonlinear (bottom two) initial conditions.} I.C. is short for ``initial condition''. The error bars are the standard deviation over 10 training runs with different random seeds. 
    Note that this table is not for the assessment of the theoretical convergence of the cylindrical approximation (see Fig.~\ref{fig: fitting}, App.~\ref{app: Cross-degree Evaluation}, and the footnote in Sec.~\ref{sec: Result FTE} instead).
    } 
    \label{tab: Mean relative and absolute error (FTE)}
    \centering
        \begin{subtable}{0.5\columnwidth}
            \begin{small}
            \begin{sc}
            \begin{tabular}{@{}rlll@{}} 
                \toprule
                Degree & Relative error (Linear I.C.)  \\ 
                \midrule
                4      &  $(1.26820 \pm 0.31421) \times 10^{-3}$   \\ 
                20     &  $(2.01716 \pm 0.21742)  \times 10^{-3}$  \\ 
                100    &  $(6.24740 \pm 0.33492)   \times 10^{-3}$ 
                \vspace{0.04 in} \\ 
                \toprule
                Degree & Absolute error (Linear I.C.)  \\ 
                \midrule
                4      &   $(1.32203 \pm 0.44061)\times 10^{-3}$   \\ 
                20     &   $(2.29632 \pm 0.16459) \times 10^{-3} $  \\ 
                100    &   $(1.23312 \pm 0.18931) \times 10^{-3}$  
                \vspace{0.04 in} \\ 
                \toprule
                Degree & Relative error (Nonlinear I.C.)  \\ 
                \midrule
                4      &  $(1.79295 \pm 0.28535) \times 10^{-3}$   \\ 
                100     & $(7.63769 \pm 0.90872)  \times 10^{-3}$  \\ 
                1000   &  $(8.27096 \pm 1.19378)   \times 10^{-3}$ 
                \vspace{0.04 in} \\ 
                \toprule
                Degree & Absolute error (Nonlinear I.C.) \\ 
                \midrule
                4      &  $(2.37627 \pm 0.15278) \times 10^{-4}$   \\ 
                100     & $(1.84506 \pm 0.15765)  \times 10^{-3}$  \\ 
                1000   &  $(1.76470 \pm 0.36885)  \times 10^{-3}$
            \end{tabular}
            \end{sc}
            \end{small}
        \end{subtable}
\end{wraptable}
\paragraph{Computational Complexity}
The total computational complexity w.r.t. $m$ up to the computation of functional derivatives is given by $\mcO(m) + \mcO(m^r) = \mcO(m^r)$, where $r \geq 1$ is the order of the functional derivative included in the target FDE (typically $1$ or $2$). 
The first term $\mcO(m)$ comes from the input layer of the PINN.
The second term $\mcO(m^r)$ comes from the computation of functional derivatives under the cylindrical approximation (Eq.~\eqref{eq: cylindrical approx of functional derivatives with omitting second term}).
See App.~\ref{app: Computational Complexity of Loss Function} for more detailed discussions on computational complexity.

This is a notable reduction from discretization-based methods such as finite difference and element methods, which typically require exponentially large computational complexity w.r.t. the dimension of PDE $m$.
Also, $\mcO(m^r)$ is significantly smaller than the state-of-the-art cylindrical approximation algorithm, the CP-ALS \cite{venturi2018numerical}, which requires $\mathcal{O}(m^6)$ (the derivation is given in App.~\ref{app: Computational Complexity of CP-ALS}).
Consequently, given that $m$ represents the ``class size'' of input functions and functional derivatives (Eqs.~\eqref{eq:1} \& \eqref{eq: cylindrical approx of functional derivatives with omitting second term}), our approach significantly extends the range of input functions and functional derivatives that can be represented via the cylindrical approximation.
In fact, our approach extends the degrees of polynomials or Fourier series used for the approximation from $6$ \cite{venturi2018numerical} to $1000$ (Sec.~\ref{sec: Experiment}).

Finally, we note that the selection of basis functions influences computational efficiency. The choice depends on the specific FDE, boundary conditions, symmetry, function spaces, and numerical stability. For further discussions, see App.~\ref{app: Orthogonal Bases Compared}). 

In summary, our proposed approach transforms an FDE into a high-dimensional PDE using cylindrical approximation and then solves it with a PINN, which serves as a universal approximator of the solution (Figs.~\ref{fig: figure1}~\&~\ref{fig: PINN}). 
It is important to note that our model employs the basic PINN framework, allowing for seamless integration with any techniques developed within the PINN community.

\begin{table}[t] 
\caption{\textbf{Mean relative (top) and absolute (bottom) errors.} The models are trained on the BHE. 
The error bars are the standard deviation over 10 training runs with different random seeds.
Note that this table is not for the assessment of the theoretical convergence of the cylindrical approximation (see Fig.~\ref{fig: fitting}, App.~\ref{app: Cross-degree Evaluation}, and the footnote in Sec.~\ref{sec: Result FTE} instead).
} 
\label{tab: Mean relative (top) and absolute (bottom) error (BHE).}
    \centering
    \begin{center}
        \begin{subtable}{0.7\textwidth}
        \begin{small}
        \begin{sc}
            \begin{tabular}{@{}rcc@{}} 
                \toprule
                \multicolumn{3}{c}{\quad \quad Initial conditions} \\ 
                \cmidrule(l){2-3}
                Degree & Delta  & Constant \\ \midrule
                4  & $(2.93905 \pm 0.17403) \times 10^{-4}$ & $(12.1782 \pm 8.54758) \times 10^{-5}$ \\ 
                20 & $(2.20842 \pm 0.28531) \times 10^{-4}$ & $(6.14352 \pm 1.28641) \times 10^{-5}$ \\ 
                100& $(2.41667 \pm 0.25264) \times 10^{-4}$ & $(5.50375 \pm 1.75507) \times 10^{-5}$ \\ 
            \end{tabular}
        \end{sc}
        \end{small}
        \end{subtable}
    \end{center}
    \begin{center}
        \begin{subtable}{0.7\textwidth}
        \begin{small}
        \begin{sc}
            \begin{tabular}{@{}rcc@{}} 
                \toprule
                \multicolumn{3}{c}{\quad \quad Initial conditions} \\ 
                \cmidrule(l){2-3}
                Degree & Delta  & Constant \\ \midrule
                4  & $(1.66451 \pm 0.47547) \times 10^{-5}$ & $(1.62849 \pm 0.66896) \times 10^{-5}$ \\ 
                20 & $(1.34640 \pm 0.12105) \times 10^{-5}$ & $(1.12371 \pm 0.32367) \times 10^{-5}$ \\ 
                100& $(1.60980 \pm 0.08952) \times 10^{-5} $ & $(1.05558 \pm 0.16673) \times 10^{-5}$ \\ 
            \end{tabular}
        \end{sc}
        \end{small}
        \end{subtable}
    \end{center}
\end{table}

\section{Experiment} 
\label{sec: Experiment}
\begin{wrapfigure}[31]{r}[0pt]{0.5\columnwidth}
 \centering
    \centerline{\includegraphics[width=0.5\columnwidth]{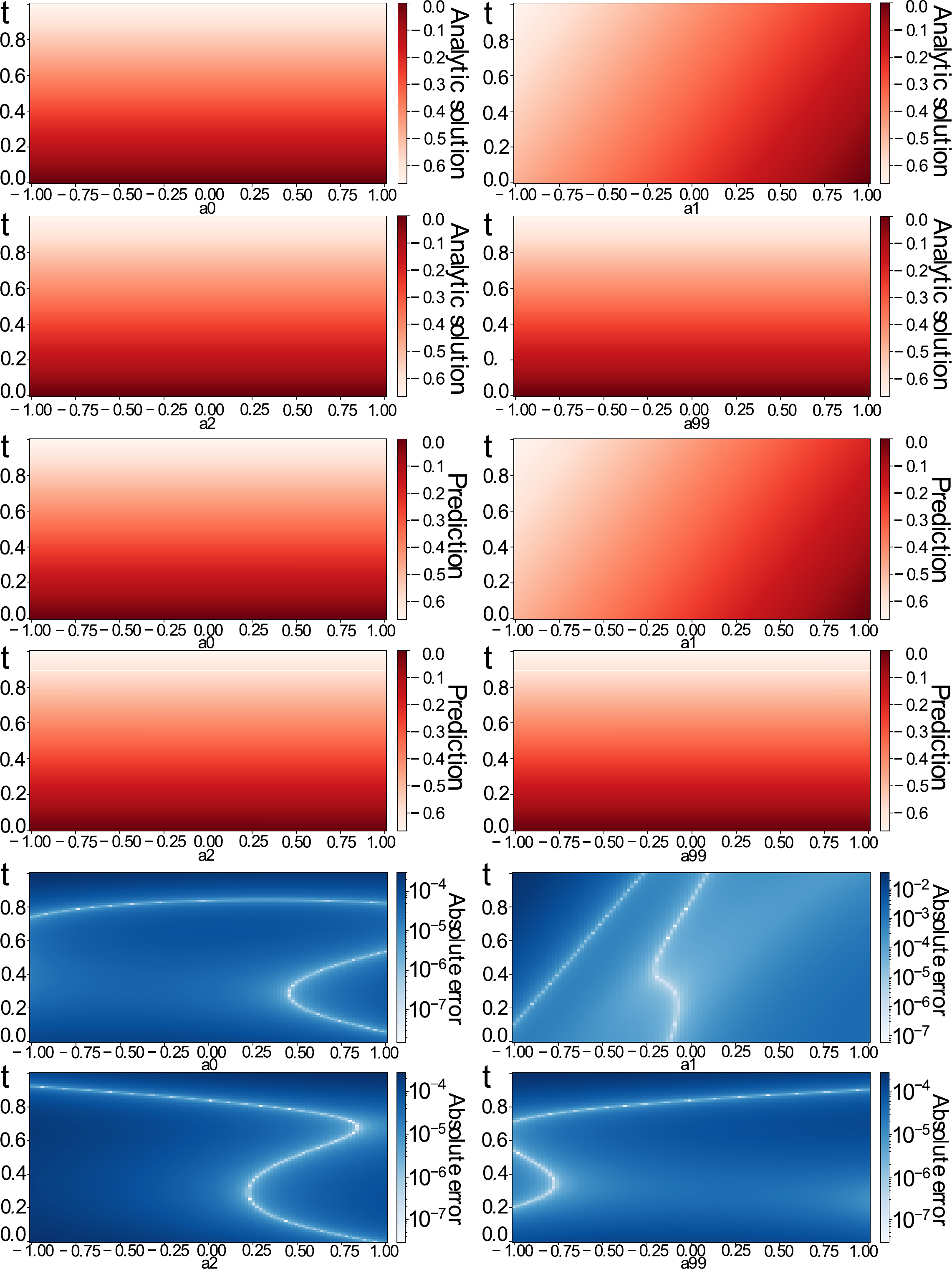}}
    \hfill
    \caption{
        \textbf{Analytic solution (top four panels), prediction (second four panels), and absolute error (bottom four panels) of FTE with degree 100 under linear initial condition.} 
        The horizontal axes represent $a_k$ for $k=0, 1, 2, 99$, with all the other coefficients set to 0. 
        Our model successfully learns the FTE.
    }
    \label{fig: fig_heat_AnalSolPredAbsWithTauVsAk_FTE_deg100_w1024_it500k_stat_model0_FTE}
\end{wrapfigure}
As a proof of concept for our approach, we numerically solve the FTE and BHE.\footnote{Code: \url{https://github.com/TaikiMiyagawa/FunctionalPINN}.} 
These two FDEs are suitable for numerical experiments because their analytic solutions are available, allowing for the computation of absolute and relative errors, major metrics in the numerical analysis of PDEs and FDEs. 
Note that the analytic solutions for most FDEs are currently unknown due to their mathematical complexity.

\paragraph{Setups.} 
We use a 4-layer PINN with $3 \times$ (linear + sin activation + layer normalization \cite{ba2016layer_LayerNorm_layer_normalization_org}) + last linear layer.
The total loss function is the smooth $L^1$ loss or the sum of $L^1$ and $L^\infty$ losses.
Softmax loss-reweighting is employed.
The optimizer is AdamW \cite{loshchilov2018AdamW_SGDW}.
The learning rate scheduler is the linear warmup with cosine annealing with warmup \cite{loshchilov2017sgdr_cosine_annealing_org}.
Latin hypercube sampling \cite{mckay2000comparison, helton2003latin} is used for the training, validation, and test sets.
For the BHE, the sampling range is decayed quadratically in terms of $k \in \{0,1,\ldots,m-1\}$ to stabilize the training.
We use $L^1$ relative and absolute errors, standard performance metrics for numerical analysis of PDEs and FDEs.
Absolute error $\frac{1}{N}\Sigma_{i=1}^{N}| f(\boldsymbol{a}_i, t_i) - \hat{f}(\boldsymbol{a}_i, t_i)|$ is used instead of relative error $\frac{1}{N}\Sigma_{i=1}^{N}| f(\boldsymbol{a}_i, t_i) - \hat{f}(\boldsymbol{a}_i, t_i)|/|f(\boldsymbol{a}_i, t_i)|$ when the analytic solution is close to zero because relative error in such a region blows up by definition, regardless of the model's prediction.
$\upsilon_0$, $\rho_0$, $\bar{\mu}$, and $\sigma^2$ are set to $1, 1, 8$, and $10$, respectively.
In App.~\ref{app: Detailed Experimental Settings}, we provide more detailed setups for reproducibility, including the range of $a_k$.

\begin{figure}[tb]
    \centering
      \centering
      \includegraphics[width=.75\textwidth]{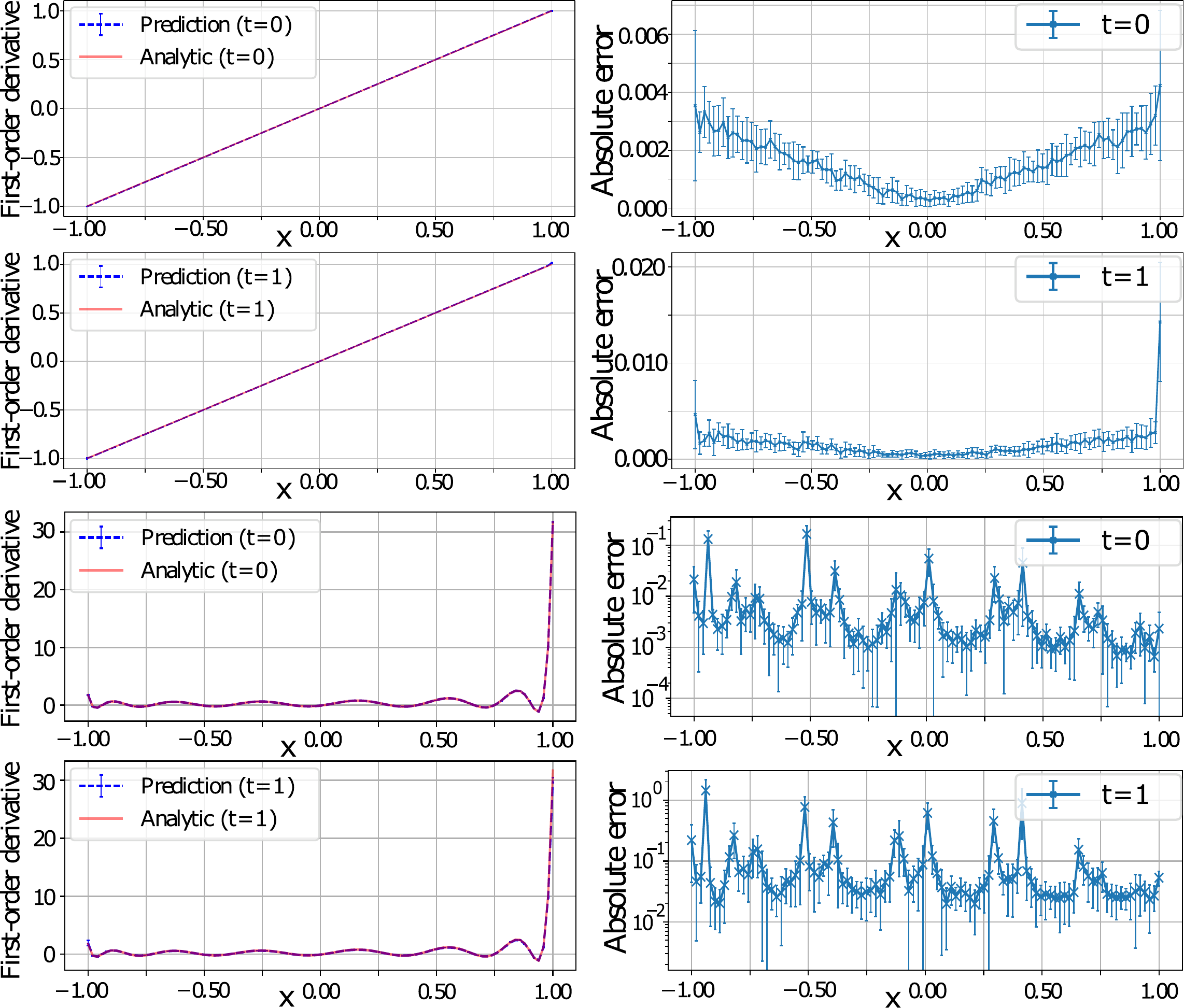}
      \captionof{figure}{
          \textbf{Absolute error of first-order functional derivative of FTE with degree 100 (top) and 1000 (bottom) under linear (top) and nonlinear (bottom) initial conditions.} The error bars represent the standard deviation over 10 runs with different random seeds.
      }
      \label{fig: fig_edit_plot_FirstMomentAbs_wERRBARdeg100_20_w1024_it500k_LinICICspectrum15_stat_model}
\end{figure}

\begin{figure}[H]
      \centering
      \includegraphics[width=.75\textwidth]{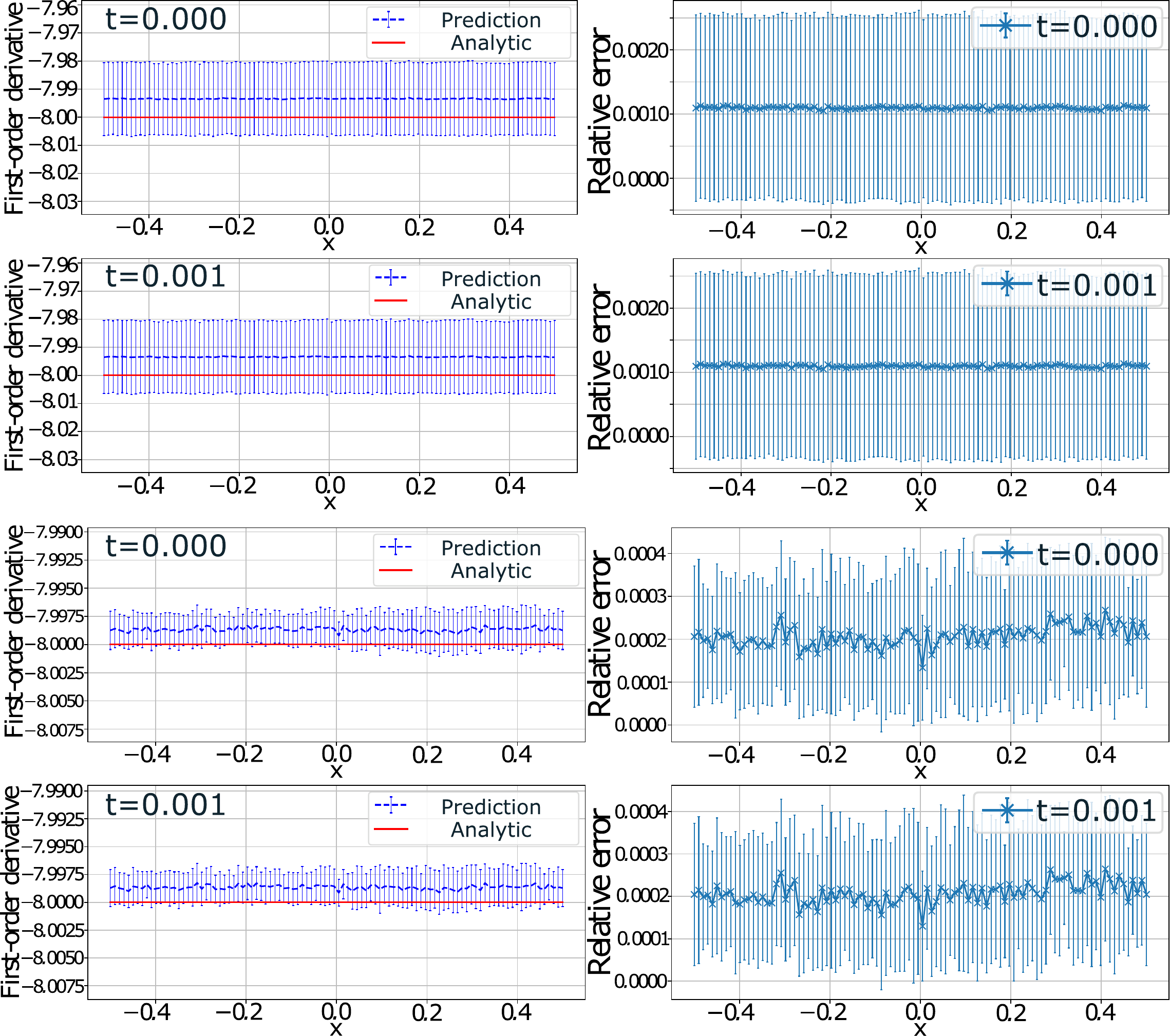}
      \captionof{figure}{
            \textbf{Relative error of first-order functional derivative of BHE with degree 100 under delta initial condition.}
            The error bars represent the standard deviation over 10 runs.
            The top/bottom four panels show the results with/without the loss function $\| W([0], t)\|$, respectively.
            With this loss function, the error reduces by $10^{-1}$.}
      \label{fig: fig_edit_plot_FirstMomentRel_wERRBARinit050_sig1e1_WandWObcW0_stat_model}
\end{figure}

\subsection{Result : Functional Transport Equation} 
\label{sec: Result FTE}
\textbf{Tab.~\ref{tab: Mean relative and absolute error (FTE)}} shows the main result.\footnote{
    Note that Tabs.~\ref{tab: Mean relative and absolute error (FTE)} \& \ref{tab: Mean relative (top) and absolute (bottom) error (BHE).} are not for the assessment of the theoretical convergence of the cylindrical approximation, unlike Fig.~\ref{fig: fitting}, because the analytic solutions used for error computation vary across each row, depending on the degree. 
    See App.~\ref{app: Cross-degree Evaluation} instead, where we additionally perform a cross-degree evaluation.
}
Our model achieves typical relative error orders of PINNs $\sim 10^{-3}$, even when the degree is as large as $1000$, which means our model's capability of representing $\theta$ and $\delta F([\theta], t)/ \delta \theta (x)$ as polynomials of degree $1000$. This is a notable improvement from the state-of-the-art algorithm \cite{venturi2018numerical}, which can handle $m\leq6$.

\textbf{Fig.~\ref{fig: fig_heat_AnalSolPredAbsWithTauVsAk_FTE_deg100_w1024_it500k_stat_model0_FTE}} visualizes the analytic solution, prediction, and absolute error of a model trained on the FTE with degree 100 under the linear initial condition.
Note that some of the collocation points used for plotting Fig.~\ref{fig: fig_heat_AnalSolPredAbsWithTauVsAk_FTE_deg100_w1024_it500k_stat_model0_FTE} are not in the training set, as can be seen from Figs.~\ref{fig: fig_hist_NormOfCoeffs_TestSet_FTE_de4_w1024_it500k_stat_FTE}--\ref{fig: fig_hist_NormOfCoeffs_TestSet_BHE_init1_deg50_sig1e1_stat_BHE} in App.~\ref{app: Latin Hypercube Sampling and Curse of Dimensionality}. 
This aspect highlights the model's ability to extrapolate beyond its training data ($a_k \sim 0$).

\textbf{Fig.~\ref{fig: fig_edit_plot_FirstMomentAbs_wERRBARdeg100_20_w1024_it500k_LinICICspectrum15_stat_model}} shows the absolute error of the first-order functional derivative estimated at $t=1$ and $\th=0$.
Again, $\th = 0$ is not included in the training set.
The errors in the top four panels increase at the edges of intervals ($x = \pm 1$) due to Runge's phenomenon \cite{runge1901runge_phenon_org}.

\subsection{Result: Burgers-Hopf Equation} 
\label{sec: Result BHE}
\begin{wrapfigure}[31]{r}[0pt]{0.5\columnwidth}
    \centering
    \centerline{\includegraphics[width=0.5\columnwidth]{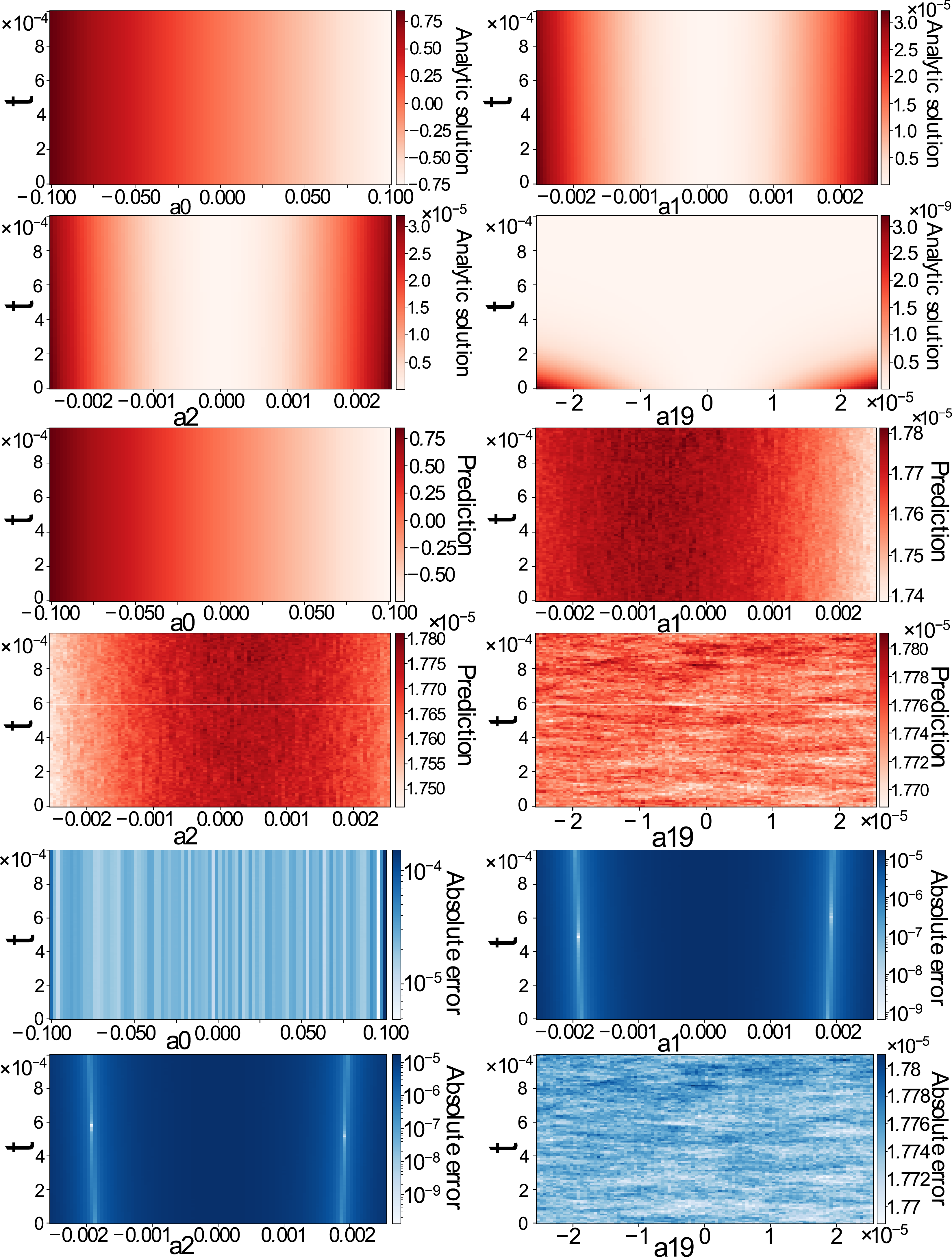}}
    \caption{\textbf{Analytic solution (top four panels), prediction (second four panels), and absolute error (last four panels) of BHE with degree 20 under delta initial condition.} 
    The horizontal axes are $a_k$ ($k=0,1,2,$ or $19$). The other coefficients are kept 0.
    Our model successfully learns the BHE.
    }
    \label{fig: fig_edit_heat_AnalSolPredRelWithTauVsAk_BHE_init0_deg10_sig1e1_stat_model0}
\end{wrapfigure}
\textbf{Tab.~\ref{tab: Mean relative (top) and absolute (bottom) error (BHE).}} shows the main result.
Again, our model successfully achieves $\sim 10^{-3}$, the typical order of relative error of PINNs.
See Fig.~\ref{fig: fitting}, App.~\ref{app: Cross-degree Evaluation}, and the footnote in Sec.~\ref{sec: Result FTE} for the assessment of the theoretical convergence of the cylindrical approximation.

\textbf{Fig.~\ref{fig: fig_edit_plot_FirstMomentRel_wERRBARinit050_sig1e1_WandWObcW0_stat_model}} shows the relative error of first-order functional derivatives at $\theta=0$.
Note again that some of the collocation points used for this figure are not included in the training dataset, highlighting the model's ability to extrapolate beyond its training dataset ($a_k \sim 0$).
Additionally, the error is reduced by a factor of $10^{-1}$ by incorporating a loss term corresponding to the identity $W([0], t) = 0$ (bottom four panels).

\textbf{Fig.~\ref{fig: fig_edit_heat_AnalSolPredRelWithTauVsAk_BHE_init0_deg10_sig1e1_stat_model0}} visualizes the analytic solution, prediction, and absolute error of a model trained on the BHE with degree $20$ under the delta initial condition.
The absolute error w.r.t.~$a_0$ is $10^{-4}$ times smaller than the scale of the solution; i.e., the model learns $\theta$ well in the direction of $a_0$, which dominates the analytic solution.
Conversely, the absolute error w.r.t.~$a_{19}$ is on par with the scale of the solution. 
This result is anticipated because the dependence of the solution on $a_{19}$ is much smaller
than $a_0$.
This relationship is evident from Eq.~\eqref{eq: BHE sol main text}, which indicates that the higher degree terms decay exponentially in terms of $k$, $l$, and $t$, and the solution is dominated by $a_k$ with $k \lesssim 1$.
Therefore, optimizing the model in the direction of $a_{19}$ has only a negligible effect on minimizing the loss function.

Finally, many additional experimental results are provided in App.~\ref{app: Additional Results I}, including a comparison with the CP-ALS algorithm.

\section{Limitations} 
\label{sec: Conclusion and Limitations}
One limitation of our work is that the spacetime dimension is limited to $1+1$ ($t$ and $x$) in our experiments.
However, generalization to $1+d$ dimensions is feasible, albeit with additional computational costs.
For $d>1$ dimensional spaces, we have several options for expansion bases \cite{rodgers2024tensor}.

Another limitation is that the orders of functional derivatives in FDEs in our experiments are limited to $r=1$. However, extending to $r\geq2$ is straightforward. 
For instance, the cylindrical approximation of the second-order functional derivative is expressed as $\delta^2 F([\theta], t)/ \delta \theta(x) \delta \theta(y) \approx \Sigma_{k,l=0}^{m-1} \partial^2 f(\boldsymbol{a}, t) / \partial a_k \partial a_l \phi_k(x) \phi_l (y)$, which can be computed via backpropagation twice.

Furthermore, this paper focuses exclusively on the abstract evolution equation. 
While this includes important FDEs (see Sec.~\ref{sec: Theory: Cylindrical Approximation}), it does not cover certain equations, such as the Schwinger-Dyson equation or the Wetterich equation. 
Nonetheless, the mathematical foundations regarding the existence and uniqueness of these FDEs remain unestablished, which is beyond the scope of our paper. 
Once these foundations are defined, applying our model to these equations would be straightforward.
More discussions on limitations, including technical ones, are provided in App.~\ref{app: Limitations}.

\begin{ack}

Taiki Miyagawa, an employee at NEC Corporation, contributed to this paper independently of his company role. He undertook this research as an independent researcher under the permission by the company.
Takeru Yokota was supported by the RIKEN Special Doctoral Researchers Program.

\end{ack}

\clearpage
{
\small
\bibliography{__main}
\bibliographystyle{abbrv}
}

\clearpage
\appendix
\section*{Appendix}
\tableofcontents

\clearpage
\section{Supplementary Related Work} \label{app: Supplementary Related Work}
\paragraph{Note on \cite{venturi2018numerical} \& \cite{venturi2021spectral}.}
Most of the numerical results presented in \cite{venturi2018numerical} and \cite{venturi2021spectral} are derived from simulations based on analytically specified functions and functionals, without involving the numerical integration of FDEs.
For instance, Fig.~2 in \cite{venturi2021spectral} does not depict the result of numerically solving an FDE. Instead, it illustrates the approximation error of the cylindrical approximation of an analytically given functional (refer to Eq.~(127) in \cite{venturi2021spectral}). The performance of numerical integration of FDEs using CPU-based algorithms is given in Fig.~38 in \cite{venturi2018numerical}, which shows the application of the CP-ALS and hierarchical Tucker (HT) methods for cases with $m \leq 6$. Note that the HT is reported to be slower than the CP-ALS \cite{venturi2018numerical}.

\paragraph{Applications of functionals.}
Functionals play a fundamental role in stochastic systems \cite{wiener1966nonlinear, fox1986functional, klyatskin2005dynamics112}, Fokker-Planck equations \cite{fox1986functional}, the statistical theory of turbulence \cite{hopf1952statistical, alankus1988generating, monin2013statistical145}, the theory of superfluidity \cite{bogoliubov1947theory, seiringer2011excitation}, quantum field theories \cite{venturi2013conjugate225, venturi2018numerical}, mean-field games \cite{carmona2018probabilistic}, many-body Schr{\"o}dinger equations \cite{hohenberg1964inhomogeneous}, mean-field optimal control \cite{weinan2018mean, ruthotto2020machine}, and unnormalized optimal transport \cite{gangbo2019unnormalized}.  

\paragraph{Examples of FDEs.}
Examples of FDEs include the Schwinger-Dyson equation in quantum field theory \cite{itzykson2012quantum}, the Hopf characteristic functional equation in fluid mechanics and random processes \cite{beran1965statistical, klyatskin2005dynamics112, monin2013statistical145}, probability density functional equations, and effective Fokker-Planck systems \cite{venturi2018numerical}.

\paragraph{Operator learning.}
Functionals can be seen as operators that map a function to a scalar; thus, operator learning \cite{anandkumar2019neural_Neural_Operator_org, li2021fourier_FNO_org, lu2019deeponet_DeepONet_org_arXiv_ver, lu2021learning_DeepONet_org_Nature_ver} appears to be a promising approach to learning functionals. 
However, this method requires simulation or observation data unless PINNs are used simultaneously \cite{li2021physics_PINO_org, wang2021learning_PI_DeepONet_org}. In other words, operator learning methods solve inverse problems, while we focus on forward problems, where only the equation to be solved is given.

\paragraph{Other approximation methods for FDEs.}
A common class of solvers for FDEs is based on truncating power series expansions at a finite order \cite{monin2013statistical145}. 
This includes the functional Taylor expansion, which expands a functional in terms of its argument functions. 
However, its applicability is limited because solutions obtained from the functional Taylor expansion can only be used for input functions close to the expansion center. 

Another type of expansion used in the theory of functional renormalization group is the derivative expansion \cite{dupuis2021nonterturbative_FRGreview}. It is an expansion in terms of derivatives of the input functions. However, solutions are only feasible for inputs close to constants. For example, in the three-dimensional $\mathcal{O}(N)$ statistical model, derivative expansions up to the sixth order have been executed \cite{dupuis2021nonterturbative_FRGreview}, but they are limited to calculations in uniform states and cannot handle non-uniform states.

Yet another expansion method uses the Reynolds number to distinguish between laminar and turbulent flow and has been applied to the Hopf equation \cite{monin2013statistical145}. However, increasing the truncation order poses a significant challenge. Specifically, calculating each expansion coefficient requires spatial integrals, leading to an exponential increase in computational costs.

Influence functions can be used for approximating Gateaux derivatives.
In \cite{jordan2022empirical}, the proposed approach is based on a finite-difference approximation of Gateaux derivatives, which requires a computational mesh for input function space when solving FDEs. Such an approach is infeasible because the number of mesh points grows exponentially with respect to the size of the input-function space.

\paragraph{High-dimensional PDEs.}
In our experiments, with the help of PINNs, we numerically solved 1000-dimensional PDEs, which are impossible to handle with discretization-based methods, such as the finite element method.
Numerical computation of high-dimensional PDEs is known to suffer from the curse of dimensionality, making PDEs with dimensions $d \geq 40$ particularly challenging \cite{weinan2021algorithms}. 

However, rapid development in this field, especially in PINNs, has enabled solving much higher-dimensional PDEs. 
For example, a 100,000-dimensional PDE is now tractable \cite{hu2023tackling}, which can be combined with our model.
Nevertheless, $d \sim \mcO(100)$ is typically sufficient in practice as long as input functions are regular. See also Fig.~\ref{fig: fitting functions}.

\paragraph{PINNs.}
PINNs are a type of mesh-free universal approximators of PDE solutions \cite{raissi2019physics_PINN_org}.
There are several machine learning-based mesh-free solvers, e.g., \cite{chen2021solving}.

\paragraph{Automatic functional differentiation of higher-order functions.}
Automatic differentiation of higher-order functions has a long history in theoretical computer science (see \cite{neidinger2010introduction, baydin2018automatic} and the references therein)."
Most studies focus on the mathematical aspects of programming languages, particularly how to implement reliable automatic functional differentiation, which is beyond our scope.
We cite two recent papers that explicitly mention functional derivatives. 
Di et al. \cite{di2023language} develop a language to compute automatic functional derivatives; however, the implementation is not available.
Lin \cite{lin2024automatic} provides a JAX implementation of functional derivatives w.r.t. parameterized input functions. However, it supports only local, semilocal, and several nonlocal operators, limiting the functional space.  In contrast, our model extends its approximability as $m$ increases.

\paragraph{Density functional theory (DFT).}
An alternative neural network-based approach to functional analysis utilizes finite element methods for spacetime grid approximation, commonly employed in first-principles computations of density functional theory (DFT) \cite{fiedler2022deep}. Examples include a neural network, $\hat{F}(\{y_j := f(\mathbf{r}_j)\}_j)$, 
approximating a target functional $F([f])$ by evaluating $f$ at specific grid points $\{\mathbf{r}_j\}_j$. 
Functional derivatives at each grid point can be computed using automatic differentiation: $\{ \frac{\delta F([f])}{\delta f(\boldsymbol{r})} \}_{\boldsymbol{r}} \fallingdotseq \{ \frac{\partial \hat{F}(\{ y_j \}_j)}{\partial y_i}\}_i$. 
However, the central focus of the machine learning studies for DFT does not include solving PDEs, let alone FDEs.

\paragraph{Neural functional networks.}
Implicit Neural Representations (INRs) is another strategy to handle functions as the inputs to neural networks \cite{zhou2023neural, spatialfuncta23, luigi2023deep, franzese2023continuoustime}. However, this method requires a large number of weight parameters, resulting in substantial computational demands.

\clearpage
\section{Additional Introduction to FDEs} \label{app: Extra Introduction to FDEs}

\subsection{Background}
Functional differential equations (FDEs) are prevalent across various scientific fields, including the Hopf equation in statistical turbulence theory \cite{hopf1952statistical}, the Schwinger-Dyson equation in quantum field theory \cite{peskin1995introduction_to_QFT}, the functional renormalization group equation \cite{wegner1973renormalizaton_FRG, wilson1974renormalization_FRG, polchinski1984renormalization_FRG, wetterich1993exact_FRG}, the Fokker-Planck equation in statistical mechanics \cite{fox1986functional}, and equations for the energy density functional in DFT \cite{polonyi2002Effective_FRGDFT, schwenk2004Towards_FRGDFT}. The strength of FDEs lies in their comprehensive nature, enabling the derivation of various statistical properties of physical systems. For instance, the Hopf equation yields the characteristic functional, encompassing all information about simultaneous correlations of velocities at different positions, a crucial quantity in turbulence theory. Therefore, a highly accurate, efficient, and universal method for solving FDEs significantly impacts a broad range of scientific research but is yet to be explored.

A common class of solvers for FDEs is based on truncating power series expansions at a finite order \cite{monin2013statistical145}. 
This includes the functional Taylor expansion, which expands a functional in terms of its argument functions. 
However, its applicability is limited because solutions obtained from the functional Taylor expansion can only be used for input functions close to the expansion center. 

Another type of expansion used in the functional renormalization group theory is the derivative expansion \cite{dupuis2021nonterturbative_FRGreview}. It is an expansion in terms of derivatives of the input functions. However, solutions are only feasible for inputs close to constants. For example, in the three-dimensional $\mathcal{O}(N)$ statistical model, derivative expansions up to the sixth order have been executed \cite{dupuis2021nonterturbative_FRGreview}, but they are limited to calculations in uniform states and cannot treat non-uniform states.

Yet another expansion method uses the Reynolds number, distinguishing between laminar and turbulent flow, and has been applied to the Hopf equation \cite{monin2013statistical145}. However, increasing the truncation order poses a significant challenge. Specifically, calculating each expansion coefficient requires spatial integrals, leading to an exponential increase in computational costs.

An alternative to solving FDEs is the cylindrical approximation \cite{venturi2018numerical}. 
In this method, the input function is expanded using a set of basis functions truncated at a finite degree. 
The cylindrical approximation transforms an FDE into a high-dimensional PDE, and discretization-based methods are often used together.
To address its high computational cost, tensor decomposition methods are also used.
Canonical Polyadic (CP) tensor expansion with the Alternating Least Squares (ALS) method \cite{venturi2018numerical} is the state-of-the-art algorithm in this class. 
The reported results to date are limited to cases with six or fewer basis functions. 
The computational cost related to $m$ is at least $\mcO(m^6)$, presenting a challenge in increasing the number of bases $m$.
In contrast, our model scales $\sim \mcO(m) + \mcO(m^r)$, where $r$ is the order of the functional derivative included in the target FDE. 
$r$ is typically 1 or 2, and thus the dependence on $m$ is $\sim \mcO(m)$ or $\mcO(m^2)$.

An alternative to solving FDEs is the cylindrical approximation \cite{venturi2018numerical}. In this method, the input function is expanded using a set of basis functions truncated at a finite degree. The cylindrical approximation transforms an FDE into a high-dimensional PDE, often solved using discretization-based methods. To address the high computational cost, tensor decomposition methods are also used. The CP-ALS method \cite{venturi2018numerical} is the state-of-the-art algorithm in this class. The reported results to date are limited to cases with six or fewer basis functions. The computational cost related to $m$ is at least $\mcO(m^6)$, presenting a challenge in increasing the number of basis functions $m$. In contrast, our model scales as $\mcO(m^r)$, where $r$ is the order of the functional derivative in the target FDE. Typically, $r$ is 1 or 2.

Below, we provide examples from the fields of turbulence, quantum field theory, and density functional theory.

\subsection{Turbulence} \label{app: Turbulence}
Turbulence appears everywhere, from natural systems (e.g., river flows and wind currents) to artificial systems (e.g., water flow in pipes and airflow over airplane surfaces). Understanding its properties is important both in natural sciences and in engineering. However, turbulence is a very complex system involving many degrees of freedom, and the only way to theoretically represent the properties of turbulence is through statistical methods. The Hopf equation \cite{hopf1952statistical} is an FDE that comprehensively describes the properties of turbulence. For example, the Hopf equation for a fluid following the Navier-Stokes equations is written as follows:
\begin{align}
    \label{eq: Navier-Stokes Hopf equation}
    \frac{\partial \Phi([\boldsymbol{\theta}],t)}{\partial t}
    =
    \sum_{k=1}^3
    \int_{V} d\boldsymbol{x}
    \theta_k(\boldsymbol{x})
    \left(
    \frac{i}{2}
    \sum_{j=1}^3
    \frac{\partial}{\partial x_j}\frac{\delta^2 \Phi([\boldsymbol{\theta}],t)}{\delta \theta_j(\boldsymbol{x})\delta \theta_k(\boldsymbol{x})}
    +
    \nu 
    \nabla^2
    \frac{\delta \Phi([\boldsymbol{\theta}],t)}{\delta \theta_k(\boldsymbol{x})}
    \right).
\end{align}
Here, $\Phi([\boldsymbol{\theta}],t)$ is a characteristic functional of a test function $\boldsymbol{\theta}(\boldsymbol{x})=(\theta_1(\boldsymbol{x}),\theta_2(\boldsymbol{x}),\theta_3(\boldsymbol{x}))$, and $\nu$ is the kinematic viscosity. The characteristic functional contains all the information about the correlations of velocities at different locations at the same time. Indeed, the moments of velocity can be obtained from the derivatives of the characteristic functional as follows:
\begin{align}
    \label{eq: derivative of characteristic functional}
    \langle
    u_{i_1}(\boldsymbol{x}_1,t)
    \cdots
    u_{i_n}(\boldsymbol{x}_n,t)
    \rangle_{\mathrm{m}}
    =
    (-i)^n
    \left.
    \frac{\delta^n \Phi([\boldsymbol{\theta}],t)}
    {
    \delta \theta_{i_1}(\boldsymbol{x}_1)
    \cdots
    \delta \theta_{i_n}(\boldsymbol{x}_n)
    }
    \right|_{\boldsymbol{\theta}=\boldsymbol{0}},
\end{align}
where $u_{i}(\boldsymbol{x},t)$ is the $i$the component of the velocity at $(\boldsymbol{x},t)$. In particular, the Fourier transform of the second-order cumulant is known as the energy spectrum, which represents the contribution of eddies at various momentum scales to turbulence. The search for a comprehensive method to describe the behavior of the energy spectrum across a wide range of momentum scales continues, and for this reason, developing methods to solve the Hopf equation holds significant importance.

More specifically, let us consider the fluid mechanics of aircraft or pipeline flow. 
For such systems, the second-order functional derivative of the solution of the Hopf characteristic functional gives the two-point correlation function of the velocity field at arbitrary two positions $x$ and $y$ and an arbitrary time $t$.
The Fourier transform of it w.r.t. $x$ and $y$ is called the energy spectrum, which indicates which scales of motion are most energetic in the fluid flow.
The energy spectrum is used to model and predict the behavior of turbulence, e.g., constructing safe and efficient shapes of airplanes or pipelines \cite{nichols2010turbulence}.

\subsection{Quantum Field Theory}
Quantum mechanics tells us that physical quantities in the microscopic world do not always have deterministic values but fluctuate. In quantum field theory (QFT), which is a branch of quantum mechanics and forms the basis of modern particle physics, particles are described as fluctuating \textit{fields} spreading throughout spacetime. QFT allows us to describe the properties of elementary particles in a statistical way, i.e., the correlation functions of fields at different points in spacetime.
Therefore, developing methods to calculate correlation functions is very important for understanding the properties of elementary particles. 

Several FDEs provide information on the correlation function of fields. A well-known example is the Schwinger-Dyson equation \cite{peskin1995introduction_to_QFT}. In the statistical model known as the $\phi^4$ model, which is described by the following action
\begin{align}
    S([\phi])
    =
    \int d\boldsymbol{x}
    \left(
    \frac{1}{2}\phi(\boldsymbol{x})\left(-\nabla^2+m^2 \right)\phi(\boldsymbol{x})
    +
    \frac{\lambda}{4!}\phi(\boldsymbol{x})^4
    \right),
\end{align}
the Schwinger-Dyson equation is given as follows:
\begin{align}
    \label{eq:Schwinger-Dyson}
    -
    \nabla^2
    \frac{\delta Z([J])}{\delta J(\boldsymbol{x})}
    +
    m^2
    \frac{\delta Z([J])}{\delta J(\boldsymbol{x})}
    -
    \frac{\lambda}{3!}
    \frac{\delta^3 Z([J])}{\delta J(\boldsymbol{x})^3}
    -
    iJ(\boldsymbol{x})Z([J])=0.
\end{align}
Here, $Z([J])$ is a quantity known as the partition function, and by functionally differentiating this quantity w.r.t. $J(\boldsymbol{x})$, all correlation functions for the field $\phi(\boldsymbol{x})$ can be obtained. Another method to describe the correlation functions is the functional renormalization group \cite{wegner1973renormalizaton_FRG,wilson1974renormalization_FRG,polchinski1984renormalization_FRG,wetterich1993exact_FRG}. The renormalization group is a method of analyzing physical systems based on the operation of reducing spacetime resolution. 
Under such operations, we can define an FDE for the effective action $\Gamma([\phi])$, which is calculated by the Legendre transformation of $\ln Z([J])$. $\Gamma([\phi])$ contains all the information of the correlation functions, similarly to $Z([J])$. 
$\Gamma([\phi])$ satisfies the following FDE \cite{wetterich1993exact_FRG}:
\begin{align}
    \label{eq:Wetterich-eq}
    &\partial_k
    \Gamma_k([\phi])
    =
    \int d\boldsymbol{x}
    \int d\boldsymbol{x}'
    \partial R_k(\boldsymbol{x}-\boldsymbol{x}')
    G_k([\phi], \boldsymbol{x}',\boldsymbol{x}).
\end{align}
$k$ represents the momentum scale that specifies the resolution at which spacetime is observed, $R_k(\boldsymbol{x})$ is a function manually provided to realize the operation of the renormalization group,
and $G_k([\phi], \boldsymbol{x}',\boldsymbol{x})$ is the regulated propagator defined as
\begin{align}
    \int d\boldsymbol{x}'
    \left(
    \frac{\delta^2 \Gamma_k([\phi])}{\delta \phi(\boldsymbol{x}) \delta \phi(\boldsymbol{x}')}
    +
    R_k(\boldsymbol{x}-\boldsymbol{x}')
    \right)
    G_k([\phi], \boldsymbol{x}',\boldsymbol{x}'')
    =
    \delta(\boldsymbol{x}-\boldsymbol{x}'').
\end{align}

\subsection{Density Functional Theory}
Density functional theory (DFT) is widely used in material science, quantum chemistry, and nuclear physics.
The properties of materials and molecules are determined by the state of electrons, which follows the Schr\"odinger equation. Solving the Schr\"odinger equation becomes challenging especially when the material contains many electrons. Hohenberg and Kohn demonstrated that it is possible to determine the state of electrons by solving a variational equation for the density \cite{hohenberg1964inhomogeneous}, instead of solving the Schr\"odinger equation directly. This is a density-based variational equation and has been shown to be easier to solve than the Schr\"odinger equation.

However, in DFT, exact calculations are usually not possible. The reason is that the Energy Density Functional (EDF), which provides the variational equation for density, cannot be precisely determined. Whether a good approximation for the EDF can be provided or not significantly influences the success of DFT calculations. There has been a lot of research on finding EDFs, including empirical approaches, for a long time. One recent approach is based on FDEs. 
Specifically, several FDEs are known to describe the evolution of the EDF when the two-body interaction $U(\boldsymbol{x} - \boldsymbol{x}')$, e.g., the Coulomb interaction between electrons, gradually increases \cite{polonyi2002Effective_FRGDFT,schwenk2004Towards_FRGDFT}. When the interaction is replaced with $\lambda U(\boldsymbol{x} - \boldsymbol{x}')$, and when $\lambda$ gradually increases, the FDE becomes:
\begin{align} 
    &\partial_\lambda \Gamma_\lambda([n])
    =
    \frac{1}{2}
    \int d\tau \int d\boldsymbol{x} \int d\boldsymbol{x}'
    U(\boldsymbol{x}-\boldsymbol{x}')
    \left[
    n(\boldsymbol{x},\tau)
    n(\boldsymbol{x}',\tau)
    +
    G_\lambda([n],\boldsymbol{x},\tau,\boldsymbol{x}',\tau')
    -
    n(\boldsymbol{x},\tau)
    \delta(\boldsymbol{x}-\boldsymbol{x}')
    \right], \label{eq: tmp000}
    \\
    &
    \int d\boldsymbol{x}'
    \int d\tau'
    \frac{\delta^2 \Gamma_\lambda([n])}
    {\delta n(\boldsymbol{x},\tau) \delta n(\boldsymbol{x}',\tau')}
    G_\lambda([n], \boldsymbol{x}',\tau',\boldsymbol{x}'',\tau'')
    =
    \delta(\boldsymbol{x}-\boldsymbol{x}'')
    \delta(\tau-\tau''). \label{eq: tmp001}
\end{align}
Here, $n(\boldsymbol{x}, \tau)$ is the density of electrons, and $\Gamma_\lambda([n])$ represents an effective action, which is an extension of the EDF \cite{fukuda1994density,valiev1997generalized}. In addition to the coordinates $\boldsymbol{x}$ and $\boldsymbol{x}'$, a virtual dimension, known as the imaginary time $\tau$, is introduced. This FDE is expected to provide a new method for constructing the EDF \cite{polonyi2002Effective_FRGDFT,schwenk2004Towards_FRGDFT}.
For example, the EDF of the three-dimensional electron system is derived from Eqs.~(\ref{eq: tmp000}--\ref{eq: tmp001}) based on the functional Taylor expansion \cite{yokota2021abinitio}.

\clearpage
\section{Theoretical Background of Cylindrical Approximation and Convergence} 
\label{app: Theorems Related to Cylindrical Approximation and Convergence}
The mathematical background of our theoretical contribution is provided to make our paper self-contained. 
This section is based on the recent development of the theory of cylindrical approximation \cite{venturi2021spectral} and the classical spectrum theory \cite{hesthaven2007spectral, shen2011spectral}. 
The differentiability of functionals is discussed in App.~\ref{app: Differentiation of Functionals}, showing that the solutions of the FDEs in our experiment are differentiable. 
The equivalence and difference between the functional, Fr\'{e}chet, and G\^{a}teaux derivative are summarized in App.~\ref{app: Differentiation of Functionals}. They are equivalent in practical settings, and we do not distinguish them in this paper.

\subsection{Continuity and Compactness of Functionals}
Let $X$ be a Banach space unless otherwise stated. In this paper, a functional on $X$ is defined as a map $F$ from $D(F) \subset X$ to $\mbr$, where $D(F)$ is the domain of $F$. Note that $\mbr$ cannot be replaced with $\mbq$ in all the statements below. 
We first define pointwise continuity, uniform continuity, compactness, and complete continuity of functionals.

\begin{definition}[Pointwise continuity of $F$]
    A functional $F: X \supset D(F) \rightarrow \mbr$ is continuous at $\theta \in D(F) \subset X$ if for any Cauchy sequence $\{ \theta_n \}_n \subset D(F)$,
        \begin{equation}
        \lim_{n \rightarrow \infty} || \th_n - \th ||_X = 0 \Rightarrow \lim_{n \rightarrow \infty} | F([\th_n]) - F([\th]) | = 0,    
        \end{equation}
    where $||\cdot||_X$ is the norm induced by $X$.
\end{definition}

\begin{definition}[Uniform continuity of $F$]
    A functional $F: X \supset D(F) \rightarrow \mbr$ is uniformly continuous on $D(F)$ if 
    \begin{align}
        & \forall \ep >0, \,\, \exists \delta > 0 \,\, s.t. \nn
        & \forall \th_1, \th_2 \in D(F) \,\, \text{satisfying} || \th_1 - \th_2 ||_X \leq \delta, \nn 
        & | F([\th_1]) - F ([\th_2]) | < \ep \text{ holds},       
    \end{align}
    where $||\cdot||_X$ is the norm induced by $X$.
\end{definition}
We simply say ``continuous'' in the following.

\begin{definition}[Compactness of $F$]
    A functional $F: X \supset D(F) \rightarrow \mbr$ is compact on $D(F)$ if $F$ maps any bounded subset of $D(F)$ into a pre-compact subset of $\mbr$.
\end{definition}
Recall that $A \subset \mbr$ is a pre-compact subset if the closure of $A$, denoted by $\bar{A}$, is compact.

\begin{definition}[Complete continuity of $F$]
    A functional $F: X \supset D(F) \rightarrow \mbr$ is completely continuous on $D(F)$ if $F$ is continuous and compact on $D(F)$.    
\end{definition}
Based on these concepts, functional derivatives are defined.

\subsection{Boundedness, Closedness, Compactness, and Pre-compactness of Metric Space of Functions}
Next, we define boundedness, closedness, compactness, and pre-compactness of a metric space of functions.

\begin{definition}[Boundedness of metric space of functions]
    Let $X$ be a metric space of functions. $K \subset X$ is bounded if $\exists M \in \mbr$ s.t. $\forall \th \in K$,  $|| \th ||_X < M$.
\end{definition}

\begin{definition}[Closedness of metric space of functions]
    Let $X$ be a metric space of functions. $K \subset X$ is closed if any convergent sequence in $K$ has a limit in $K$.
\end{definition}

\begin{definition}[Compactness of metric space of functions]
    Let $X$ be a metric space of functions. $K \subset X$ is compact if any open cover of $K$ has a finite subcover. Equivalently, $K$ is compact if and only if any sequence in $K$ is a bounded subsequence whose limit is in $K$.
\end{definition}

\begin{definition}[Pre-compactness of metric space of functions]
    Let $X$ be a metric space of functions. $K \subset X$ is pre-compact if its closure $\bar{K}$ is compact. Equivalently, $K$ is pre-compact if and only if any sequence in $K$ has a convergent subsequence whose limit is in $X$.   
\end{definition}
A critical characteristic of pre-compactness is given by the following theorem (a necessary and sufficient condition for the pre-compactness of metric spaces).
\begin{theorem}[Necessary and sufficient condition of pre-compactness of metric space \cite{melrose2005introduction}]
    \label{thm: Necessary and sufficient condition of pre-compactness of metric space}
    A subset $E$ of a real separable Hilbert space $H$ is pre-compact if and only if $E$ is (i) bounded, (ii) closed, and (iii) for any orthonormal basis $\{ \phi_0, \phi_1,\ldots \}$ of $H$ and for any $\ep > 0$, there exists $m \in \mbn$ s.t.
    \begin{equation}
        \label{eq: equi-samll tail condition}
        \forall \th \in E, \sum_{k=m+1}^\infty | (\th, \phi_k)_H |^2 \leq \ep \, ,
    \end{equation}
    where $(\cdot, \cdot)_H$ is the inner product defined on $H$.
\end{theorem}
Eq.~\eqref{eq: equi-samll tail condition} is known as the equi-small tail condition. Thm.~\ref{thm: Necessary and sufficient condition of pre-compactness of metric space} characterizes the domain of functions $D(F)$ to which the cylindrical approximation can be applied.
Note that boundedness is not a necessary nor sufficient condition for the equi-small tail condition.

\subsection{Differentiation of Functionals} \label{app: Differentiation of Functionals}
Next, we define the G\^{a}teaux differential, Fr{\'e}chet differential, and functional derivative. Higher-order functional derivatives can be defined in a similar way \cite{venturi2018numerical}.

\begin{definition}[G\^{a}teaux differential]
    A functional $F: X \overset{\mathrm{open}}{\supset} D(F) \rightarrow \mbr$ is G\^{a}teaux differentiable at $\th \in D(F)$ if 
    \begin{equation}
        dF^{\rm G}_\eta([\th]) := \lim_{\ep \rightarrow 0} \frac{ F([\th + \ep \eta]) - F([\th]) }{\ep} \label{eq: def of Gateaux differential}
    \end{equation}
    exists and is finite for all $\eta \in D(F)$, where $dF^{\rm G}_\eta([\th])$ is called the G\^{a}teaux differential of $F$ in the direction $\eta$.
\end{definition}



There are some patterns of differentiability conditions. One of them is:
\begin{theorem}[G\^{a}teaux differentiability of Lipschitz functionals \cite{mankiewicz1973differentiability, Aronszajn1976differentiability, lindenstrauss2003Frechet}]
    Let Banach space $X$ be separable. 
    Then, any Lipschitz functional $F: X \overset{\mathrm{open}}{\supset} D(F) \rightarrow \mbr$ is G\^{a}teaux differentiable outside a Gauss-null set.  
\end{theorem}
Note that a Gauss-null set is a Borel set $A \subset X$ s.t. $\forall$ non-degenerate Gaussian measure $\mu$ on $X$, $\mu(A)$ is equal to 0.
In this theorem, there is no guarantee for non-Lipschitz functionals, e.g., $F([\theta]) = \int dx \sqrt{ |\theta(x)| }$, where $\theta(0)=0$.




Under mild conditions, the G\^{a}teaux derivative is defined, based on the G\^{a}teaux differential.
\begin{theorem}[G\^{a}teaux derivative \cite{vauinberg1964variational}]
    If the following two conditions are satisfied, then the G\^{a}teaux differential $dF_\eta([\theta])$ of functional $F: X \overset{\mathrm{open}}{\supset} D(F) \rightarrow \mbr$ at $\theta \in D(F)$ in the direction $\eta \in D(F)$ can be represented as a linear operator acting on $\eta$, denoted by $F^\prime([\theta])$, s.t. 
        \begin{equation}
            dF^{\rm G}_\eta([\theta]) = F^\prime([\theta])\eta \, , \label{eq: def of Gateaux derivative}
        \end{equation}
    where $F^\prime([\theta]): D(F) \rightarrow \mbr$ is a linear operator, or a linear functional, depending on $\theta$ and is called the G\^{a}teaux derivative of $F$ at $\th$.
    \begin{enumerate}
        \item $dF^{\rm G}_\eta([\theta])$ exists in some neighborhood of $\th_0 \in D(F)$ and is continuous w.r.t. $\theta$ at $\th_0$.
        \item $dF^{\rm G}_\eta([\theta_0])$ is continuous w.r.t. $\eta$ at $\eta = \eta_0$, where $||\eta_0||_X = 0$.
    \end{enumerate}
\end{theorem}

Next, we define the second type of differentials, Fr\'{e}chet differential.
\begin{definition}[Fr{\'e}chet differential]
    A functional $F: X \overset{\mathrm{open}}{\supset} D(F) \rightarrow \mbr$ is Fr{\'e}chet differentiable at $\th \in D(F)$ if $dF^{\rm F}_\eta([\th]) \in \mbr$ s.t.
    \begin{equation}
        \lim_{\ep \rightarrow 0} \frac{| F([\th + \ep \eta]) - F([\th]) - dF^{\rm F}_\eta([\th])|}{\ep} = 0
    \end{equation}
    exists and is finite for all $\eta \in D(F)$, where $dF^{\rm F}_\eta([\th])$ is called the Fr{\'e}chet differential of $F$ in the direction $\eta$.
\end{definition}

There are also some patterns of differentiability conditions. One of them is:
\begin{theorem}[Fr\'{e}chet differentiability of Lipschitz functionals \cite{preiss1990differentiability}]
    Let $K$ be a compact subset of a Hilbert space $H$. Then, any locally-Lipschitz functional $F: K \overset{\mathrm{open}}{\supset} D(F) \rightarrow \mbr$ is Fr\'{e}chet differentiable on a dense subset of $K$.
\end{theorem}
An example functional that is G\^{a}teaux differentiable but Fr\'{e}chet non-differentiable is $F([\theta]) = \frac{a_k^3}{a_k^2 + a_l^2}$, where $\theta (x) = \sum_{i=0}^\infty a_i \phi_i(x)$ and $k, l \in \mbn$.

\paragraph{Relationship between G\^{a}teaux and Fr\'{e}chet differential.}
If $F$ has a continuous G\^{a}teaux derivative on $D(F)$, then $F$ is Fr\'{e}chet differentiable on $D(F)$, and these two derivatives are equal \cite{vauinberg1964variational}.
The G\^{a}teaux derivative of the aforementioned example, $F([\theta]) = \frac{a_k^3}{a_k^2 + a_l^2}$, is not continuous, and thus, the Fr\'{e}chet differentiability is not guaranteed.
In the following, we consider functionals $F$ that are continuously G\^{a}teaux differentiable in $D(F)$; therefore, we do not care about differentiability and do not distinguish G\^{a}teaux derivatives from Fr\'{e}chet derivatives.
Hereafter, we write $dF_\eta([\th]) := dF^{\rm G}_\eta([\th]) = dF^{\rm F}_\eta([\th]) = F^\prime([\th]) \eta$ 

Next, we define the third type of derivatives.
\begin{definition}[Functional derivative] \label{def: Functional derivative}
    The functional derivative of a functional $F: X \overset{\mathrm{open}}{\supset} D(F) \rightarrow \mbr$ w.r.t. $\th(x)$ is defined as 
    \begin{equation}
        \frac{\delta F ([\th])}{\delta \th (x)} := \lim_{\ep \rightarrow 0} \frac{ F([\th(y) + \ep \de(x - y)]) - F([\th(y)]) }{\ep},
    \end{equation}
    if it exists and is finite, where $\de(x)$ is the Dirac delta.
\end{definition}
Strictly speaking, this definition may be informal because $\th(y)$ is a function, while $\de(x - y)$ is a distribution.
The representation theorem below (Thm.~\ref{thm: Representation theorem of Frechet derivative}) is sometimes regarded as the definition of the functional derivative.

Lem.~\ref{lem: Compactness of Frechet derivative} and the Riesz representation theorem prove the following relation between the Fr\'{e}chet derivative and the functional derivative.

\begin{lemma}[Compactness of Fr\'{e}chet derivative] \label{lem: Compactness of Frechet derivative}
    Let $K$ be a compact subset of a real separable Hilbert space $H$. Let $F$ be a continuous functional on $H$. If the Fr\'{e}chet derivative $F^\prime([\th])$ exists at $\th \in K$, then it is a compact linear operator.
\end{lemma}

\begin{theorem}[Representation theorem of Fr\'{e}chet derivative] \label{thm: Representation theorem of Frechet derivative}
    Let $K$ be a compact subset of a real separable Hilbert space $H$. Let $F$ be a continuous functional on $H$. If the Fr\'{e}chet derivative $F^\prime([\th])$ exists at $\th \in K$, then the following unique integral representation of the Fr\'{e}chet derivative holds:
    \begin{equation}
        \forall \eta \in H, \,\, F^\prime([\th]) \eta = \left( \frac{\delta F([\th])}{\delta \theta}, \eta \right)_H \, ,
        \label{eq: representation theorem}
    \end{equation}
    where $\frac{\delta F([\th])}{\delta \theta(x)} \in H$. 
\end{theorem}

The representation theorem \ref{thm: Representation theorem of Frechet derivative} is the foundation of the cylindrical approximation of functional derivatives, which is shown below.

\subsection{Cylindrical Approximation}
\subsubsection{Functionals}
Let $H$ be a real separable Hilbert space. The cylindrical approximation is based on the fact that any $\th \in H$ can be represented uniquely in terms of an orthonormal basis $\{ \phi_0, \phi_2,\ldots \}$ as $\th(x) = \sum_{k=0}^{\infty} (\th, \phi_k)_H \phi_k(x)$. 




Thus, we can define 
\begin{equation}
    f (a_0, a_1, \ldots) := F([ \sum_{k=0}^{\infty} a_k \phi_k ]), \label{eq: basis function expansion of F}
\end{equation} 
where $a_k := (\th, \phi_k)_H$. 
Truncating $k \leq m-1$ ($m \in \mbn$) gives the cylindrical approximation of functionals:
\begin{definition}[Cylindrical approximation of functionals \cite{friedrichs1957integration, baez1997functional, gikhman2004theory, van2008stochastic}] \label{def: Cylindrical approximation of functionals}
    Let $P_m$ be the projection operator s.t. $P_m \th := \sum_{k=0}^{m-1} (\th, \phi_k)_H \phi_k$.
    Let $D_m$ be the finite-dimensional space induced by $P_m$; i.e., $D_m := \mathrm{span}\{ \phi_0, \phi_1, \ldots, \phi_{m-1}\}$.
    The cylindrical approximation of a functional $F([\th])$ on $H$ is the $m$-dimensional multivariable function 
    \begin{equation}
        f(a_0, a_1, \ldots, a_{m-1}) := F([P_m \th]) = F([\sum_{k=0}^{m-1} a_k \phi_k(x)]) \, ,
    \end{equation}
    where $a_k := (\th, \phi_k)_H$. In short, $F([\th]) \sim f(a_0, \ldots, a_{m-1})$.
\end{definition}
The cylindrical approximation of functionals is uniform:
\begin{theorem}[Uniform convergence of cylindrical approximation of functionals \cite{prenter1970weierstrass}] \label{thm: Uniform convergence of cylindrical approximation of functionals}
    Let $K$ be a compact subset of a real separable Hilbert space $H$. Let $F$ be a continuous functional on $H$.
    Then, 
    \begin{equation}
        \forall \ep > 0, \,\, \exists M \in \mbn  \,\, s.t. \,\,  \forall m \geq M, \,\,  \forall \th \in K, \,\,  | F([\th]) - F([P_m\th]) | < \ep
    \end{equation}
    holds; i.e., $F([P_m\th])$ converges uniformly to $ F([\th])$ on $K$.
\end{theorem}
The convergence rate is given by the mean value theorem of functionals:
\begin{theorem}[Convergence rate of cylindrical approximation of functionals \cite{venturi2021spectral}] \label{thm: Convergence rate of cylindrical approximation of functionals}
    Let $K$ be a compact and convex subset of a real separable Hilbert space $H$. Let $F$ be a continuously differentiable functional on $K$.
    Then, 
    \begin{equation}
        \forall \th \in K, \,\, | F([\th]) - F([P_m \th]) | \leq \sup_{\eta \in K} \left\| F^\prime([\eta]) \right\| \left\| \th - P_m \th \right\|_H \,.
    \end{equation}  
\end{theorem}
$||F^\prime([\eta])||$ is the operator norm of the linear operator $F^\prime([\eta])$; i.e., $\| F^\prime([\th]) \| := \sup_{\eta (\neq 0) \in H} \frac{| F^\prime([\th]) \eta |}{\| \eta \|_H}$.
The convergence rate of $\mcO(\| \th - P_m \th \|_H)$ depends on the basis and is provided in \cite{hesthaven2007spectral} (Chapters 4 \& 6) and \cite{shen2011spectral} (Sec. 3.5) for several bases.

\subsubsection{Functional Derivatives} \label{app: Functional Derivatives}
The cylindrical approximation of functional derivatives is motivated by the representation theorem \ref{thm: Representation theorem of Frechet derivative}, which states that (i) $\frac{\delta F([\th])}{\delta \theta(x)} \in H$ and (ii) $F^\prime([\th]) \eta = ( \frac{\delta F([\th])}{\delta \theta(x)}, \eta)_H$. Statement (i) means that $\frac{\delta F([\th])}{\delta \theta(x)}$ can be represented in terms of an orthogonal basis as 
\begin{equation}
    \frac{\delta F([\th])}{\delta \theta(x)} = \sum_{k=0}^\infty (\frac{\delta F([\th])}{\delta \theta(x)}, \phi_k)_H \phi_k(x) \, .
    \label{eq: basis function expansion of dF/dtheta}
\end{equation} 
Statement (ii) means that 
\begin{align}
    \frac{\partial f}{\partial a_k} 
    &= \lim_{\ep \rightarrow 0} \frac{f(a_0,\ldots,a_k + \ep,\ldots) - f(a_0,\ldots, a_k,\ldots)}{\ep} \\
    &= \lim_{\ep \rightarrow 0} \frac{1}{\ep}\left[ F([\sum_{l=0}^\infty a_l \phi_l + \ep \phi_k]) - F([\sum_{l=0}^\infty a_l \phi_l]) \right] \hspace{20pt} \text{($\because$ Eq.~\eqref{eq: basis function expansion of F})} \\
    &= F^\prime([\sum_{l=0}^\infty a_l \phi_l]) \phi_k  \hspace{20pt} \text{($\because$ Eqs.~(\ref{eq: def of Gateaux differential}--\ref{eq: def of Gateaux derivative}))} \\
    &= (\frac{\delta F([\th])}{\delta \phi_k}, \phi_k)_H \hspace{20pt} \text{($\because$ Eq.~\eqref{eq: representation theorem})},
    \label{eq: df/da_k}
\end{align}
where $\th (x) = \sum_{l=0}^\infty a_l \phi_l(x)$ with $a_l = (\th, \phi_l)_H$.
Eqs.~\eqref{eq: basis function expansion of dF/dtheta} and \eqref{eq: df/da_k} gives
\begin{equation}
    \frac{\delta F([\th])}{\delta \theta(x)} = \sum_{k=0}^\infty \frac{\partial f}{\partial a_k} \phi_k(x) \, .
\end{equation}
Therefore, truncating $k \leq m-1$ ($m \in \mbn$) gives the cylindrical approximation of functional derivatives:
\begin{definition}[Cylindrical approximation of functional derivatives \cite{venturi2018numerical, venturi2021spectral}]
The cylindrical approximation of a functional derivative $\frac{\delta F([\th])}{\delta \theta(x)} \in H$ is defined as
    \begin{equation} \label{eq: CA of FD with tail term}
        \frac{\delta F([P_m \th])}{\delta \theta(x)} 
        = \sum_{k=0}^{m-1}    \frac{\partial f}{\partial a_k} \phi_k (x) 
            + \sum_{k=m}^{\infty} \left( \frac{\delta F([P_m \th])}{\delta \th}, \phi_k \right)_H \phi_k(x) \, .\nonumber
    \end{equation}
    If $\frac{\delta F([P_m \th])}{\delta \theta(x)} \in D_m$, the second term on the left-hand side is equal to zero.
    $\sum_{k=m}^{\infty} \left( \frac{\delta F([P_m \th])}{\delta \th}, \phi_k \right)_H \phi_k(x)$ is called the \textit{tail term} in this paper.
\end{definition}
In short, $\frac{\delta F([\th])}{\delta \theta(x)} \sim \sum_{k=0}^{m-1} \frac{\partial f}{\partial a_k} \phi_k (x)$ for $\frac{\delta F([\th])}{\delta \theta(x)}$ that satisfies the equi-small tail condition.

Note that this version of the cylindrical approximation of functional derivatives is different from ours (Eq.~\ref{eq: cylindrical approx of functional derivatives with omitting second term}). Specifically, in \cite{venturi2018numerical, venturi2021spectral}, $P_m$ is not applied to $\delta F([\theta]) / \delta \theta(x)$, and the emerging tail term $\Sigma_{k=m}^{\infty} (\delta F([\theta]) / \delta \theta, \phi_k) \phi_k(x)$ is simply ignored without any rationale.





The cylindrical approximation of functional derivatives is uniform:
\begin{theorem}[Uniform convergence of cylindrical approximation of functional derivatives \cite{venturi2021spectral}]
\label{thm: Uniform convergence of cylindrical approximation of functional derivatives}
    Let $K$ be a compact subset of a real separable Hilbert space $H$. Let $F$ be a continuously differentiable functional on $K$.
    Then, 
    \begin{equation}
        \forall \ep > 0, \,\, \exists M \in \mbn \,\, s.t. \,\, \forall m \geq M, \,\, \forall \th \in K, \,\, 
        \left\| \frac{\delta F([\th])}{\delta \th} - \frac{\delta F([P_m \th])}{\delta \th} \right\|_H < \ep
    \end{equation}
    holds; i.e., $\frac{\delta F([P_m \th])}{\delta \th}$ converges uniformly to $\frac{\delta F([\th])}{\delta \th}$.
\end{theorem}

Note that this theorem proves the \textit{uniform} convergence, while ours (Thm.~\ref{thm: Pointwise convergence of cylindrical approximation (informal)}) proves the \textit{pointwise} convergence.
The difference comes from that the uniform convergence of the tail term, which is assumed in the above theorem, can be violated when our cylindrical approximation Eq.~\eqref{eq: cylindrical approx of functional derivatives with omitting second term} is used because the tail term is absent in Eq.~\eqref{eq: cylindrical approx of functional derivatives with omitting second term}.

Similarly, the cylindrical approximation of Fr\'{e}chet derivatives is uniform:
\begin{theorem}[Uniform convergence of cylindrical approximation of Fr\'{e}chet derivatives \cite{venturi2021spectral}]
    Let $K$ be a compact subset of a real separable Hilbert space $H$. Let $F$ be a continuously differentiable functional on $K$.
    Then, 
    \begin{equation}
        \forall \ep > 0, \,\, \exists M \in \mbn \,\, s.t. \,\, \forall m \geq M, \,\, \forall \th \in K, \,\, 
        \|  F^\prime([\th]) - F^\prime([P_m \th]) \| < \ep
    \end{equation}
    holds; i.e., $F^\prime([P_m \th])$ converges uniformly to $F^\prime([\th])$.
\end{theorem}
The convergence rate is given by:
\begin{theorem}[Convergence rate of cylindrical approximation of Fr\'{e}chet derivatives \cite{venturi2021spectral}] \label{thm: Convergence rate of cylindrical approximation of Frechet derivatives}
    Let $K$ be a compact and convex subset of a real separable Hilbert space $H$. Let $F$ be a differentiable functional on $K$ with continuous first- and second-order Fr\'{e}chet derivatives.
    Then, 
    \begin{equation}
        \forall \th \in K, \,\, \| F^\prime([\th]) - F^\prime([P_m \th]) \| 
        \leq \sup_{\eta \in K} \| F^{\prime\prime}([\eta]) \| \| \th - P_m \th \|_H \,,
    \end{equation}
    where $\| F^{\prime\prime}([\eta]) \| := 
    \underset{\zeta, \xi \in H, \zeta \neq 0, \xi \neq 0}{\sup}
    \frac{|F^{\prime\prime}([\eta])\zeta \xi|}{\| \zeta \|_H \| \xi \|_H}$.
\end{theorem}



In terms of a functional derivative, this is rewritten as
\begin{align} \label{eq: dFdtheta conv rate}
    \left\| 
    \frac{\delta F([\theta])}{\delta \theta}-\frac{\delta F([P_m\theta])}{\delta \theta}
    \right\|_H
    &\leq
    \sup_{\zeta \in K}\left(
    \left\| \frac{\delta^2 F([\zeta])}{\delta \theta \delta \theta} \right\|
    \right)
    \|\theta - P_m \theta \|_H \, ,
\end{align}
where $\left\| \frac{\delta^2 F([\zeta])}{\delta \theta \delta \theta} \right\| := 
\sup_{\xi,\xi'\in H, \xi \neq 0, \xi^\prime \neq 0}
    \frac{
    \left| 
    \frac{\delta^2 F([\zeta])}{\delta \theta \delta \theta} \xi \xi' 
    \right|
    }{
    \|\xi\|_H \|\xi'\|_H
    }
:= \sup_{\xi,\xi'\in H, \xi \neq 0, \xi^\prime \neq 0}
    \frac{
    \left|\left(\frac{\delta}{\delta \theta}
    \left(\frac{\delta F([\zeta])}{\delta \theta},\xi\right)_H,
    \xi' \right)_H\right|
    }{\|\xi\|_H \|\xi'\|_H}$; i.e., we regard $\frac{\delta^2 F([\zeta])}{\delta \theta \delta \theta}$ as an operator acting on $H \times H$.

\subsubsection{Pointwise Convergence of Functional Derivatives under Cylindrical Approximation}
We prove the pointwise convergence of the functional derivative under the cylindrical approximation.
We already noted that the cylindrical approximation of functional derivatives \eqref{eq: CA of FD with tail term} uniformly converges.
Below, we show that the convergence becomes pointwise if we omit the second term of the r.h.s. of Eq.~\eqref{eq: CA of FD with tail term}; i.e., we use    
\begin{equation}
    P_m\frac{\delta F([P_m \th])}{\delta \theta(x)} 
    = \sum_{k=0}^{m-1}    \frac{\partial f}{\partial a_k} \phi_k (x) \, \text{ (Eq.(\ref{eq: cylindrical approx of functional derivatives with omitting second term}))} \nonumber
\end{equation}
as the approximated functional derivative.

\begin{theorem}[Thm.~\ref{thm: Pointwise convergence of cylindrical approximation (informal)}. Pointwise convergence of cylindrical approximation] 
    Let $K$ be a compact subset of a real separable Hilbert space $H$. Let $F$ be a continuously differentiable functional on $K$. Then, 
    \begin{equation}
        \forall \text{orthonormal basis }\lbrace \phi_0,\phi_1,\ldots \rbrace,  \,\,
        \forall \epsilon > 0, \,\,
        \forall \theta \in K, \,\, 
        \exists M \in \mathbb{N} \,\, s.t. \,\, \forall m \geq M, \,\,  
        \left\| 
        \frac{\delta F([\theta])}{\delta \theta} 
        - 
        P_m
        \frac{\delta F([P_m \theta])}{\delta \theta} 
        \right\|_H < \epsilon,
    \end{equation}
    where $P_m$ is the projection onto $\mathrm{span}\lbrace \phi_0,\ldots,\phi_{m-1}\rbrace$.
\end{theorem}
Now, the convergence becomes pointwise. Technically, this is because the set $K'=\lbrace \frac{\delta F([\th])}{\delta \th}: \th \in K \rbrace$ is not guaranteed to satisfy the equi-small tail condition 
\begin{align}
    \forall \text{ orthonormal basis }\lbrace \phi_0,\phi_1,\ldots \rbrace,  \,\,
    \forall \epsilon' > 0, \,\,
    \exists M \in \mathbb{N} \,\, s.t. \,\, 
    \forall m \geq M, \,\, 
    \forall \theta \in K, \,\, 
    \sum_{k=m}^{\infty}\left|
    \left(
    \frac{\delta F([\theta])}{\delta \theta},\phi_k
    \right)_H
    \right|^2< \epsilon \, ,
\end{align}
while its boundedness $\sup_{\theta \in K} \|\frac{\delta F([\th])}{\delta \th}\|_{H}<\infty$ holds according to Thm.~\ref{lem: Compactness of Frechet derivative}. An example that converges pointwisely but not uniformly is 
\begin{equation}
    F([\theta])
    =
    \begin{cases}
        \sum_{k=1}^\infty e^{-(k-\tan a_0)^2}a_k 
        &
        (0 \leq a_0 < \pi/2)
        \\
        0
        &
        (a_0 = \pi/2)
    \end{cases},\,\,
    a_k=(\theta,\phi_k)_H,
\end{equation}
defined on a compact subset
\begin{equation}
    K=\left\lbrace\, \theta\,
    :\,
    a_k=(\theta,\phi_k)_H, \,\, 
    0\leq a_k \leq \frac{\pi}{2(k+1)} \text{ for }k=0,1,\ldots \right\rbrace.
\end{equation}
Anyways, we have to use large $m$ in either case (Eq.~\eqref{eq: CA of FD with tail term} or \eqref{eq: cylindrical approx of functional derivatives with omitting second term}) when we want to approximate a complicated functional derivative, and the degree $m$ that is required for a sufficiently small approximation error depends on the smoothness, or spectral tail, of $\frac{\delta F([P_m\th])}{\delta \th(x)}$.
As discussed in App.~\ref{app: Abstract Evolution Equations}, while the lack of uniform convergence may affect the convergence of the cylindrical approximation for linear FDEs in general, this is not problematic in our experiment. 
We use Eq.~\eqref{eq: cylindrical approx of functional derivatives with omitting second term} as the approximated functional derivative in our experiment.

\subsubsection{Abstract Evolution Equations} \label{app: Abstract Evolution Equations}
We first provide related theorems to the convergence of equations (\textit{consistency}) \cite{venturi2021spectral} and then those to the convergence of solutions (\textit{stability}) \cite{engel2000one, guidetti2004approximation, venturi2021spectral}.
\paragraph{Definitions.}
Let $\mcf(H)$ be a Banach space of functionals from a real separable Hilbert space $H$ into $\mbr$.
We consider an abstract evolution equation \cite{guidetti2004approximation}
\begin{equation}
    \frac{\partial F([\th], t)}{\partial t} = \mcl([\th]) F([\th], t) \text{ with } F([\th], 0) = F_0([\th]) \, ,
\end{equation}
where $F $ is in $ \mcf(H)$, and $\mcl ([\th])$ is a linear operator in the set of closed, densely-defined, and continuous linear operators on $\mcf(H)$ denoted by $\mcc(\mcf)$.
Let $D(\mcl)$ denote the domain of operator $\mcl$.
Let $\mcf_m(H)$ be the Banach space of functionals on $H$ such that $\mcf_m(H) := \{ F_m \, | \, F_m([\th],t) = F([P_m \th], t) \}$; in other words, $\mcf_m(H)$ is the set of cylindrically approximated functionals.
Using $\mcf(H)$ and $\mcf_m(H)$, we define a continuous linear operator $B_m: \mcf(H) \ni F([\th], t) \mapsto F([P_m \th], t) \in \mcf_m(H)$, which represents the cylindrical approximation of functionals. $B_m$ allows us to decompose the right-hand side of the abstract evolution equation: 
\begin{equation}
    \label{eq: Bm decomposition}
    B_m(\mcl([\th]) F([\th], t)) = \mcl_m([\th]) F_m([\th], t) + R_m([\th], t),
\end{equation}
where $\mcl_m([\th])$ is a linear operator acting on the $m$-dimensional multivariable function $f(a_0,\ldots,a_{m-1}, t) = F_m([\th], t)$ (note that $\mcl_m([\th])$ have nothing to do with $t$), and $R_m([\th], t)$ is the functional residual.

\paragraph{Convergence of equations (consistency).}
Now, we can show that $| \mcl([\th]) F([\th]) - \mcl_m([\th]) F_m([\th])  | = \mathcal{O}(\| \th - P_m \th \|_H)$ \cite{venturi2021spectral}.
\begin{definition}[Consistency of sequence of operators]
    A sequence of linear operators $\{ \mcl_m \} \in \mcc(\mcf_m)$ is consistent with a linear operator $\mcl \in \mcc(\mcf)$ if $\forall F \in D(\mcl)$, $\exists$ a sequence $F_m \in D(\mcl_m)$ s.t. $\| F - F_m \| \xrightarrow{m\rightarrow\infty} 0$ and $\| \mcl F - \mcl_m F_m \| \xrightarrow{m\rightarrow\infty} 0$. If $\| \mcl F - \mcl_m F_m \| \xrightarrow{m\rightarrow\infty} \mathcal{O}(m^{-p})$, $\{ \mcl_m \}$ is consistent with $\mcl$ to order $p(>0)$.
\end{definition}

\begin{theorem}[Consistency of FDEs under cylindrical approximation \cite{venturi2021spectral}] \label{thm: Consistency of cylindrical approximation of FDEs}
    Let $H$ be a real separable Hilbert space. Let $F \in \mcf(H)$ and $\mcl \in \mcl(\mcf)$. If $\mcl([\th]) F([\th])$ is continuous in $\th$, then the sequence of operators $\{ \mcl_m \}$ is consistent with $\mcl$ on arbitrary compact subset $K$ of $H$, provided that $\forall \th \in K$, $| R_m([\th]) | \xrightarrow{m \rightarrow \infty} 0$.
\end{theorem}

\begin{corollary}[Convergence of cylindrical approximation of FDEs \cite{venturi2021spectral}]
    Let $H$ be a real separable Hilbert space. Let $F \in \mcf(H)$ and $\mcl \in \mcl(\mcf)$. If $\mcl([\th]) F([\th])$ is continuous in $\th$, then the sequence of operators $\{ \mcl_m \}$ is consistent with $\mcl$ to the same order as $\| \th - P_m \th \|_H$ on arbitrary compact, convex subset $K$ of $H$, provided that $\forall \th \in K$, $| R_m([\th]) | = \mathcal{O}(\| \th - P_m \th \|_H)$ as $m \rightarrow \infty$ and that $\mcl([\th]) F([\th])$ is continuously Fr\'{e}chet differentiable in $K$. In short, $| \mcl([\th]) F([\th]) - \mcl_m([\th]) F_m([\th])  | = \mathcal{O}(\| \th - P_m \th \|_H)$.
\end{corollary}

\paragraph{Convergence of solutions (stability).}
Next, we show that $F_m([\th], t) \rightarrow F([\th], t)$ if and only if the cylindrical approximation is stable and consistent.
Let us consider the approximated abstract evolution equation $\frac{\partial F_m ([\th], t)}{\partial t} = \mcl_m ([\th]) F_m([\th], t)$ with $F_m([\th], 0) = B_m F_0([\th])$. It is said to be \textit{consistent} with the original abstract evolution equation if Thm.~\ref{thm: Consistency of cylindrical approximation of FDEs} holds.
\begin{definition}[Stability of approximated equation] \label{def: Stability of approximated equation}
    Suppose that $\mcl_m$ of the approximated abstract evolution equation generates a strongly continuous semigroup $e^{t \mcl_m}$.
    The approximated abstract evolution equation is stable if $\exists M, \omega$ s.t. $\| e^{t\mcl_m} \| \leq M e^{\omega t}$, where $M$ and $\om$ are independent of $m$.
\end{definition} 








\begin{theorem}[Convergence of solutions under cylindrical approximation \cite{engel2000one, guidetti2004approximation, venturi2021spectral}] \label{thm: Convergence of solutions under cylindrical approximation in App}
    Let $K$ be a compact subset of a real separable Hilbert space $H$.
    Suppose that the approximated abstract evolution equation is well-posed, in the sense of an initial value problem, in the time interval $[0, T]$ with a finite $T$. 
    Suppose also that $\mcl([\th]) \in \mcc(\mcf)$ generates a strongly continuous semigroup in $[0, T]$. Then, the approximated abstract evolution equation is stable and consistent in $K$ if and only if $\underset{t\in[0,T]}{\max}\underset{\th \in K}{\max} | F_m([\th], t) - F([\th], t) | \xrightarrow{m \rightarrow \infty} 0$, provided that $F_m([\th], 0) \xrightarrow{m \rightarrow \infty} F_0([\th])$. In short, $F_m([\th], t) \rightarrow F([\th], t)$ if and only if the cylindrical approximation is stable and consistent.
\end{theorem}
For example, the cylindrical approximation of the following initial value problem is stable and consistent \cite{venturi2021spectral}: $\frac{\partial F ([\th], t)}{\partial t} = \int_0^{2\pi} \th(x) \frac{\partial }{\partial x}\frac{\delta F([\th], t)}{\delta \th(x)} dx$ with $F([\th], 0) = F_0([\th])$.
To our knowledge, the convergence rate has been unknown so far.

\paragraph{Remark 1.}
Most of the approximation results for compact subsets of real separable Hilbert spaces hold also in compact subsets of Banach spaces admitting a basis. We refer the readers to Sec. 8 in \cite{venturi2021spectral}.

\paragraph{Remark 2.}
Finally, we comment on how the lack of uniform convergence of $\frac{\delta F([\theta])}{\delta \theta} - P_m\frac{\delta F([P_m\theta])}{\delta \theta}$ affects the cylindrical approximation of linear FDEs. The difference $\frac{\delta F([\theta])}{\delta \theta} - P_m\frac{\delta F([P_m\theta])}{\delta \theta}$ is manifested in the functional residual $R_{m}([\theta], t)$ in Eq.~\eqref{eq: Bm decomposition}. 
The lack of uniform convergence may have a negative effect on the consistency of the cylindrical approximation, given that the convergence of $R_{m}([\theta], t)$ is required in Thm.~\ref{thm: Consistency of cylindrical approximation of FDEs}. 
This issue, however, is circumvented in many cases. 
In fact, let us consider the scenario where functional derivatives in an FDE are expressed in terms of the inner-product $(v, \frac{\delta F([\theta])}{\delta \theta})_H$, which is satisfied by our examples ($v = u$ in the FTE and $v = A \boldsymbol{a}$ in Eq.~\eqref{eq: BH cylindrical approx eq}) (see App.~\ref{app: Details about our FDEs}). 
Importantly, its cylindrical approximation $(v, P_m\frac{\delta F([P_m\theta])}{\delta \theta})_H$ uniformly converge to $(v, \frac{\delta F([\theta])}{\delta \theta})_H$ even if $P_m\frac{\delta F([P_m\theta])}{\delta \theta}$ does not uniformly converge to $\frac{\delta F([\theta])}{\delta \theta}$.
\begin{lemma}[Uniform convergence of inner products (ours)]
    \label{lem: Uniform convergence of inner products}
    Let $K$ and $K'$ be compact subsets of a real separable Hilbert space $H$. Let $F$ be a continuously differentiable functional on $K$.
    Then, 
    \begin{equation}
        \forall \epsilon > 0, \,\, \exists M \in \mathbb{N} \,\, 
        s.t. 
        \,\, \forall m \geq M, 
        \,\, \forall \theta \in K, 
        \,\, \forall \theta' \in K', 
        \,\, 
        \left| 
        \left(\theta', \frac{\delta F([\theta])}{\delta \theta} \right)_H
        -
        \left(\theta', P_m\frac{\delta F([P_m \theta])}{\delta \theta} \right)_H
        \right| < \epsilon
    \end{equation}
    holds; i.e., 
    $\left(\theta', P_m\frac{\delta F([P_m \theta])}{\delta \theta} \right)_H$ 
    converges uniformly to 
    $\left(\theta', \frac{\delta F([\theta])}{\delta \theta} \right)_H$ on $K$ and $K'$.
\end{lemma}
The proof is given in App.~\ref{app: Theorem: Pointwise Convergence of Functional Derivatives under Cylindrical Approximation}. In App.~\ref{app: Details about our FDEs}, we employ this lemma to show the consistency of our FDEs.

In short, Lem.~\ref{lem: Uniform convergence of inner products} states that the cylindrical approximation of inner products $(v, P_m\frac{\delta F([P_m\theta])}{\delta \theta})_H$ uniformly converge to $(v, \frac{\delta F([\theta])}{\delta \theta})_H$ even if $P_m\frac{\delta F([P_m\theta])}{\delta \theta}$ does not uniformly converge to $\frac{\delta F([\theta])}{\delta \theta}$.
Because of this mechanism, in App.~\ref{app: Details about our FDEs}, we show that the uniform convergence of $R_m$ is ensured, which is one of the assumptions for stability.
The full proof of the convergence of solutions is provided in App.~\ref{app: Details about our FDEs}.

\clearpage
\section{Proofs I}  
\label{app: Proofs and Derivations}
\subsection{Theorem: Pointwise Convergence of Functional Derivatives under Cylindrical Approximation} 
\label{app: Theorem: Pointwise Convergence of Functional Derivatives under Cylindrical Approximation}
\begin{theorem}[Thm.~\ref{thm: Pointwise convergence of cylindrical approximation (informal)}. Pointwise convergence of cylindrical approximation]
    Let $K$ be a compact subset of a real separable Hilbert space $H$. Let $F$ be a continuously differentiable functional on $K$. Then, 
    \begin{align}
        \forall \text{ orthonormal basis }\lbrace \phi_0,\phi_1,\ldots \rbrace,  \,\,
        \forall \epsilon > 0, \,\,
        \forall \theta \in K, \,\, 
        \exists M \in \mathbb{N} \nn s.t. \,\, \forall m \geq M, \,\,  
        \left\| 
        \frac{\delta F([\theta])}{\delta \theta} 
        - 
        P_m
        \frac{\delta F([P_m\theta])}{\delta \theta} 
        \right\|_H < \epsilon,
    \end{align}
    where 
    $P_m$ is the projection onto $\mathrm{span}\lbrace \phi_0,\ldots,\phi_{m-1}\rbrace$.
\end{theorem}
\begin{proof}
The triangle inequality gives
\begin{align}
    &
    \left\|
    \frac{\delta F([\theta])}{\delta \theta}
    -
    P_m\frac{\delta F([P_m\theta])}{\delta \theta}
    \right\|_{H}
    \leq
    \left\|
    P_m
    \left(
    \frac{\delta F([\theta])}{\delta \theta}
    -
    \frac{\delta F([P_m\theta])}{\delta \theta}
    \right)
    \right\|_{H}
    +
    \left\|
    (1-P_m)
    \frac{\delta F([\theta])}{\delta \theta}
    \right\|_{H}
    \notag
    \\
    &
    \leq
    \left\|
    \frac{\delta F([\theta])}{\delta \theta}
    -
    \frac{\delta F([P_m\theta])}{\delta \theta}
    \right\|_{H}
    +
    \sqrt{
    \sum_{k=m}^{\infty} 
    \left|
    \left( \frac{\delta F([\theta])}{\delta \theta}, \phi_k \right)_H
    \right|^2
    }.
    \label{eq: dFdt dfda ineq}
\end{align}
The first term on the right-hand side converges to zero uniformly according to Thm.~\ref{thm: Uniform convergence of cylindrical approximation of functional derivatives}. 
As for the second term, we first note that $\frac{\delta F([\theta])}{\delta \theta(x)}$ is bounded on $K$ in the sense of a function in $H$, according to Lem.~\ref{lem: Compactness of Frechet derivative}. This, together with Parseval's identity, implies $\| \frac{\delta F([\theta])}{\delta \theta} \|_H^2 = \sum_{k=0}^{\infty}| ( \frac{\delta F([\theta])}{\delta \theta}, \phi_k )_H |^2 < \infty$. Therefore, the sequence of the partial sums $S_m=\sum_{k=0}^{m-1}|( \frac{\delta F([\theta])}{\delta \theta}, \phi_k )_H|^2$ is a convergent sequence and thus is a Cauchy sequence; i.e.,
\begin{equation}
    \forall \epsilon' > 0, \,\,
    \exists M \in \mathbb{N}, \,\,
    s.t.\,\,
    \forall m,n \geq M, \,\,
    |S_m-S_n|<\epsilon'.
\end{equation}
By taking $n\to \infty$, we can claim that
\begin{align}
    \forall \text{ orthonormal basis }\lbrace \phi_0,\phi_1,\ldots \rbrace,  \,\,
    \forall \epsilon' > 0, \,\,
    \forall \theta \in K, \,\, 
    \exists M \in \mathbb{N} \nn 
    s.t. \,\, \forall m \geq M, \,\,  
    \sum_{k=m}^{\infty} 
    \left|
    \left( \frac{\delta F([\theta])}{\delta \theta}, \phi_k \right)_H
    \right|^2 < \epsilon' \, .
\end{align}
Therefore, the second term on the right-hand side of Ineq.~\eqref{eq: dFdt dfda ineq} converges pointwisely.
The theorem was thus proved.
\end{proof}

\paragraph{Convergence rate.}
The convergence rates of the approximated functional derivative can be derived from Ineq.~\eqref{eq: dFdt dfda ineq}.
The first term on the r.h.s. converges at the same rate as $||\theta - P_m \theta||$ (Eq.~\eqref{eq: dFdtheta conv rate}). The convergence rate of $||\theta - P_m \theta||$ depends on the basis functions and is provided in \cite{hesthaven2007spectral} (Chapters 4 \& 6) and \cite{shen2011spectral} (Sec.~3.5) for several bases. 
The convergence rate of the second term on the r.h.s. depends on the compact subset of functions $K \in H$ under consideration, and further assumptions on $K$ are required.

\begin{lemma}[Lem.~\ref{lem: Uniform convergence of inner products}. Uniform convergence of inner products] 
    Let $K$ and $K'$ be compact subsets of a real separable Hilbert space $H$. Let $F$ be a continuously differentiable functional on $K$.
    Then, 
    \begin{align}
        \forall \epsilon > 0, \,\, \exists M \in \mathbb{N} \,\, 
        s.t. 
        \,\, \forall m \geq M, 
        \,\, \forall \theta \in K, 
        \,\, \forall \theta' \in K', 
        \nn
        \left| 
        \left(\theta', \frac{\delta F([\theta])}{\delta \theta} \right)_H
        -
        \left(\theta', P_m\frac{\delta F([P_m \theta])}{\delta \theta} \right)_H
        \right| < \epsilon
    \end{align}
    holds; i.e., 
    $\left(\theta', P_m\frac{\delta F([P_m \theta])}{\delta \theta} \right)_H$ 
    converges uniformly to 
    $\left(\theta', \frac{\delta F([\theta])}{\delta \theta} \right)_H$ on $K$ and $K'$.
\end{lemma}
\begin{proof}
    Using the triangle inequality and the Cauchy-Schwarz inequality, we have
    \begin{align}
        &
        \left| 
        \left(\theta', \frac{\delta F([\theta])}{\delta \theta} \right)_H
        -
        \left(\theta', P_m\frac{\delta F([P_m \theta])}{\delta \theta} \right)_H
        \right|
        \notag
        \\
        \leq&
        \left| 
        \left(\theta', P_m
        \left(\frac{\delta F([\theta])}{\delta \theta} 
        -
        \frac{\delta F([P_m \theta])}{\delta \theta}
        \right)
        \right)_H
        \right|
        +
        \left| 
        \left(\theta', (1-P_m)\frac{\delta F([\theta])}{\delta \theta} \right)_H
        \right|
        \notag
        \\
        \leq&
        \left\| 
        \theta'
        \right\|_H
        \left\| 
        P_m
        \left(
        \frac{\delta F([\theta])}{\delta \theta} 
        -
        \frac{\delta F([P_m \theta])}{\delta \theta} 
        \right)
        \right\|_H
        +
        \left| 
        \left((1-P_m)\theta', \frac{\delta F([\theta])}{\delta \theta} \right)_H
        \right|
        \notag
        \\
        \leq&
        \left\| 
        \theta'
        \right\|_H
        \left\| 
        \frac{\delta F([\theta])}{\delta \theta} 
        -
        \frac{\delta F([P_m \theta])}{\delta \theta} 
        \right\|_H
        +
        \left\| 
        \frac{\delta F([\theta])}{\delta \theta} 
        \right\|_H
        \sqrt{\sum_{k=m}^{\infty}\left|\left(\theta', \phi_k\right)_H\right|^2}. \label{eq: tmp1}
    \end{align}
    Note that $\frac{\delta F([\theta])}{\delta \theta}$ and $\theta'$ are bounded on $K$ and $K'$, respectively; 
    thus, $C$ and $C'$ s.t.
    $\sup_{\theta \in K}\left\| \frac{\delta F([\theta])}{\delta \theta} \right\|_H \leq C<\infty$
    and
    $\sup_{\theta' \in K'}\left\| \theta'\right\|_H \leq C'<\infty$.
    Therefore, Ineq.~\eqref{eq: tmp1} gives
    \begin{align}
        \label{eq: inner prod inequality}
        &\left| 
        \left(\theta', \frac{\delta F([\theta])}{\delta \theta} \right)_H
        -
        \left(\theta', P_m\frac{\delta F([P_m \theta])}{\delta \theta} \right)_H
        \right|
        \leq
        C'
        \left\| 
        \frac{\delta F([\theta])}{\delta \theta} 
        -
        \frac{\delta F([P_m \theta])}{\delta \theta} 
        \right\|_H
        +
        C\sum_{k=m}^{\infty}\left|\left(\theta', \phi_k\right)_H\right|^2.
    \end{align}
    Finally, according to Thm.~\ref{thm: Necessary and sufficient condition of pre-compactness of metric space}, the compactness of $K'$ means
    \begin{align}
        \forall \epsilon' > 0, \,\,
        \exists M \in \mathbb{N} \,\, 
        s.t. \,\, 
        \forall \theta' \in K', \,\, 
        \forall m \geq M, \,\, 
        \sum_{k=m}^{\infty}\left|\left(\theta', \phi_k\right)_H\right|^2<\epsilon'.
    \end{align}
    This, together with Ineq.~\eqref{eq: inner prod inequality} and Thm.~\ref{thm: Uniform convergence of cylindrical approximation of functional derivatives}, proves the lemma.
\end{proof}

\paragraph{Tail term.}
In the cylindrical approximation \eqref{eq: cylindrical approx of functional derivatives with omitting second term}, $P_m$ projects the ``tail term'' $\sum_{k=m}^{\infty} (\frac{\delta F([\theta])}{\delta \theta}, \phi)_H \phi_k(x)$ to zero, unlike the cylindrical approximation \eqref{eq: CA of FD with tail term} adopted in \cite{venturi2018numerical, venturi2021spectral}.
The tail term vanishes if one considers an inner product of a functional derivative and $ v \in D_m$. 
However, it is not always the case that functional derivatives appear in the form of $(\frac{\delta F([\theta])}{\delta \theta}, v)$ \textbf{with $ v \in D_m$} in FDEs. 
In fact, in the FTE, the functional derivative appears as an inner product with $u(x)$, which can be chosen arbitrarily. 
The point is that Lem.~\ref{lem: Uniform convergence of inner products}, which plays a fundamental role in proving the convergence of approximated solutions, guarantees the convergence of the inner product $(\frac{\delta F([\theta])}{\delta \theta}, v)$ in a wider variety of situations including $v \notin D_m$. In other words, our theorems extend the class of FDEs whose uniform convergence of the approximated solution is theoretically guaranteed.

\subsection{Derivation of Exact Solution of Burgers-Hopf Equation} 
\label{app: Derivation of Exact Solution of BHE}
\subsubsection{Derivation of Eq.~\eqref{eq: BH exact solution}}
\label{app: Derivation of Eq. BH exact solution}
We show the derivation of Eq.~\eqref{eq: BH exact solution}. It is based on the functional Taylor expansion
\begin{align}
    \label{eq:ftaylor}
    W([\Theta],\tau)
    =
    \sum_{n=0}^{\infty}
    \frac{1}{n!}
    \int_{-\frac{1}{2}}^{\frac{1}{2}} d\xi_1 
    \cdots 
    \int_{-\frac{1}{2}}^{\frac{1}{2}} d\xi_n
    W^{(n)}(\tau,\xi_1,\ldots,\xi_n)
    \Theta(\xi_1)\cdots\Theta(\xi_n).
\end{align}
This turns Eq.~\eqref{eq: BHE} as follows:
\begin{align}
    \frac{\partial}{\partial \tau}
    W^{(n)}(\tau,\xi_1,\ldots,\xi_n)
    =
    \left(
    \frac{\partial^2}{\partial \xi_1^2}
    +
    \cdots
    +
    \frac{\partial^2}{\partial \xi_n^2}
    \right)
    W^{(n)}(\tau,\xi_1,\ldots,\xi_n).
\end{align}
In the momentum space, this is written in the following form
\begin{align}
    &\frac{\partial}{\partial \tau}
    \tilde{W}^{(n)}(\tau,k_1,\ldots,k_n)
    =
    -\left(k_1^2+\cdots+k_n^2\right)
    \tilde{W}^{(n)}(\tau,k_1,\ldots,k_n),
    \\
    &W^{(n)}(\tau,\xi_1,\ldots,\xi_n)
    =
    \sum_{k_1}\cdots\sum_{k_n}
    e^{-i k_1 \xi_1+\cdots -ik_n \xi_n}
    \tilde{W}^{(n)}(\tau,k_1,\ldots,k_n),
\end{align}
where $k_{1,\cdots,n}=2\pi l_{1,\cdots,n}$ with $l_{1,\cdots,n}\in\mathbb{Z}$. The solution of this equation is
\begin{align}
    \tilde{W}^{(n)}(\tau,k_1,\ldots,k_n)
    =
    e^{-\left(k_1^2+\cdots+k_n^2\right)\tau}
    \tilde{W}^{(n)}(0,k_1,\ldots,k_n).
\end{align}
Using this result, Eq.~\eqref{eq:ftaylor} gives
\begin{align}
    \label{eq:ftaylor_2}
    W([\Theta],t)
    =&
    \sum_{n=0}^{\infty}
    \frac{1}{n!}
    \sum_{k_1}\cdots\sum_{k_n}
    \tilde{W}^{(n)}(\tau,k_1,\ldots,k_n)
    \tilde{\Theta}(-k_1)\cdots\tilde{\Theta}(-k_n)
    \notag
    \\
    =&
    \sum_{n=0}^{\infty}
    \frac{1}{n!}
    \int_{-\frac{1}{2}}^{\frac{1}{2}} d\xi_1
    \cdots 
    \int_{-\frac{1}{2}}^{\frac{1}{2}} d\xi_n
    W^{(n)}(0,x_1,\ldots,x_n)
    \prod_{m=1}^n
    \left(
    \sum_{k_m}e^{-k_m^2 \tau+ik_m \xi_m}\tilde{\Theta}(-k_m)
    \right)
    \notag
    \\
    =&
    \sum_{n=0}^{\infty}
    \frac{1}{n!}
    \int_{-\frac{1}{2}}^{\frac{1}{2}} d\xi_1
    \cdots 
    \int_{-\frac{1}{2}}^{\frac{1}{2}} d\xi_n
    W^{(n)}(0,\xi_1,\ldots,\xi_n)
    \prod_{m=1}^n
    \left(
    \int_{-\frac{1}{2}}^{\frac{1}{2}}d\xi_m'
    \Theta(\xi_m')
    \sum_{k_m}e^{-k_m^2 \tau+ik_m (\xi_m-\xi_m')}
    \right)
    \notag
    \\
    =&
    \sum_{n=0}^{\infty}
    \frac{1}{n!}
    \int_{-\frac{1}{2}}^{\frac{1}{2}} d\xi_1
    \cdots 
    \int_{-\frac{1}{2}}^{\frac{1}{2}} d\xi_n
    W^{(n)}(0,\xi_1,\ldots,\xi_n)
    \prod_{m=1}^n
    \left(
    \int_{-\frac{1}{2}}^{\frac{1}{2}}d\xi_m'
    \Theta(\xi_m')
    \sum_{l_m=-\infty}^{\infty}
    e^{-4\pi^2 \tau l_m^2+2\pi i l_m(\xi_m-\xi_m')}
    \right),
\end{align}
where
\begin{align}
    \tilde{\Theta}(k)
    =
    \int_{-\frac{1}{2}}^{\frac{1}{2}}d\xi
    e^{ik\xi}\Theta(\xi).
\end{align}
The Poisson's summation formula transforms the summation w.r.t. $l_m$:
\begin{align}
    \sum_{l_m=-\infty}^{\infty}
    e^{-4\pi^2\tau l_m^2+2\pi i l_m(\xi_m-\xi_m')}
    =&
    \sum_{q=-\infty}^{\infty}
    \int_{-\infty}^{\infty}dy
    e^{-2\pi i q y}\left(
    e^{-4\pi^2\tau y^2+2\pi i y(\xi_m-\xi_m')}
    \right)
    \notag
    \\
    =&
    \frac{1}{\sqrt{4\pi\tau}}
    \sum_{q=-\infty}^{\infty}
    e^{-\frac{1}{4\tau}
    \left(\xi_m-\xi_m'-q\right)^2}.
\end{align}
Plugging this into Eq.~\eqref{eq:ftaylor_2}, we arrive at the following expression:
\begin{align}
    W([\Theta],\tau)
    =&
    \sum_{n=0}^{\infty}
    \frac{1}{n!}
    \int_{-\frac{1}{2}}^{\frac{1}{2}} d\xi_1
    \cdots 
    \int_{-\frac{1}{2}}^{\frac{1}{2}} d\xi_n
    W^{(n)}(0,\xi_1,\ldots,\xi_n)
    \prod_{m=1}^n
    \left(
    \int_{-\frac{1}{2}}^{\frac{1}{2}}d\xi_m'
    \frac{\Theta(\xi_m')}{\sqrt{4\pi\tau}}
    \sum_{q=-\infty}^{\infty}
    e^{-\frac{1}{4\tau}
    \left(\xi_m-\xi_m'-q\right)^2}
    \right).
\end{align}
From this result, we conclude that the solution is given by
\begin{align*}
    W([\Theta],\tau)
    =
    W_0([\Theta_\tau([\Theta])]),
    \quad
    \Theta_\tau([\Theta],\xi)
    =
    \frac{1}{\sqrt{4\pi\tau}}
    \sum_{q=-\infty}^{\infty}\int_{-\frac{1}{2}}^{\frac{1}{2}}d\xi'
    e^{-\frac{1}{4\tau}\left(\xi-\xi'-q\right)^2}\Theta(\xi')
    \quad
    (\text{Eq.~\eqref{eq: BH exact solution}}).
\end{align*}

\subsubsection{Derivation of Eq.~\eqref{eq: BH exact solution in fourier}}
\label{app: Derivation of BH exact solution in exp}
We show the derivation of Eq.~\eqref{eq: BH exact solution in fourier}. The cylindrical approximation of Eq.~\eqref{eq: BH exact solution} is represented by
\begin{align*}
    W([P_{2M}\Theta],\tau)
    =
    W_0([\Theta_\tau([P_{2M}\Theta])]),
    \quad
    \Theta_\tau([P_{2M}\Theta],\xi)
    =
    \frac{1}{\sqrt{4\pi\tau}}
    \sum_{q=-\infty}^{\infty}
    \int_{-1/2}^{1/2}d\xi'
    e^{-\frac{1}{4\tau}
    \left(\xi-\xi'-q\right)^2}
    P_{2M}\Theta(\xi').
\end{align*}
The basis \eqref{eq: fourier basis} satisfies 
\begin{align}
    \phi_k(\xi+\xi')
    &=
    \begin{cases}
        1 & (k=0)
        \\
        \sqrt{2}\sin(\pi (k+1)(\xi+\xi')) & (k:\text{odd})
        \\
        \sqrt{2}\cos(\pi k(\xi+\xi')) & (k:\text{nonzero even})
    \end{cases}
    \notag
    \\
    &=
    \begin{cases}
        1 & (k=0)
        \\
        \sqrt{2}
        \sin(\pi (k+1)\xi)
        \cos(\pi (k+1)\xi')
        +
        \sqrt{2}
        \cos(\pi (k+1)\xi) 
        \sin(\pi (k+1)\xi')
        &
        (k:\text{odd})
        \\
        \sqrt{2}
        \cos(\pi k\xi)
        \cos(\pi k\xi')
        -
        \sqrt{2}
        \sin(\pi k\xi) 
        \sin(\pi k\xi') 
        &
        (k:\text{nonzero even})
    \end{cases}
    \notag
    \\
    &=
    \begin{cases}
        \cos(\pi (k+1)\xi')
        \phi_{k}(\xi)
        +
        \sin(\pi (k+1)\xi')
        \phi_{k+1}(\xi)
        &
        (k:\text{odd})
        \\
        \cos(\pi k\xi')
        \phi_k(\xi)
        -
        \sin(\pi k\xi') 
        \phi_{k-1}(\xi)
        &
        (k:\text{even})
    \end{cases}.
\end{align}
Using this relation, we evaluate $\Theta_\tau([P_{2M}\Theta], \xi)$ as
\begin{align}
    \Theta_\tau([P_{2M}\Theta], \xi)
    =&
    \frac{1}{\sqrt{4\pi\tau}}
    \sum_{q=-\infty}^{\infty}
    \int_{-1/2}^{1/2}d\xi'
    e^{-\frac{1}{4\tau}(\xi^{\prime}-q)^2}
    \sum_{k=0}^{2M-1}
    a_k \phi_k(\xi'+\xi)
    \notag
    \\
    =&
    \frac{1}{\sqrt{4\pi \tau}}
    \int_{-\infty}^{\infty} d\xi'
    e^{-\frac{\xi^{\prime 2}}{4\tau}}
    \sum_{k=0}^{M-1}
    \left[
    \cos(2\pi k\xi')\phi_{2k}(\xi)
    -
    \sin(2\pi k \xi') \phi_{2k-1}(\xi)
    \right]
    a_{2k}
    \notag
    \\
    &
    +
    \frac{1}{\sqrt{4\pi\tau}}
    \int_{-\infty}^{\infty} d\xi'
    e^{-\frac{\xi^{\prime 2}}{4\tau}}
    \sum_{k=0}^{M-1}
    \left[
    \cos(2\pi (k+1)\xi')
    \phi_{2k+1}(\xi)
    +
    \sin(2\pi (k+1)\xi')
    \phi_{2k+2}(\xi)
    \right]
    a_{2k+1}
    \notag
    \\
    =&
    \sum_{k=0}^{M-1}
    \left(
    e^{-4\pi^2 k^2 \tau}\phi_{2k}(\xi)a_{2k}
    +
    e^{-4\pi^2 (k+1)^2\tau}
    \phi_{2k+1}(\xi)
    a_{2k+1}
    \right),
\end{align}
where $a_k=\left(\phi_k,\Theta\right)_{L_{\mathrm{p}}^2([-1/2,1/2])}$.

When the initial condition is given by Eq.~\eqref{eq: BH Gaussian initial condition},
\begin{align*}
    &
    W([P_{2M}\Theta],\tau)
    =
    -
    \overline{\mu}
    \int_{-1/2}^{1/2} d\xi 
    \Theta_{\tau}(\xi)
    +
    \frac{1}{2}
    \int_{-1/2}^{1/2} d\xi
    \int_{-1/2}^{1/2} d\xi'
    C(\xi,\xi')
    \Theta_{\tau}([P_{2M}\Theta],\xi)
    \Theta_{\tau}([P_{2M}\Theta],\xi')
    \notag
    \\
    =&
    -
    \overline{\mu}a_{0}
    +
    \frac{1}{2}
    \sum_{k=0}^{M-1}
    \sum_{l=0}^{M-1}
    \int_{-1/2}^{1/2} d\xi
    \int_{-1/2}^{1/2} d\xi'
    C(\xi,\xi')
    \notag
    \\
    &
    \quad\quad
    \times
    \left(
    e^{-4\pi^2 k^2 \tau}
    \phi_{2k}(\xi)
    a_{2k}
    +
    e^{-4\pi^2 (k+1)^2\tau}
    \phi_{2k+1}(\xi)
    a_{2k+1}
    \right)
    \left(
    e^{-4\pi^2 l^2 \tau}
    \phi_{2l}(\xi')
    a_{2l}
    +
    e^{-4\pi^2 (l+1)^2\tau}
    \phi_{2l+1}(\xi')
    a_{2l+1}
    \right)
    \notag
    \\
    =&
    -
    \overline{\mu}a_{0}
    +
    \frac{1}{2}
    \sum_{k=0}^{M-1}
    \sum_{l=0}^{M-1}
    \left(
    e^{-4\pi^2 (k^2+l^2) \tau}
    \tilde{C}_{2k,2l}
    a_{2k}
    a_{2l}
    +
    e^{-4\pi^2 (k^2+(l+1)^2) \tau}
    \tilde{C}_{2k,2l+1}
    a_{2k}
    a_{2l+1}
    \right.
    \notag
    \\
    &
    \left.
    \quad
    +
    e^{-4\pi^2 ((k+1)^2+l^2)\tau}
    \tilde{C}_{2k+1,2l}
    a_{2k+1}
    a_{2l}
    +
    e^{-4\pi^2 ((k+1)^2+(l+1)^2)\tau}
    \tilde{C}_{2k+1,2l+1}
    a_{2k+1}
    a_{2l+1}
    \right)
    \quad
    (\text{Eq.~\eqref{eq: BH exact solution in fourier}}),
\end{align*}
where we have introduced
\begin{align*}
    \tilde{C}_{kl}
    =
    \int_{-1/2}^{1/2} d\xi
    \int_{-1/2}^{1/2} d\xi'
    C(\xi,\xi')
    \phi_{k}(\xi)
    \phi_{l}(\xi').
\end{align*}

\clearpage
\section{Proofs II: Details of Functional Transport Equation and Burgers-Hopf Equation 
\label{app: Details about our FDEs}}
In this Appendix, we provide the detailed background of the FTE and BHE and prove Thm.~\ref{thm: main theorem (informal)}, i.e., the convergence of solutions under the cylindrical approximation \eqref{eq: cylindrical approx of functional derivatives with omitting second term}.
The main materials are Lem.~\ref{lem: Uniform convergence of inner products} and Thm.~\ref{thm: Convergence of solutions under cylindrical approximation in App}. Technical assumptions are summarized in Thm.~\ref{thm: Convergence of solutions under cylindrical approximation in App}.

\subsection{Functional Transport Equation} 
\label{app: Functional Transport Equation}
\begin{figure}[htbp]
\vskip 0.1in
    \begin{center}
    \centerline{\includegraphics[width=0.8\columnwidth]{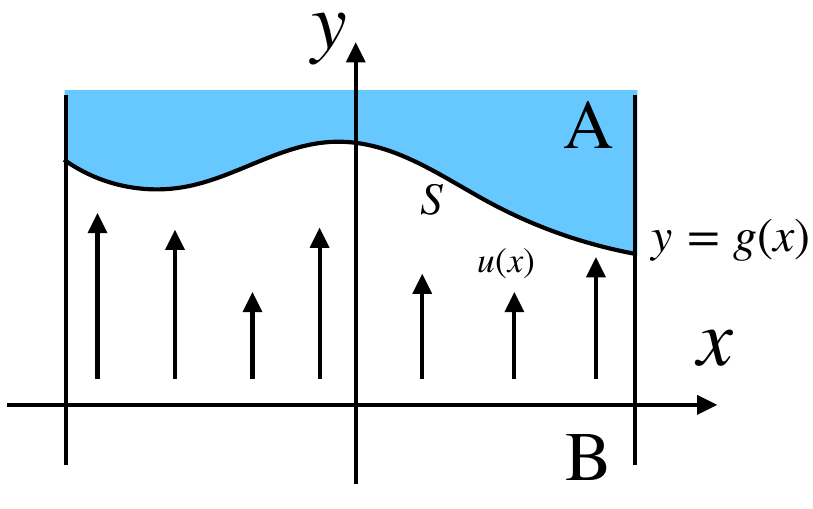}}
    \caption{\textbf{Two-dimensional fluid.} Let $\boldsymbol{u}(\boldsymbol{x})=(0,u(x))^\top$ denote velocity field. The surface $S$ is defined by $y=g(x)$, which separates the space into two areas, A and B.}
    \label{fig: fig_fluid_FTE}
    \end{center}
    \vskip 0.1in
\end{figure}
We introduce the \textit{functional transport equation} (FTE), a generalization of the transport equation (continuity equation) \cite{landau2013fluid_mechanics}. 
We consider a two-dimensional fluid system in $x$-$y$ coordinates (Fig.~\ref{fig: fig_fluid_FTE}). 
We assume that the velocity field is given by $\boldsymbol{u}(\boldsymbol{x}) = (0, u(x))^\top$.
The surface $S$, defined by $y = g(x)$, separates the space into two distinct areas, $y > g(x)$ (area A) and $y < g(x)$ (area B). 

We calculate the amount of fluid entering the area A from B per unit time:
\begin{align}
    \label{eq: amount of incoming fluid}
    F([g],t)=\int dx\, u(x)\rho(x,g(x),t) \, ,
\end{align}
where  $\rho(x, y, t)$ denotes the fluid density at $(x, y)$ and at time $t$.
Differentiating both sides w.r.t. $t$ and using the transport equation for the fluid density
\begin{align}
    \label{eq: density transport equation}
    \frac{\partial}{\partial t}
    \rho(x,y,t)
    =
    -u(x)\frac{\partial}{\partial y}\rho(x, y,t)\, ,
\end{align}
we have
\begin{align}
    \frac{\partial}{\partial t}F([g],t)
    =&
    -\int dx \,
    u(x)\frac{\delta F([g],t)}{\delta g(x)} \, .
\end{align}
We refer to this linear FDE as the FTE.

\subsubsection{Analytic Solution}
Let $L^2([-1,1])$ be a real Hilbert space of functions on $[-1,1]$ with the inner product
\begin{align}
    \left(f,g\right)_{L^2([-1,1])} :=\int_{-1}^1 dx f(x)g(x),
\end{align}
where $f,g \in L^2([-1,1])$.
Let $\mathcal{F}(L^2([-1,1]))$ be a Banach space of functionals that map $L^2([-1,1])$ to $\mathbb{R}$. The functional transport equation 
\begin{align}
    \label{eq: functional transport equation}
    \frac{\partial}{\partial t}F([g],t)
    =&
    -\int_{-1}^1 dx \,
    u(x)\frac{\delta F([g],t)}{\delta g(x)},\,\,
    F([g],0)=F_0([g]),
\end{align}
is a linear FDE in $\mathcal{F}(L^2([-1,1]))$, where $F_0$ is a given functional. Here, $u\in L^2([-1,1])$ is assumed to have a finite norm $\| u\|_{L^2([-1,1])}<\infty$.

The solution of this equation is given by
\begin{align}
    \label{eq: FTE exact solution}
    F([g],t)=&F_0[g_t],\quad 
    g_t(x) := g(x)-u(x)t,
\end{align}
as can be seen immediately by substituting this into Eq.~\eqref{eq: functional transport equation}.

\subsubsection{Cylindrical Approximation and Convergence}
We prove that the solution of the approximated equation converges to that of the non-approximated original equation.
Let $\lbrace \phi_0,\phi_{1},\ldots \rbrace$ be an orthonormal basis of $L^2([-1,1])$ 
and $P_m$ the projection onto $\mathrm{span}\lbrace \phi_0,\ldots,\phi_{m-1} \rbrace$. 
Following App.~\ref{app: Abstract Evolution Equations}, we apply $B_m$ on both sides and the initial condition, to obtain
\begin{align} \label{eq: approximated FTE with Rm}
    \frac{\partial}{\partial t}f(\boldsymbol{a},t)
    =&
    -\sum_{k=0}^{m-1} u_k
    \frac{\partial}{\partial a_k}f(\boldsymbol{a},t)+R_m([g],t),
    \,\,
    \quad f(\boldsymbol{a},0)=f_0(\boldsymbol{a}),
\end{align}
where we have introduced
\begin{align}
    &f(\boldsymbol{a},t) := F([P_m g],t),
    \,\,
    f_0(\boldsymbol{a},t) := F_0([P_m g]),
    \nn
    &a_k := \left(\phi_k,g\right)_{L^2([-1,1])},
    \,\,
    u_k := \left(\phi_k,u\right)_{L^2([-1,1])}.    
\end{align}
The residual functional $R_m$ is defined as
\begin{align}
    R_m([g],t)
    :=&
    -
    \left(u,\frac{\delta F([g],t)}{\delta g}\right)_{L^2([-1,1])}
    +
    \left(u,P_m\frac{\delta F([P_m g],t)}{\delta g}\right)_{L^2([-1,1])}.
\end{align}
The cylindrical approximation of Eq.~\eqref{eq: functional transport equation} is given by
\begin{align} \label{eq: approximated FTE}
    \frac{\partial}{\partial t}f(\boldsymbol{a},t)
    =&
    -\sum_{k=0}^{m-1} u_k
    \frac{\partial}{\partial a_k}f(\boldsymbol{a},t), 
    \,\,
    \quad f(\boldsymbol{a},0)=f_0(\boldsymbol{a}).
\end{align}

We first prove the consistency of Eq.~\eqref{eq: approximated FTE}.
According to Lem.~\ref{lem: Uniform convergence of inner products},\footnote{
    In this case, $K'$ in Lem.~\ref{lem: Uniform convergence of inner products} is set to a singleton $K'=\lbrace u\rbrace$. Obviously, this is compact because $\|u\|_{L^2([-1,1])}<\infty$.
} $R_m([g],t)$ uniformly converges to zero on an arbitrary compact subset $K$ in $L^2([-1,1])$ if $F([g],t)$ is a continuously differentiable functional on $K$. Thus, we can use Thm.~\ref{thm: Consistency of cylindrical approximation of FDEs}, which proves the consistency of the cylindrical approximation \eqref{eq: approximated FTE}.

We next prove the stability  (Def.~\ref{def: Stability of approximated equation}) of the cylindrical approximation \eqref{eq: approximated FTE} in the $L^{\infty}(\mathbb{R}^{m})$ norm. Suppose that $f_0(\boldsymbol{a})$ is bounded by a constant $c$ independent of $m$.\footnote{
    Strictly speaking, if $a_k$ are defined on an infinite interval, this assumption of boundedness is not valid for the FDEs used in our experiments because $\| F_0 \|_{L^{\infty}(\mathbb{R}^{m})} \rightarrow \infty$ as $a_k \rightarrow \infty$. 
    However, the range of $a_k$ is usually set to a finite interval in numerical experiments, and thus $\| F_0 \|_{L^{\infty}(\mathbb{R}^{m})}$ is also finite; i.e., the assumption of boundedness holds.
} 
From the cylindrical approximation of the solution Eq.~\eqref{eq: FTE exact solution}, $f(\boldsymbol{a},t)=F_0[P_m (g(x)-u(x)t)]$, we see
\begin{equation}
    \sup_{\boldsymbol{a}}
    |f(\boldsymbol{a},t)|
    =
    \sup_{\boldsymbol{a}}
    |F_0[P_m (g(x)-u(x)t)]|
    =
    \sup_{\boldsymbol{a}}
    |F_0[P_m g(x)]|
    =
    \sup_{\boldsymbol{a}}
    |f_0(\boldsymbol{a})|.
\end{equation}
Therefore, we obtain
\begin{align}
    \left\| f(t)\right \|_{L^{\infty}(\mathbb{R}^{m})} 
    (=
    \left\| e^{t\mcl_m}f_0\right \|_{L^{\infty}(\mathbb{R}^{m})})
    =
    \left\| f_0\right \|_{L^{\infty}(\mathbb{R}^{m})}\leq c,\,\,
    \forall m\in \mathbb{N},
\end{align}
for the operator $\mcl_m=-\sum_{k=0}^{m-1} u_k\frac{\partial}{\partial a_k}$. From this result, the operator norm of $e^{t\mcl_m}$ is given by
\begin{align}
    \|e^{t\mcl_m}\|
    =
    \sup_{f_0\neq 0}
    \frac{
    \left\| e^{t\mcl_m}f_0\right \|_{L^{\infty}(\mathbb{R}^{m})} 
    }{\left\| f_0\right \|_{L^{\infty}(\mathbb{R}^{m})}}
    =
    \sup_{f_0\neq 0}
    \frac{
    \left\| f(t)\right \|_{L^{\infty}(\mathbb{R}^{m})} 
    }{\left\| f_0\right \|_{L^{\infty}(\mathbb{R}^{m})}}
    =
    1.
\end{align}
This shows that $\exists M, \omega$ s.t.~$\forall m \in \mathbb{N}, \|e^{t\mcl_m} \| \leq Me^{\omega t}$ ($M=1$ and $\omega=0$), and thus the cylindrical approximation is stable in the sense of Def.~\ref{def: Stability of approximated equation}.

According to Thm.~\ref{thm: Consistency of cylindrical approximation of FDEs}, the fact that the cylindrical approximation \eqref{eq: approximated FTE} is consistent and stable guarantees the convergence
$F([P_m g], t) \xrightarrow{m \rightarrow \infty} F([g],t)$ on a compact subset in a finite time interval $[0, T]$ if $F([P_m g], 0) \xrightarrow{m \rightarrow \infty} F([g],0)$ (this is satisfied by the initial conditions used in our experiment) and if the approximated functional transport equation is well-posed in $[0, T]$.

\subsubsection{Initial Conditions}
We define the initial condition $F([g],0)$ in Eq.~\eqref{eq: functional transport equation} as
\begin{align} \label{eq: initial condition FTE}
    F([g],0)
    =
    \rho_0\int_{-1}^1 dx\, u(x)g(x) \,.
\end{align}
The exact solution is given by
\begin{align}
    F([g],t)
    =&
    \rho_0
    \int_{-1}^1 dx\, u(x)\left(g(x)-u(x)t\right) \, .
\end{align}
Under the cylindrical approximation, the solution is
\begin{align} \label{eq: solution FTE}
    F([P_m g],t) = f(\boldsymbol{a}, t)
    =& \rho_0\sum_{k=0}^{m-1} u_k (a_k - u_k t) \, .
\end{align}
With $\upsilon_0$ being a constant and $\{\phi_k\}_{k \geq 0}$ being the Legendre polynomials, we consider two types of $u(x)$:
\begin{align}
    u_k :=& \upsilon_0 \text{ for $k=1$, otherwise } 0,  \\
    u_k :=& \upsilon_0 \text{ for $k \leq 14$, otherwise } 0,
\end{align}
corresponding to two types functional transport equations \eqref{eq: functional transport equation}, initial conditions \eqref{eq: initial condition FTE}, and solutions \eqref{eq: solution FTE}. 
For convenience, we call them the linear initial condition and the nonlinear initial condition, respectively, though the form of Eq.~\eqref{eq: functional transport equation} also changes in accordance with the form of $u(x)$.

\subsection{Burgers-Hopf Equation} 
\label{app: Burgers-Hopf Equation}
The BHE, the one-dimensional analog of Eq.~\eqref{eq: Navier-Stokes Hopf equation} in App.~\ref{app: Turbulence}, describes the statistical properties of one-dimensional fluids and is a basic tool for studying turbulence, as mentioned in App.~\ref{app: Turbulence}.
We consider fluid in a one-dimensional box $[-L/2,L/2]$ that evolves in accordance with the Burgers equation
\begin{align}
    \frac{\partial u(x,t)}{\partial t}+u(x,t)\frac{\partial u(x,t)}{\partial x}
    =
    \nu \frac{\partial^2 u(x,t)}{\partial x^2},
\end{align}
where $u(x,t)$ is a velocity field and $\nu$ is the kinematic viscosity. 
Specifically, we focus on the case where the initial value $u(x,0)$ is given randomly, thus making $u(x,t)$ a random field. Let us introduce the characteristic functional
\begin{align}
    \Phi([\theta],t)
    =
    \mathop{\mathbb{E}}\limits_{\lbrace u(x,0) \rbrace\sim \mathcal{P}_0}
    \left[\exp\left(i\int_{-L/2}^{L/2} dx\, u(x,t)\theta(x)\right)\right],
\end{align}
where $\mathcal{P}_0$ is the probability distribution for the initial velocity field $u(x,0)$, and $\theta(x)$ is called the test function. It provides statistical properties of the velocity field because the functional derivatives are equal to the moments of the velocity field:
\begin{align}
    \mathop{\mathbb{E}}\limits_{\lbrace u(x,0) \rbrace\sim \mathcal{P}_0}
    \left[u(x_1,t)\cdots u(x_n,t)
    \right]
    =
    (-i)^n
    \left.
    \frac{\delta^n \Phi([\theta],t)}
    {
    \delta \theta(x_1)
    \cdots
    \delta \theta(x_n)
    }
    \right|_{\theta=0},
\end{align}
as is already shown in Eq.~\eqref{eq: derivative of characteristic functional} in App.~\ref{app: Turbulence}. 
The time evolution of $\Phi([\theta],t)$ is known to follow the BHE \cite{venturi2018numerical}:
\begin{align}
    \label{eq: BH original equation}
    \frac{\partial \Phi([\theta],t)}{\partial t}
    =
    \int_{-L/2}^{L/2} dx\,
    \theta(x)
    \left(
    \frac{i}{2}
    \frac{\partial}{\partial x}\frac{\delta^2 \Phi([\theta],t)}{\delta \theta(x)\delta \theta(x)}
    +
    \nu 
    \frac{\partial^2}{\partial x^2}
    \frac{\delta \Phi([\theta],t)}{\delta \theta(x)}
    \right).
\end{align}
The initial condition $\Phi([\theta],0)$ is determined by $\mathcal{P}_0$. For example, in the  case of the Gaussian random field, we have
\begin{align} \label{eq: BH original initial condition}
    \Phi([\theta],0)
    =
    \exp\left(
    i\int_{-L/2}^{L/2} dx \mu(x)\theta(x)-\frac{1}{2}\int_{-L/2}^{L/2} dx\int_{-L/2}^{L/2} dx' C(x,x')\theta(x)\theta(x')
    \right),
\end{align}
where $\mu(x)$ is the mean velocity and $C(x,x')$ is the infinite-dimensional covariance matrix.

We modify Eqs.~\eqref{eq: BH original equation} and \eqref{eq: BH original initial condition} to facilitate numerical computation. 
First, we replace $-i\theta(x)$ with $\theta(x)$ and regard $\theta(x)$ as a real-valued function to avoid complex numbers. 
Second, we make Eqs.~\eqref{eq: BH original equation} and \eqref{eq: BH original initial condition} dimensionless, a common convention in numerical computation, by introducing
\begin{equation}
    T:=\frac{L^2}{\nu},\quad
    \tau:=\frac{t}{T},\quad
    \xi:=\frac{x}{L},\quad
    \Theta(\xi):=\frac{L^2}{T}\theta(L\xi) \, .
\end{equation}
Third, we introduce 
\begin{equation}
    W([\Theta],\tau) :=\ln \Phi([\theta],t)
\end{equation}
to remove the exponential function in $\Phi([\theta],t)$ and stabilize numerical computation.
Note that the functional derivatives of $W([\Theta],\tau)$ give the cumulants, not moments, of the velocity field.
Fourth, we neglect the advection term (the first term on the right-hand side of Eq.~\eqref{eq: BH original equation}).
We derive the analytic solution for the equation in the following.

\subsubsection{Analytic Solution}
Let us consider $L_{\mathrm{p}}^2([-1/2,1/2])$, a real Hilbert space of periodic functions on $[-1/2,1/2]$ with the inner product
\begin{align}
    \left(f,g\right)_{L_{\mathrm{p}}^2([-1/2,1/2])} := \int_{-1/2}^{1/2} dx f(x)g(x) \, ,
\end{align}
where $f, g \in L_{\mathrm{p}}^2([-1/2,1/2])$.
Let $\mathcal{F}(L_{\mathrm{p}}^2([-1/2,1/2]))$ be a Banach space of functionals that map $L_{\mathrm{p}}^2([-1/2,1/2])$ to $\mathbb{R}$. 
In $\mathcal{F}(L_{\mathrm{p}}^2([-1/2,1/2]))$, the BHE without the advection term is given by
\begin{align}
    \label{eq: BHE}
    \frac{\partial W([\Theta],\tau)}{\partial \tau}
    =
    \int_{-\frac{1}{2}}^{\frac{1}{2}} d\xi \Theta(\xi)\frac{\partial^2}{\partial \xi^2}\frac{\delta W([\Theta],\tau)}{\delta \Theta(\xi)},
    \quad
    W([\Theta],0)=W_0([\Theta]).
\end{align}
The exact solution of this equation is given by
\begin{align}
    \label{eq: BH exact solution}
    W([\Theta],\tau)
    =
    W_0([\Theta_\tau([\Theta],\xi)]),
    \quad
    \Theta_\tau([\Theta],\xi)
    :=
    \frac{1}{\sqrt{4\pi\tau}}
    \sum_{q=-\infty}^{\infty}\int_{-\frac{1}{2}}^{\frac{1}{2}}d\xi'
    e^{-\frac{1}{4\tau}\left(\xi-\xi'-q\right)^2}
    \Theta(\xi'),
\end{align}
where $W_0$ is a given functional. The proof is a bit technical and is given in App.~\ref{app: Derivation of Eq. BH exact solution}.

\subsubsection{Cylindrical Approximation and Convergence}
In this section, we prove the convergence of the solution for the BHE under the cylindrical approximation.
Let $\lbrace \phi_0,\phi_{1},\ldots \rbrace$ be an orthonormal basis of $L_{\mathrm{p}}^2([-1/2,1/2])$ and $P_m$ be the projection onto $\mathrm{span}\lbrace \phi_0,\ldots,\phi_{m-1} \rbrace$. 
Following App.~\ref{app: Abstract Evolution Equations}, we apply $B_m$ on both sides and the initial condition, to obtain
\begin{align} \label{eq: BH cylindrical approx eq with Rm}
    \frac{\partial}{\partial \tau}w(\boldsymbol{a},\tau)
    =&
    \sum_{k,l=0}^{m-1} 
    A_{kl}
    a_k
    \frac{\partial}{\partial a_l}
    w(\boldsymbol{a},\tau)
    +
    R_m([\Theta],\tau),
    \,\,
    w(\boldsymbol{\alpha},0)
    =
    w_0(\boldsymbol{a}),
\end{align}
where we have introduced
\begin{align}
    &w(\boldsymbol{a},\tau):=W([P_m \Theta],\tau),
    \,\,
    w_0(\boldsymbol{a},\tau):=W_0([P_m \Theta]),
    \,\,
    a_k:=\left(\phi_k,\Theta\right)_{L_{\mathrm{p}}^2([-1/2,1/2])},
    \nn
    &A_{kl}
    :=
    \left(
    \varphi_k,\frac{\partial^2 \varphi_l}{\partial \xi^2}
    \right)_{L_{\mathrm{p}}^2([-1/2,1/2])}.
\end{align}
The residual functional is defined as
\begin{align}
    R_m([\Theta],\tau)
    =&
    \left(
    \frac{\partial^2 \Theta}{\partial \xi^2}, 
    \frac{\delta W([\Theta],\tau)}{\delta \Theta}
    \right)_{L_{\mathrm{p}}^2([-1/2,1/2])}
    -
    \left(
    \frac{\partial^2 \Theta}{\partial \xi^2}, 
    P_m\frac{\delta W([P_m\Theta],\tau)}{\delta \Theta}
    \right)_{L_{\mathrm{p}}^2([-1/2,1/2])}.
\end{align}
The cylindrical approximation of Eq.~\eqref{eq: BHE} is defined as
\begin{align} \label{eq: BH cylindrical approx eq}
    \frac{\partial}{\partial \tau}w(\boldsymbol{a},\tau)
    =&
    \sum_{k,l=0}^{m-1} 
    A_{kl}
    a_k
    \frac{\partial}{\partial a_l}
    w(\boldsymbol{a},\tau) \, ,
    \,\,
    w(\boldsymbol{a},0)
    =
    w_0(\boldsymbol{a}).
\end{align}

We first prove the consistency of Eq.~\eqref{eq: BH cylindrical approx eq}.
Let $K$ be a compact subset in $L_{\mathrm{p}}^2([-1/2,1/2])$ such that $\overline{\lbrace \frac{\partial^2 \Theta}{\partial \xi^2} : \Theta \in K \rbrace}$ is compact. According to Lem.~\ref{lem: Uniform convergence of inner products}, $R_m([W],\tau)$ uniformly converges to zero on $K$ if $W([\Theta],\tau)$ is a continuously differentiable functional on $K$. Thus, we can use Thm.~\ref{thm: Consistency of cylindrical approximation of FDEs}, which proves the consistency of Eq.~\eqref{eq: BH cylindrical approx eq}.

We next prove the stability (Def.~\ref{def: Stability of approximated equation}) of Eq.~\eqref{eq: BH cylindrical approx eq} in the $L^{\infty}(\mathbb{R}^{m})$ norm. 
Suppose that $w_0(\boldsymbol{a})$ is bounded by a constant $c$ independent of $m$.\footnote{
    Strictly speaking, if $a_k$ are defined on an infinite interval, this assumption of boundedness is not valid for the FDEs used in our experiments because $\| w_0 \|_{L^{\infty}(\mathbb{R}^{m})} \rightarrow \infty$ as $a_k \rightarrow \infty$. 
    However, the range of $a_k$ is usually set to a finite interval in numerical experiments, and thus $\| w_0 \|_{L^{\infty}(\mathbb{R}^{m})}$ is also finite; i.e., the assumption of boundedness holds.
}
From the cylindrical approximation of the solution Eq.~\eqref{eq: FTE exact solution}, $w(\boldsymbol{a},t)=W_0([P_m\Theta_{\tau}[P_m\Theta]])$, we see
\begin{align}
    \sup_{\boldsymbol{a}\in\mathbb{R}^{m}}|w(\boldsymbol{a},\tau)|
    =
    \sup_{\boldsymbol{a}\in\mathbb{R}^{m}}|W_0([P_m\Theta_{\tau}[P_m\Theta]])|
    =
    \sup_{\boldsymbol{a}\in S_\tau}|W_0([P_m\Theta])|
    \leq 
    \sup_{\boldsymbol{a}\in\mathbb{R}^{m}}|W_0([P_m\Theta])|
    =
    \sup_{\boldsymbol{a}\in\mathbb{R}^{m}}|w_0(\boldsymbol{a})|,
\end{align}
where
\begin{align}
    &S_\tau
    =
    \left\lbrace
    \boldsymbol{a}':
    a_l'
    =
    \sum_{k=0}^{m-1}
    \left(
    \frac{1}{\sqrt{4\pi\tau}}
    \sum_{q=-\infty}^{\infty}
    \int_{-1/2}^{1/2}d\xi
    \int_{-1/2}^{1/2}d\xi'
    e^{-\frac{1}{4\tau}
    \left(\xi-\xi'-q\right)^2}
    \phi_l(\xi)\phi_k(\xi')
    \right)
    a_k,
    \,\,
    \boldsymbol{a}\in \mathbb{R}^{m}
    \right\rbrace \nn
    &\subseteq
    \mathbb{R}^{m}.
\end{align}
Therefore, we obtain
\begin{align}
    \left\| w(\tau)\right \|_{L^{\infty}(\mathbb{R}^{m})} 
    =
    \left\| e^{t\mcl_m}w_0\right \|_{L^{\infty}(\mathbb{R}^{m})} 
    \leq
    \left\| w_0\right \|_{L^{\infty}(\mathbb{R}^{m})}\leq c,\,\,
    \forall m\in \mathbb{N},
\end{align}
for the operator $\mcl_m=-\sum_{k,l=0}^{m-1} A_{kl}a_k\frac{\partial}{\partial a_l}$. From this result, the operator norm of $e^{t\mcl_m}$ is evaluated as
\begin{align}
    \|e^{t\mcl_m}\|
    =
    \sup_{w_0\neq 0}
    \frac{
    \left\| e^{t\mcl_m}w_0\right \|_{L^{\infty}(\mathbb{R}^{m})} 
    }{\left\| w_0\right \|_{L^{\infty}(\mathbb{R}^{m})}}
    \leq
    1.
\end{align}
This shows that $\exists M, \omega$ s.t.~$\forall m \in \mathbb{N}, e^{t\mcl_m}\leq Me^{\omega t}$ (e.g.,~$M=1$ and $\omega=0$), and thus the cylindrical approximation is stable in the sense of Def.~\ref{def: Stability of approximated equation}.

Therefore, the cylindrical approximation \eqref{eq: BH cylindrical approx eq} is consistent and stable, and Thm.~\ref{thm: Convergence of solutions under cylindrical approximation in App} gives the convergence $W([P_m \Theta], \tau) \xrightarrow{m \rightarrow \infty} W([\Theta],\tau)$ on a compact subset in finite time interval $[0, T]$ if $W([P_m \Theta], 0) \xrightarrow{m \rightarrow \infty} W([\Theta],0)$ and if Eq.~\eqref{eq: BH cylindrical approx eq} is well-posed in $[0, T]$.

\subsubsection{Initial Conditions}
We use the Gaussian random field as the initial condition:
\begin{align}
    \label{eq: BH Gaussian initial condition}
    W_0([\Theta])
    =
    -
    \overline{\mu}
    \int_{-1/2}^{1/2} d\xi 
    \Theta(\xi)
    +
    \frac{1}{2}
    \int_{-1/2}^{1/2} d\xi
    \int_{-1/2}^{1/2} d\xi'
    C(\xi,\xi')
    \Theta(\xi)
    \Theta(\xi'),
\end{align}
where $\overline{\mu}$ is the mean velocity and $C(\xi,\xi')$ is the covariance matrix. 
This is equivalent to the initial distribution of the velocity field following the Gaussian distribution \eqref{eq: BH Gaussian initial condition}.
We use the Fourier series as the orthonormal basis:
\begin{align}
    \label{eq: fourier basis}
    \phi_{k}(\xi)=
    \begin{cases}
        1 & (k=0)
        \\
        \sqrt{2}\sin(\pi (k+1)\xi) & (k:\text{odd})
        \\
        \sqrt{2}\cos(\pi k\xi) & (k:\text{nonzero even})
    \end{cases}.
\end{align}
Then, the analytic solution of Eq. \eqref{eq: BH exact solution} under the cylindrical approximation is given by
\begin{align}
    \label{eq: BH exact solution in fourier}
    &W([P_{2M}\Theta],\tau)
    =\nn
    &-
    \overline{\mu}
    a_{0}
    +
    \frac{1}{2}
    \sum_{k=0}^{M-1}
    \sum_{l=0}^{M-1}
    \left(
    e^{-4\pi^2 (k^2+l^2) \tau}
    \tilde{C}_{2k,2l}
    a_{2k}
    a_{2l}
    +
    e^{-4\pi^2 (k^2+(l+1)^2) \tau}
    \tilde{C}_{2k,2l+1}
    a_{2k}
    a_{2l+1}
    \right.
    \nn
    &
    \left.
    \quad
    +
    e^{-4\pi^2 ((k+1)^2+l^2)\tau}
    \tilde{C}_{2k+1,2l}
    a_{2k+1}
    a_{2l}
    +
    e^{-4\pi^2 ((k+1)^2+(l+1)^2)\tau}
    \tilde{C}_{2k+1,2l+1}
    a_{2k+1}
    a_{2l+1}
    \right),
\end{align}
where $m = 2M$ and
\begin{align}
    a_k:=\left(\phi_k,\Theta\right)_{L_{\mathrm{p}}^2([-1/2,1/2])},
    \,\,
    \tilde{C}_{ij}
    :=
    \int_{-1/2}^{1/2} d\xi
    \int_{-1/2}^{1/2} d\xi'
    C(\xi,\xi')
    \phi_{i}(\xi)
    \phi_{j}(\xi').
\end{align}
The derivation is lengthy and is given in App.~\ref{app: Derivation of BH exact solution in exp}.
Note that the higher degree terms decay exponentially in terms of $k$, $l$, and $\tau$, and the solution is dominated by $a_k$ with $k \lesssim$ 1.

We use three types of covariance matrices:
\begin{align}
    & C(\xi, \xi') = \sigma^2 \delta(\xi - \xi'), \label{eq: Cov of delta IC} \\
    & C(\xi, \xi') = \sigma^2 \sum_{k=0}^{99} e^{-k / 10} \phi_k(\xi) \sum_{l=0}^{99} e^{-l / 10} \phi_l(\xi') ,\label{eq: Cov of moderate IC} \\
    & C(\xi, \xi') = \sigma^2 . \label{eq: Cov of const IC}
\end{align}
Substituting them into Eq.~\eqref{eq: BH Gaussian initial condition}, we have three types of initial conditions: the delta, moderate, and constant initial condition, respectively.
They are equivalent to
\begin{align}
    & \tilde{C}_{ij} = \sigma^2 \text{ for all $i=j \geq 0$ (0 otherwise)}, \label{eq: Cov of delta IC in fourier} \\
    & \tilde{C}_{ij} = \sigma^2  e^{-i / 10} e^{-j / 10} \text{ for $i = j \leq 99$  (0 otherwise)},  \label{eq: Cov of moderate IC in fourier} \\
    & \tilde{C}_{ij} = \sigma^2 \text{ for $i, j = 0$  (0 otherwise)} \,. \label{eq: Cov of const IC in fourier}
\end{align}

Eq.~\eqref{eq: Cov of delta IC} is nonsmooth and represents the extremely short-range correlation of the initial velocity field; the velocities at two points in an infinitesimally small neighborhood have no correlation. The spectrum of $C(\xi, \xi')$ (Eq.~\eqref{eq: Cov of delta IC in fourier}) has an infinite tail.
Eq.~\eqref{eq: Cov of const IC} represents the extremely long-range correlation of the initial velocity; the velocities at any two points have the same correlation $\sigma^2$. The spectrum of $C(\xi, \xi')$ (Eq.~\eqref{eq: Cov of const IC in fourier}) decays immediately. 
Eq.~\eqref{eq: Cov of moderate IC} represents a moderate-range correlation of the two above; the velocities at two points have a periodic correlation. The spectrum of $C(\xi, \xi')$ (Eq.~\eqref{eq: Cov of moderate IC in fourier}) decays exponentially.
Theoretically, it is said that $C(\xi, \xi')$ that has a long tail of spectrum is hard to simulate numerically \cite{venturi2018numerical}. In our experiment with PINNs, however, the error of the solution is dominated by the optimization error of PINNs (Sec.~\ref{sec: Experiment}) and strongly depends on the training setups.

\clearpage
\section{Supplementary Discussion} \label{app: Supplementary Discussion}
\subsection{Limitations} \label{app: Limitations}
\paragraph{Higher spacetime dimensions.}
In our experiments, the spacetime dimension is limited to $1+1$ ($t$ and $x$).
Generalization to $1+d$ dimensions is feasible, albeit with increased computational costs.
For spaces with $d>1$ dimensional spaces, multiple options exist for expanding $\theta$ \cite{rodgers2024tensor}.
Despite the inclusion of additional spacetime dimensions, the computational complexity of our model, up to computing functional derivatives, remains $\mathcal{O}(m^{r})$, where $r$ is the order of the FDE.
However, note that $m$ scales exponentially w.r.t. $d$, which is typically $1, 2, 3,$ or $4$. 

\paragraph{Higher-order functional derivatives.}
The order of functional derivatives in FDEs in our experiments is limited to $r=1$; however, extending to $r\geq2$ is straightforward. 
For example, the cylindrical approximation of the second-order functional derivative is expressed as $\delta^2 F([\theta], t)/ \delta \theta(x) \delta \theta(y) \approx \Sigma_{k,l=0}^{m-1} \partial^2 f(\boldsymbol{a}, t) / \partial a_k \partial a_l \phi_k(x) \phi_l (y)$, which can be computed via backpropagation twice.

\paragraph{To further reduce errors.}
The relative errors obtained in our experiments are $\sim 10^{-3}$, which might be unsatisfactory in some applications
This is a general problem inherent in PINNs, for which PINNs are sometimes criticized.
To further reduce the errors, one can use recently developed techniques for solving high-dimensional PINNs \cite{zeng2022competitive, hu2023tackling, wang2023multi, hu2023bias, daw2023mitigating, yao2023multadam, gao2023failure}. 
These methods can be easily equipped with our model; one of the strengths of our model is that it can be integrated with arbitrary techniques developed for PINNs.
See also App.~\ref{app: Radial Gaussian Sampling} for a better sampling of collocation points.

\paragraph{Challenges toward even higher degrees.}
Instability of the numerical integration for computing $a_k$ given a function $\th(x)$, i.e., $a_k = \int \th(x) \phi_k(x) dx$, is observed for large $k$s. 
This instability comes from the intense oscillation of $\phi_k$ for large $k \gtrsim 500$, as we confirmed in our preliminary experiments.
This is a general problem in numerical computation, not exclusive to our model.
Nevertheless, such higher degrees are unnecessary to approximate smooth functions.
Extremely large degrees are required only when one wants to include nonsmooth or divergent functions in the domain of functionals $D(F)$.
Training on such functions with the cylindrical approximation is an open problem
Dealing with highly oscillatory functions is also a central research interest in numerical methods for non-smooth dynamical systems \cite{acary2008numerical}.

\paragraph{Inclusion and diversity of functions}
What class of functions can the cylindrical approximation represent? 
The equi-small tail condition (Thm.~\ref{thm: Necessary and sufficient condition of pre-compactness of metric space}) characterizes the domain of functions $D(F)$ to which the cylindrical approximation can be applied. 
In most of the convergence theorems in App.~\ref{app: Theorems Related to Cylindrical Approximation and Convergence}, such as Thm.~\ref{thm: Uniform convergence of cylindrical approximation of functionals}, the compactness of $K \subset H$ is assumed, which can be ensured if Thm.~\ref{thm: Necessary and sufficient condition of pre-compactness of metric space} holds. 
In other words, the equi-small tail condition is part of a sufficient condition for the convergence of the cylindrical approximation. 
Surprisingly or not, step functions and ReLU satisfy the equi-small tail condition. 
In contrast, functions with divergent norms ($|\theta|_H = \infty$), which are not typically of interest, cannot be handled in the cylindrical approximation.

\paragraph{Inclusion and diversity of functionals}
We assume functional differentiability in this paper. Central theorems of differentiability are given in App.~\ref{app: Differentiation of Functionals}, where functional Lipschitzness is assumed. An example of non-Lipschitz functionals is $F([\th]) = \int dx \sqrt{ |\theta(x)| }$ with $\th(0) = 0$.

\paragraph{Inclusion and diversity of FDEs}
In this paper, we consider the abstract evolution equation, a crucial class of linear FDEs. It does not include, for instance, Eq.~\eqref{eq:Schwinger-Dyson} in quantum field theory or Eq.~\eqref{eq:Wetterich-eq} in functional renormalization group theory. Nevertheless, the cylindrical approximation \textit{is} applicable to these equations, although the theoretical (non-)convergence of solutions is currently unknown. Finally, note that the approximated functional derivatives converge to the non-approximated ones independently of the specific FDEs.

\subsection{Computational Complexity of CP-ALS} \label{app: Computational Complexity of CP-ALS}
In Sec.~\ref{sec: Introduction}, we mentioned the computational complexity of the state-of-the-art method (CP-ALS) ($\mathcal{O}(m^6 T)$) 
We derived it from Eqs.~(567) and (568) in \cite{venturi2018numerical} in Sec.~7.2.2 in \cite{venturi2018numerical}. 
These equations require $\mathcal{O}(m^5)$ including the summation over the indices $e$ ($=0,\ldots, m^2$) and $z$ ($=0,\ldots, m^2$), and the Hadamard matrix product over the index $k$ ($=1,\ldots, m$). 
They are computed for each $q$ ($=1,\ldots m$) and timestep $t$ ($=1, \ldots, T$), leading to $\mathcal{O}(m^5 \times m \times T) = \mathcal{O}(m^6 T)$ in total. 
Note that Eqs.~(567) and (568) in the published version \cite{venturi2018numerical} correspond to Eqs.~(568) and (569) in the arXiv version (\texttt{arXiv:1604.05250v3}).

\subsection{Choice of Bases} \label{app: Orthogonal Bases Compared}
We use the Fourier series for periodic systems and the Legendre polynomials for non-periodic systems. Both are common basis functions in numerical computation. Legendre polynomials are part of the Jacobi polynomials family, a typical class of orthogonal bases. Many Jacobi polynomials are numerically unstable and require careful treatment; for example, the Hermite polynomials change scales factorially with degree ($1/\sqrt{m!}$) and often result in \texttt{nan}. The Laguerre polynomials have similar issues. The Chebyshev polynomials of the first kind and the Legendre polynomials are suitable for numerical computation; they are bounded, defined on a finite interval, and their weight functions are regular (note that the weight function of the Chebyshev polynomials of the first kind, $1/\sqrt{1-x^2}$, can be easily regularized by changing variables). Moreover, the Chebyshev and Legendre polynomials are known to minimize the $L^\infty$ and $L^2$ approximation errors, respectively \cite{hesthaven2007spectral}. Nevertheless, we did not observe a significant performance difference between them in curve-fitting experiments, and we use the Legendre polynomials in our experiments.

\paragraph{How to choose basis functions?}
The choice depends on the FDE, boundary conditions, symmetry, function spaces, and numerical stability. The Fourier series is suitable for periodic functions, but for non-periodic functions, it exhibits the Gibbs phenomenon, causing large numerical errors. Therefore, Legendre polynomials are a good choice for non-periodic functions. The choice of bases is a common concern in numerical analysis, such as in the finite element and spectral methods.


\paragraph{Generalization to other bases: Riesz basis}
Our approach can be generalized to general non-orthonormal bases, such as Riesz bases \cite{guo2006riesz,rabah2003generalized,brugiapaglia2021sparse,fukuda2013on,kurdyumov2010riesz,kurdyumov2008riesz,burlutskaya2009riesz}, where the orthonormality condition $(\phi_k, \phi_l)_H = \delta_{kl}$ is replaced with $(\phi_k, \phi_l)_H = g_{kl}$. Here, $g_{kl}$ is a certain matrix, or a \textit{metric}, whose choice significantly affects the computational costs of the approximated FDEs. Sparse metrics are preferred for efficient computation. See also App.~\ref{app: Computational Complexity of Loss Function}.

\subsection{Computational Complexity Revisited} \label{app: Computational Complexity of Loss Function}
We showed the computational complexity of our model \textit{up to the computation of functional derivatives} is given by $\mcO(m^r)$.
This does not include the computational complexity of the loss function of PINNs.
It could be $> \mcO(m^r)$, strongly depending on the form of the FDE and/or the choice of the orthonormal basis.
For concreteness, let us consider Eq.~\eqref{eq: BH cylindrical approx eq} (BHE).
While the BHE is first-order, the computational complexity of the loss function can be $\mcO(m^2)$ if matrix $A_{kl}$ in Eq.~\eqref{eq: BH cylindrical approx eq} is not diagonal. 
This is the case if we use the Legendre polynomials instead of the Fourier series.

On the other hand, the choice of the basis function can improve computational complexity.
For example, the computational complexity of the second-order FDE can reduce from $\mcO(m^2)$ to $\mcO(m^1)$ when the approximated FDE includes only the diagonal elements of $\{ \partial^2 f(\boldsymbol{a},t)/ \partial a_k \partial a_l \}_{k,l=0}^{m-1}$. 
Again, whether it is the case or not depends on the form of the FDE and/or the choice of the orthonormal basis.

\subsection{Comparison with DFT}
In first-principles computations of density functional theory (DFT), an NN-based approach that utilizes finite element methods for spacetime grid approximation is commonly employed. 
For example, an NN, $\hat{F}(\{y_j := f(\mathbf{r}_j)\}_j)$, approximating a target functional $F([f])$ by evaluating $f$ at specific grid points $\{\mathbf{r}_j\}_j$. 
Functional derivatives at each grid point can be computed using automatic differentiation: $\{ \frac{\delta F([f])}{\delta f(\boldsymbol{r})} \}_{\boldsymbol{r}} \fallingdotseq \{ \frac{\partial \hat{F}(\{ y_j \}_j)}{\partial y_i}\}_i$. 
However, the central focus of this area does not include solving PDEs, let alone FDEs.
Moreover, this method requires the discretization of spacetime, leading to numerical error of derivatives, while in our approach, such discretization is not necessary, and spacetime differentiation can be performed analytically if the basis functions are analytic.

The finite element methods for spacetime and the cylindrical approximation (basis function expansion) are similar in the sense that they replace spacetime degrees of freedom with other quantities: spacetime grid points and expansion coefficients, respectively.
Even when their degrees of freedom for a given problem are comparable, the latter (our approach) is suitable for solving FDEs.

\subsection{FDE Has Multiple Meanings}
The term "Functional Differential Equation" (FDE) varies in meaning across research areas. 
In mathematics, an FDE typically refers to a differential equation with a deviating argument, e.g., $f^\prime(x) = x f(x + 3)$. 
In physics, it denotes a PDE involving functional derivatives, such as the abstract evolution equation described in the main text. 
Additionally, FDE can imply a functional equation involving functional derivatives, e.g., $\frac{\delta F([\theta])}{\delta \theta(x)} = A(x)F([\theta])$, where $A(x)$ is a given function. 
Research on the last interpretation is limited, and the uniqueness, existence, and practical applications remain ambiguous.


\subsection{Radial Gaussian Sampling} \label{app: Radial Gaussian Sampling}
The region $\th \sim 0$, or ${a_k \sim 0}_k$ is the central focus when solving the BHE. 
Thus, sampling as many collocation points from $\th \sim 0$ as possible is of crucial importance. 
To make the training set free from the curse of dimensionality, we can use the following sampling method based on the polar coordinates:
\begin{enumerate}
    \item Sample a unit vector $\mathbf{v} \in \mbr^d$, where $d$ is the dimension of the input space.
    \item Sample a real number (radius) $r > 0$ from a Gaussian distribution $\mcn(0, \sigma^2)$ (or a truncated Gaussian distribution).
    \item Obtain a training collocation point  $r \mathbf{v} \in \mbr^d$.
\end{enumerate}
The collocation points thus obtained obviously concentrate around $\th \sim 0$ more than the Latin hypercube sampling, which we used in our experiment.

\subsection{Curvilinear Patterns in Fig.~\ref{fig: fig_heat_AnalSolPredAbsWithTauVsAk_FTE_deg100_w1024_it500k_stat_model0_FTE} and Noisy Patterns in Fig.~\ref{fig: fig_edit_heat_AnalSolPredRelWithTauVsAk_BHE_init0_deg10_sig1e1_stat_model0} and }
The white curves in Fig.~\ref{fig: fig_heat_AnalSolPredAbsWithTauVsAk_FTE_deg100_w1024_it500k_stat_model0_FTE} represent locations where the predictions and the exact solutions happen to coincide. 
The noisy patterns that occur for $a_k$, $k\geq 1$ in Fig.~\ref{fig: fig_edit_heat_AnalSolPredRelWithTauVsAk_BHE_init0_deg10_sig1e1_stat_model0} are because the solution is almost independent of $a_k$ with $k \gtrapprox 1$. Thus, optimizing the model in the directions of $a_k$ with $k \gtrapprox 1$ have only a negligible effect on minimizing the loss function, keeping the random predictions of the model.
If we delve deeper, they might be related to the loss landscapes of PINNs \cite{basir2022critical, gopakumar2023loss}, which is also an interesting research topic.

\subsection{Second-order Functional Derivatives}
\begin{figure}[htbp]
 \centering
    \centerline{\includegraphics[width=\columnwidth]{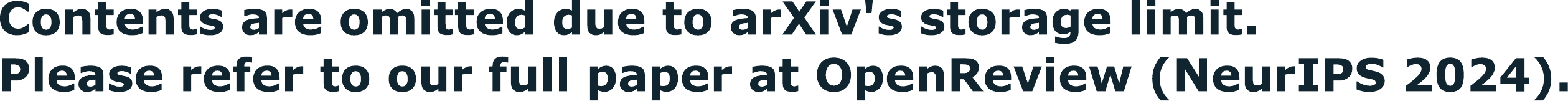}}
    \caption{{\color{blue} This figure is omitted due to arXiv's storage limit. Please refer to our full paper at OpenReview (NeurIPS 2024) or \url{https://github.com/TaikiMiyagawa/FunctionalPINN}.}
    \textbf{Second-order derivative at $\th=0$ \& $t=1$.} 
    The model is trained on the FTE with degree 100 under the linear initial condition. The errors are as small as $\sim 10^{-3}$. 
    }
    \label{fig: fig_heat_SecondMomentPredAndAnaAndRelAndAbsWithXYT1_FTE_deg100_w1024_it500k_stat_model0_FTE}
\end{figure}
\textbf{Fig.~\ref{fig: fig_heat_SecondMomentPredAndAnaAndRelAndAbsWithXYT1_FTE_deg100_w1024_it500k_stat_model0_FTE}} shows the estimated second-order derivative of the FTE's solution at $\th = 0$ and $t=0$ under 
the linear initial condition. The error is as small as $\sim 10^{-3}$, despite that $\th \sim 0$ is not included in the training set due to the curse of dimensionality and that the second-order information is not included in the loss function. 

However, such small errors only emerge when the analytic solution is so simple that, e.g., it includes only one coefficient: $f(\boldsymbol{a}, t) = \rho_0 \upsilon_0(a_1 - \upsilon_0 t)$.
The errors of the second-order derivatives obtained from other conditions and equations were significant ($\gtrsim 10^3$ in some cases). 
We could reduce the error by (i) using another sampling method for training sets that does not suffer from the curse of dimensionality (an example is given in App.~\ref{app: Radial Gaussian Sampling}) and (ii) adding the second-order derivative to the loss function (e.g., differentiate both sides of the BHE w.r.t. $a_k$ and use it to the residual loss for PINNs).

\clearpage
\section{Detailed Experimental Settings} \label{app: Detailed Experimental Settings}
We describe detailed setups of our experiment, including those on the advection-reaction functional equation (ARE) (App.~\ref{app: Advection-reaction Equation}).
The experimental settings for the ARE align with those provided in \cite{venturi2018numerical}.

\paragraph{Datasets.}
The training, validation, and test sets are sampled with Latin hypercube sampling \cite{mckay2000comparison, helton2003latin} with range $[0,1], [-1, 1], [0, 0.001]$, $[-0.1, 0.1]$, $[0, 0.5]$, and $[h-0.01, h+0.01]$ for $t$ (FTE), $a_k$ (FTE), $t$ (BHE), $a_k$ (BHE), $t$ (ARE), and $a_k$ (ARE), respectively, where $h := 0.698835274542439$.
Training mini-batches are randomly sampled every time a mini-batch is created.
For the BHE, we decay the sampling range quadratically in terms of $k \in \{0,1,\ldots,m\}$. 
This sampling stabilizes the training because most of the loss comes from the region where $k$ and $a_k$ are large, which can be seen from the definition of the BHE. That is, all the terms on the right-hand side are proportional to one of $a_k$s, while $a_k$s for large $k$s are negligible in the analytic solution. Therefore, the collocation points with large $a_k$s with large $k$s can be seen as noise in training.
Note that this decaying sampling does not limit the quality of the solution at all because the solution is the characteristic function and we are interested only in $\theta (x) \approx 0$, e.g., $\frac{\delta W([\th],t)}{\delta \th(x)}|_{\th=0} = \langle u(x) \rangle$ (see also App.~\ref{app: Turbulence}).

\paragraph{Miscellaneous settings.}
We use a 4-layer PINN with $3 \times$ (linear + sin activation + layer normalization) + last linear layer.
Unless otherwise noted, the widths are 1024 for the ARE and FTE and 2048 for the BHE.
The batch size is 1024.
The activation function is the $\sin$ function.
The loss function for the FTE with the nonlinear initial condition (main text) and the ARE is the sum of the $L^1$ and $L^\infty$ losses.
The loss function for the others is the smooth $L^1$ loss.
The optimizer is AdamW \cite{loshchilov2018AdamW_SGDW}.
The learning rate scheduler is the linear warmup with cosine annealing with warmup \cite{loshchilov2017sgdr_cosine_annealing_org}, which is defined as the following \texttt{scheduler}:
\begin{verbatim}
scheduler1 = torch.optim.lr_scheduler.LinearLR(
    optimizer, start_factor=start_factor, total_iters=milestone)
scheduler2 = torch.optim.lr_scheduler.CosineAnnealingWarmRestarts(
    optimizer, T_0=T_0, T_mult=T_mult, eta_min=eta_min)
scheduler = torch.optim.lr_scheduler.SequentialLR(
    optimizer, schedulers=[scheduler1, scheduler2], milestones=[milestone])
\end{verbatim}


See PyTorch official documentation \cite{pytorch} for the detailed definitions of the parameters.
\texttt{eta\_min} is set to 0.
\texttt{start\_factor} is set to $10^{-10}$.
$\rho_0 = 1$ and $\upsilon_0 = 1$ are used for the FTE.
$\bar{\mu} = 8$ and $\sigma^2 = 10$ are used for the BHE.
The Legendre polynomials and real Fourier series are used for the FTE and BHE, respectively, as the orthonormal basis.
For the ARE, the modified Chebyshev polynomials described in \cite{venturi2018numerical} are used.

\paragraph{Hyperparameter tuning.}
The full search space is given Tab.~\ref{tab: search space of hparameters}.
Unless otherwise noted, the number of iterations is 300,000 and 500,000 for the BHE and FTE, respectively.
The number of trials for hyperparameter tuning is 50.
Optuna \cite{Optuna} is used for hyperparameter tuning with the TPE sampler and median pruner.   
The number of tuning trials is 50 for all conditions, and the best hyperparameters within the 50 trials are used.
See the config file in our repository for more details.
\begin{table}[htbp]
    \caption{\textbf{Search space of hyperparameters.} Learning rates and weight decays are sampled log-uniformly. $N$ = 300,000 and 500,000 for the BHE and FTE, respectively. \texttt{T\_0}, \texttt{T\_mult}, and \texttt{milestone} are used for the aforementioned \texttt{scheduler}. For the ARE and the main text results of FTE with the nonlinear initial condition, weight decay is fixed to $0$ and the search space of learning rates is from $10^{-7}$ to $10^{-4}$.} \label{tab: search space of hparameters}
    \begin{center}
    \begin{small}
    \begin{sc}
        \begin{tabular}{cc}
             Hyperparameters & Search space\\
             \hline
             Learning rate & From $10^{-6}$ to $10^{-3}$ \\
             Weight decay & From $10^{-7}$ to $10^{-4}$ (others)\\
             \texttt{T\_0} & $N/5$, $N/2$, or $N$ \\
             \texttt{T\_mult} &  1 or 2 \\
             \texttt{milestone}& 0, $N / 10$, or $N / 100$ \\
        \end{tabular}
    \end{sc}
    \end{small}
    \end{center}
\end{table}

\clearpage
\paragraph{Remaining hyperparameters.}
We show the hyperparameters used in our experiments in Tabs.~\ref{tab: hparams0}--\ref{tab: hparams-1}
\begin{table}[htbp]
    \caption{
        \textbf{Hyperparameters: FTE with linear initial condition.}
    } 
    \label{tab: hparams0}
    \begin{center}
    \begin{small}
    \begin{sc}
        \begin{tabular}{rccc}
            \toprule
                \multicolumn{4}{c}{Degree} \\ 
                \cmidrule(l){2-4}
            Hyperparameters & $4$ & $20$ & $100$  \\
            \hline
            Learning rate & $ 8.675 \times 10^{-6}$ & $ 1.878 \times 10^{-5}$ & $ 2.839 \times 10^{-5}$ \\
            Weight decay & $ 4.534 \times 10^{-6}$ & $ 6.184 \times 10^{-7}$& $ 2.008 \times 10^{-5}$\\
            \texttt{T\_0} & $ 250000$ & $500000 $ & $ 500000$ \\
            \texttt{T\_mult} &  $2 $ & $1 $ & $ 1$  \\
            \texttt{milestone}& $ 0$ & $ 5000$ & $0 $  \\
        \end{tabular}
    \end{sc}
    \end{small}
    \end{center}
\end{table}

\begin{table}[htbp]
    \caption{
        \textbf{Hyperparameters: FTE with nonlinear initial condition (main text).}
    } 
    \begin{center}
    \begin{small}
    \begin{sc}
        \begin{tabular}{rccc}
            \toprule
                \multicolumn{4}{c}{Degree} \\ 
                \cmidrule(l){2-4}
            Hyperparameters & $4$  & $100$ & $1000$ \\
             \hline
             Learning rate & $ 8.456 \times 10^{-6} $ & $ 1.042 \times 10^{-5}$& $ 1.130 \times 10^{-5}$ \\
             Weight decay & $0$&  $0$& $0$\\
             \texttt{T\_0} & $500000 $ & $ 500000$ & $250000 $ \\
             \texttt{T\_mult} &  $1 $ & $  1$ & $ 1$ \\
             \texttt{milestone}& $5000 $ &  $ 0$& $ 5000$\\
        \end{tabular}
    \end{sc}
    \end{small}
    \end{center}
\end{table}

\begin{table}[htbp]
    \caption{
        \textbf{Hyperparameters: FTE with nonlinear initial condition (Appendix).}
    } 
    \begin{center}
    \begin{small}
    \begin{sc}
        \begin{tabular}{rccc}
            \toprule
                \multicolumn{4}{c}{Degree} \\ 
                \cmidrule(l){2-4}
            Hyperparameters & $4$ & $10$ & $20$  \\
             \hline
            Learning rate & $ 4.837 \times 10^{-6}$ & $ 5.810 \times 10^{-5}$ & $ 6.076 \times 10^{-6}$ \\
            Weight decay & $  8.637 \times 10^{-5}$ & $  2.877 \times 10^{-5}$& $ 2.068 \times 10^{-7}$\\
            \texttt{T\_0} & $ 250000$ & $ 500000$ & $ 500000$  \\
            \texttt{T\_mult} &  $2 $ & $2 $ & $2 $  \\
            \texttt{milestone}& $0 $ & $50000 $ & $50000 $ \\
        \end{tabular}
    \end{sc}
    \end{small}
    \end{center}
\end{table}

\begin{table}[htbp]
    \caption{
        \textbf{Hyperparameters: BHE with delta initial condition.}
    } 
    \begin{center}
    \begin{small}
    \begin{sc}
        \begin{tabular}{rccc}
            \toprule
                \multicolumn{4}{c}{Degree} \\ 
                \cmidrule(l){2-4}
            Hyperparameters & $4$ & $20$ & $100$  \\
             \hline
            Learning rate & $ 6.680 \times 10^{-5}$ & $ 1.372 \times 10^{-4}$& $ 1.204 \times 10^{-4}$ \\
            Weight decay & $ 4.429 \times 10^{-6}$ & $ 6.644 \times 10^{-7}$& $ 4.893 \times 10^{-7}$\\
            \texttt{T\_0} & $300000 $ & $300000 $ & $150000 $ \\
            \texttt{T\_mult} &  $ 2$ & $  1$ & $  1$ \\
            \texttt{milestone}& $0 $ & $0 $ & $ 3000$ \\
        \end{tabular}
    \end{sc}
    \end{small}
    \end{center}
\end{table}

\begin{table}[htbp]
    \caption{
        \textbf{Hyperparameters: BHE with constant initial condition.}
    } 
    \begin{center}
    \begin{small}
    \begin{sc}
        \begin{tabular}{rccc}
            \toprule
                \multicolumn{4}{c}{Degree} \\ 
                \cmidrule(l){2-4}
            Hyperparameters & $4$ & $20$ & $100$  \\
             \hline
            Learning rate & $5.131  \times 10^{-5}$ & $9.989  \times 10^{-5}$& $ 8.637 \times 10^{-5}$ \\
            Weight decay & $ 1.083 \times 10^{-5}$ & $ 9.353 \times 10^{-7}$& $ 8.378 \times 10^{-5}$\\
            \texttt{T\_0} & $ 300000$ & $150000 $ & $300000$ \\
            \texttt{T\_mult} &  $ 2$ & $1 $ & $1 $ \\
            \texttt{milestone}& $3000 $ & $ 3000$ & $30000 $ \\
        \end{tabular}
    \end{sc}
    \end{small}
    \end{center}
\end{table}

\begin{table}[htbp]
    \caption{
        \textbf{Hyperparameters: BHE with moderate initial condition.}
    } 
    \begin{center}
    \begin{small}
    \begin{sc}
        \begin{tabular}{rccc}
            \toprule
                \multicolumn{4}{c}{Degree} \\ 
                \cmidrule(l){2-4}
            Hyperparameters & $4$ & $20$ & $100$  \\
             \hline
            Learning rate & $2.248  \times 10^{-5}$ & $ 3.254 \times 10^{-5}$& $ 1.209 \times 10^{-5}$ \\
            Weight decay  & $3.669\times 10^{-5}$ & $ 1.771  \times 10^{-6}$& $ 1.423 \times 10^{-6}$\\
            \texttt{T\_0} & $300000 $ & $ 300000$ & $300000$ \\
            \texttt{T\_mult} &  $1 $ & $ 2$ & $1 $ \\
            \texttt{milestone}& $3000 $ & $ 0$ & $0 $ \\
        \end{tabular}
    \end{sc}
    \end{small}
    \end{center}
\end{table}

\begin{table}[htbp]
    \caption{
        \textbf{Hyperparameters: ARE.}
    } 
    \label{tab: hparams-1}
    \begin{center}
    \begin{small}
    \begin{sc}
        \begin{tabular}{rc}
            \toprule
                \multicolumn{2}{c}{\quad\quad\quad\quad\quad\quad\quad\quad\quad\, Degree} \\ 
                \cmidrule(l){2-2}
            Hyperparameters & $6$  \\
             \hline
            Learning rate & $6.3105  \times 10^{-5}$ \\
            Weight decay & $0$ \\
            \texttt{T\_0} & $250000 $  \\
            \texttt{T\_mult} &  $ 1$   \\
            \texttt{milestone}& $ 5000$  \\
        \end{tabular}
    \end{sc}
    \end{small}
    \end{center}
\end{table}

\clearpage 
\paragraph{How to choose weight parameters for each training.}
We use validation relative error, i.e., the relative error between the prediction and the analytic solution, as the measure for choosing the best weight parameter during each training trial.

\paragraph{Loss reweighting.}
The softmax reweighting is used to adaptively balance the residual and boundary loss functions. See the code below and our official code at GitHub (\url{https://github.com/TaikiMiyagawa/FunctionalPINN}), where \texttt{list\_tensors} is [residual loss, boundary loss] of a mini-batch.
\begin{verbatim}
@torch.no_grad()
def softmax_coeffs(
    list_tensors: List[Tensor], 
    temperature: float = 0.25) -> Tensor:
    """
    # Args
    - list_tensors: List of scalar loss Tensors. 
    - temperature: Temperature parameter. Default is 0.25

    # Returns
    - coeffs: Tensor with shape [len(list_tensors),].
    """
    ts = torch.tensor(list_tensors)
    ts /= torch.max(ts) + EPSILON  # generates scale-invariant weight
    coeffs = torch.softmax(ts / temperature, dim=0)
    return coeffs
\end{verbatim}
This function generates the loss reweighting factors \texttt{coeffs} = [$\lambda_1$, $\lambda_2$], where total loss is defined as $\lambda_1 \times \text{\rm residual loss} + \lambda_2 \times \text{\rm boundary loss}$.

\paragraph{Float vs. double.}
No significant differences in performance measures were observed throughout our training. 
Accordingly, we used float32 to reduce computational time.

\paragraph{Libraries and GPUs.}
All the experiments are performed on Python 3.11.8 \cite{Python3}, PyTorch 2.2.0 \cite{pytorch}, Numpy 1.26.0 \cite{numpy}, and Optuna 3.5.0 \cite{Optuna}.
An NVIDIA A100 GPU is used.

\paragraph{Runtime.}
The training process takes 2--3 hours, but this duration can be significantly improved because it largely depends on the implementation of components such as data loaders. See App.~\ref{app: Runtime} for more details on runtime. 
GPU memory consumption is about 400 MBs for the ARE and 800--1500 MBs for the FTE and BHE. 
We note that previous approaches to numerically solving FDEs rely on CPU processing, where performance varies significantly with the level of parallelization. Consequently, directly comparing the computational speed of our GPU-based approach with these methods is inherently challenging.

\paragraph{Figure \ref{fig: fitting}.}
$F([\theta])$ is approximated by $F([P_{1000} \theta])$ for convenience.

\clearpage
\section{Additional Experimental Results} \label{app: Additional Results I}

\subsection{$P_m \theta$ Converges to $\theta$}
The convergence rate of functional derivatives in Thm.~\ref{thm: Pointwise convergence of cylindrical approximation (informal)} is given by $\|\theta - P_m \theta \|$. 
The convergence of $\|\theta - P_m \theta \|$ is visualized in Fig.~\ref{fig: fitting functions} for four functions including non-smooth functions.

\begin{figure}[htbp]
    \centering
    \centerline{\includegraphics[width=\columnwidth]{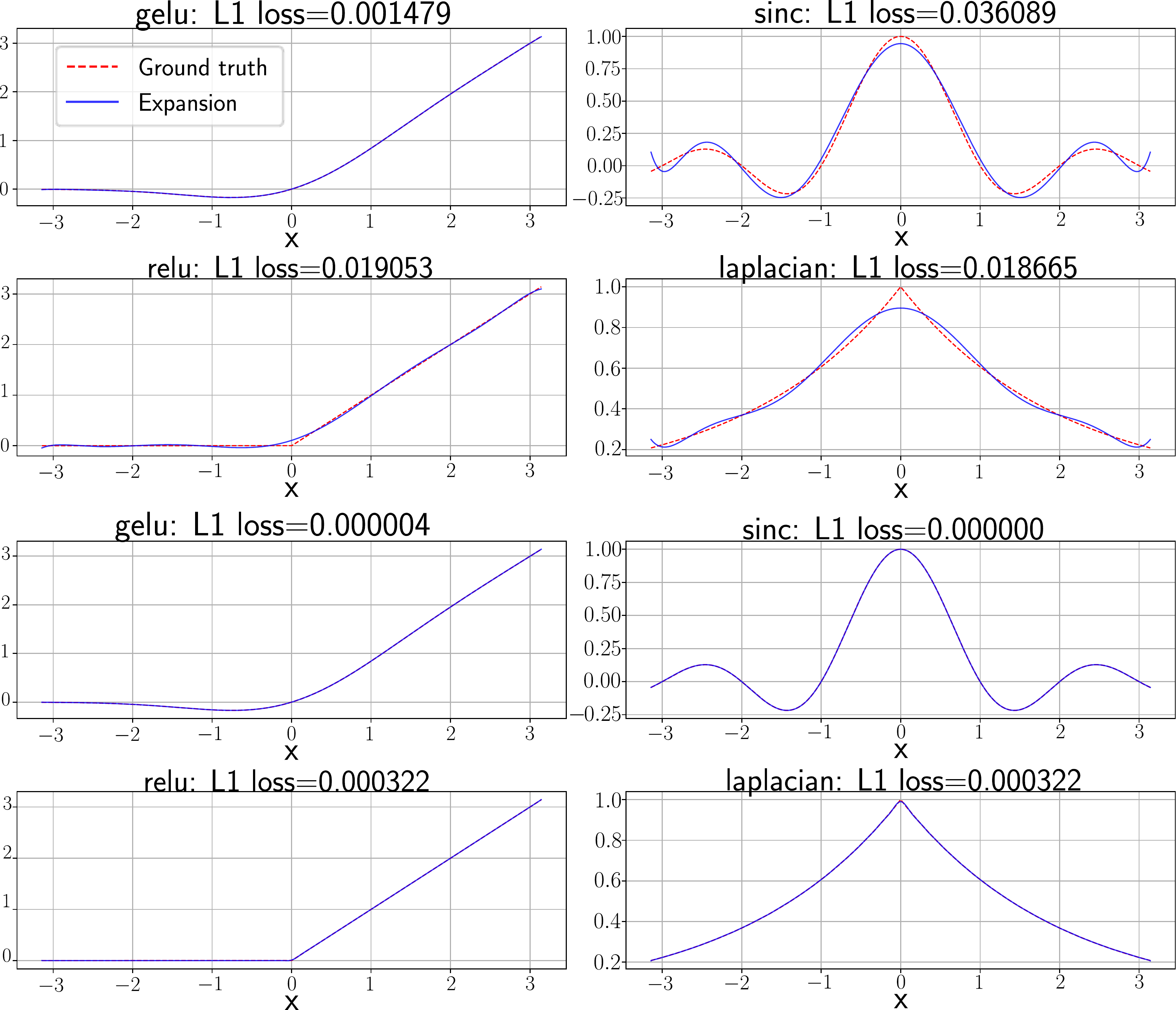}}
    \caption{
        \textbf{Cylindrical approximation of $\th(x)$ (red) by $P_m\th(x)$ (blue) with degree 10 (top four panels) and 100 (bottom four panels).}
        $L^1$ absolute error $\frac{1}{N} \sum_{i=1}^{N} | \th(x_i)-P_m\th(x_i) |$ decreases as $m$ increases. 
        $N = 10^5$. 
        $x_i \in [-\pi, \pi]$ are linearly spaced. 
        GeLU \cite{Hendrycks2016gelu_org}, sinc, ReLU, and Laplacian functions \cite{ramasinghe2022beyond} are used.
        The basis is the Legendre polynomials.
    }
    \label{fig: fitting functions}
\end{figure}

\subsection{Runtime}
\label{app: Runtime}
Tab.~\ref{tab: runtime} shows training runtime.
The runtime remains nearly constant and does not significantly vary from $m \sim 1$ to $m \sim 100$. This aspect is of practical significance as it indicates that the computational time of our model remains stable across a broad range of $m$.

This consistency in runtime may be attributed to the efficiency of the PyTorch library, where operations such as matrix multiplication and automatic differentiation are highly optimized. On the other hand, GPU memory consumption does show variation, ranging from 800 to 1500 MBs as stated App.~\ref{app: Detailed Experimental Settings}, as is naturally expected.

Finally, the previous approaches to numerically solving FDEs rely on CPU processing, where performance significantly varies with the level of parallelization; thus, a direct comparison of computational speed with our GPU-based approach is inherently challenging.

\begin{table}[htbp]
    \caption{\textbf{Degree vs. runtime.} ``X:YZ'' represents ``X hours YZ minutes.''} \label{tab: runtime}
    \begin{center}
    \begin{small}
    \begin{sc}
        \begin{tabular}{rcccc}
            \hline
                    && FTE with 500k iterations \\
            \hline
            Degree  & $4$             & $20$            & $100$           & $1000$ \\
            \#Runs   & 11              & 8               & 3               & ... \\
            Runtime & 2:18 $\pm$ 0:04 & 2:09 $\pm$ 0:10 & 2:33 $\pm$ 0:02 & ... \\
            \\
            \hline
                    && BHE with 300k iterations \\
            \hline
            Degree  & $4$             & $20$            & $100$ \\
            \#Runs   & 10              & 9               & 8 \\
            Runtime & 2:36 $\pm$ 0:02 & 2:29 $\pm$ 0:06 & 2:28 $\pm$ 0:00 \\
            \\
            \hline
                    && ARE with 500k iterations \\
            \hline
            Degree  & $6$             &             &  \\
            \#Runs   & 10              &               & \\
            Runtime & 3:05 $\pm$ 0:04 &  &  \\            
        \end{tabular}
    \end{sc}
    \end{small}
    \end{center}
\end{table}

\clearpage
\subsection{Latin Hypercube Sampling and Curse of Dimensionality} 
\label{app: Latin Hypercube Sampling and Curse of Dimensionality}
We show the histograms of the $L^2$ norm of collocation points in the test sets.
The sampling method is the Latin hypercube sampling, as mentioned in Sec.~\ref{sec: Experiment}.
Apparently, the collocation points s.t. norms $\sim 0$ are not sampled, which would cause the curse of dimensionality in training.
In other words, the Latin hypercube sampling cannot completely address the curse of dimensionality in sampling.

\paragraph{Why does extrapolation mean (Sec.~\ref{sec: Experiment})?}
Firstly, Figs.~\ref{fig: fig_hist_NormOfCoeffs_TestSet_FTE_de4_w1024_it500k_stat_FTE}--\ref{fig: fig_hist_NormOfCoeffs_TestSet_BHE_init1_deg50_sig1e1_stat_BHE} show that no collocation points such that $\| \boldsymbol{a} \| \approx 0$ were included in the training sets.
Secondly, Figs.~\ref{fig: fig_heat_AnalSolPredAbsWithTauVsAk_FTE_deg100_w1024_it500k_stat_model0_FTE} \& \ref{fig: fig_edit_heat_AnalSolPredRelWithTauVsAk_BHE_init0_deg10_sig1e1_stat_model0} plot the errors of PINN predictions at the collocation points where $\boldsymbol{a} = (0, 0, \dots, 0, a_k, 0, \dots, 0)$ with $k = 0, 1, 2, 19,$ or $99$.
Therefore, the collocation points such that $a_k \approx 0$ in Figs.~\ref{fig: fig_heat_AnalSolPredAbsWithTauVsAk_FTE_deg100_w1024_it500k_stat_model0_FTE} \& \ref{fig: fig_edit_heat_AnalSolPredRelWithTauVsAk_BHE_init0_deg10_sig1e1_stat_model0} were not included in the training sets, but the errors were as small as the region where $a_k \not\approx 0$. We refer to this as "extrapolation."

\begin{figure}[htbp]
\vskip 0.1in
    \begin{center}
    \centerline{\includegraphics[width=\columnwidth]{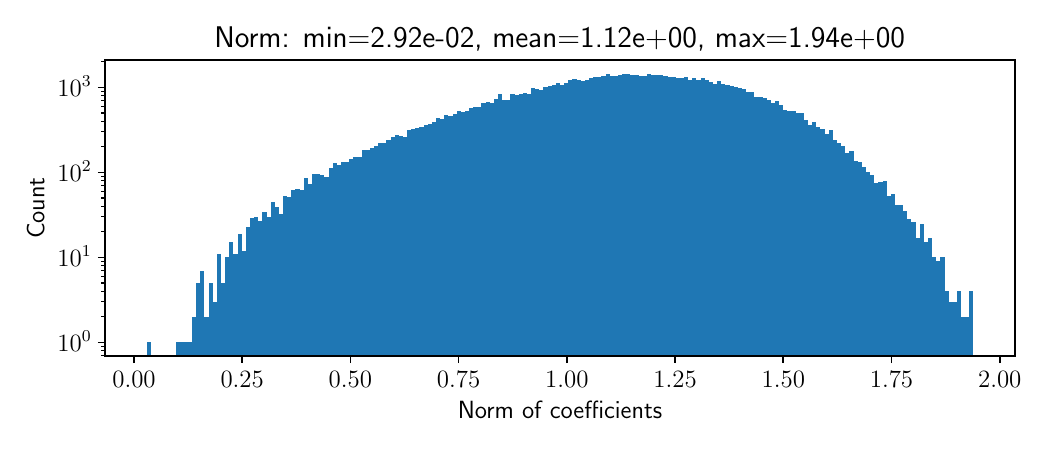}}
    \caption{\textbf{Histogram of $L^2$ norms of 4-dimensional coefficients in test set for the FTE.}
    }
    \label{fig: fig_hist_NormOfCoeffs_TestSet_FTE_de4_w1024_it500k_stat_FTE}
    \end{center}
    \vskip 0.1in
\end{figure}
\begin{figure}[htbp]
\vskip 0.1in
    \begin{center}
    \centerline{\includegraphics[width=\columnwidth]{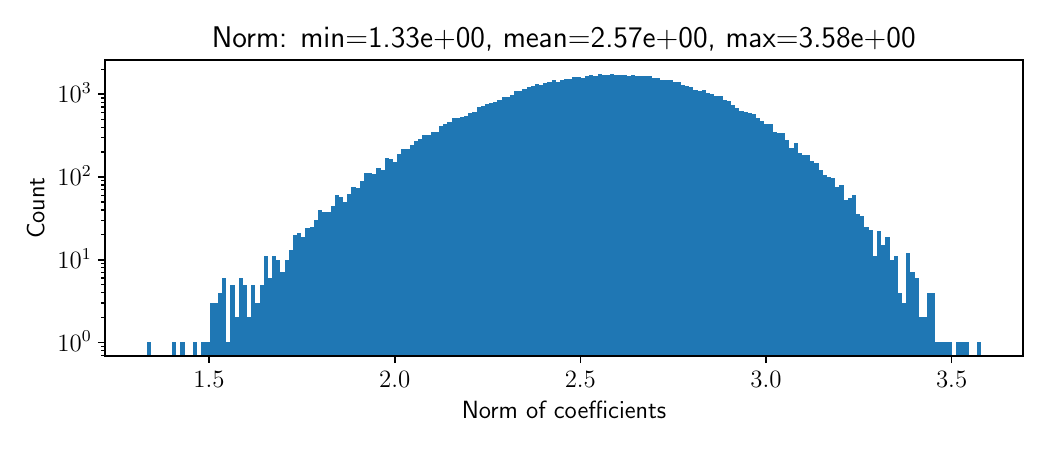}}
    \caption{\textbf{Histogram of $L^2$ norms of 20-dimensional coefficients in test set for the FTE.}
    }
    \label{fig: fig_hist_NormOfCoeffs_TestSet_FTE_deg20_w1024_it500k_stat_FTE}
    \end{center}
    \vskip 0.1in
\end{figure}
\begin{figure}[htbp]
\vskip 0.1in
    \begin{center}
    \centerline{\includegraphics[width=\columnwidth]{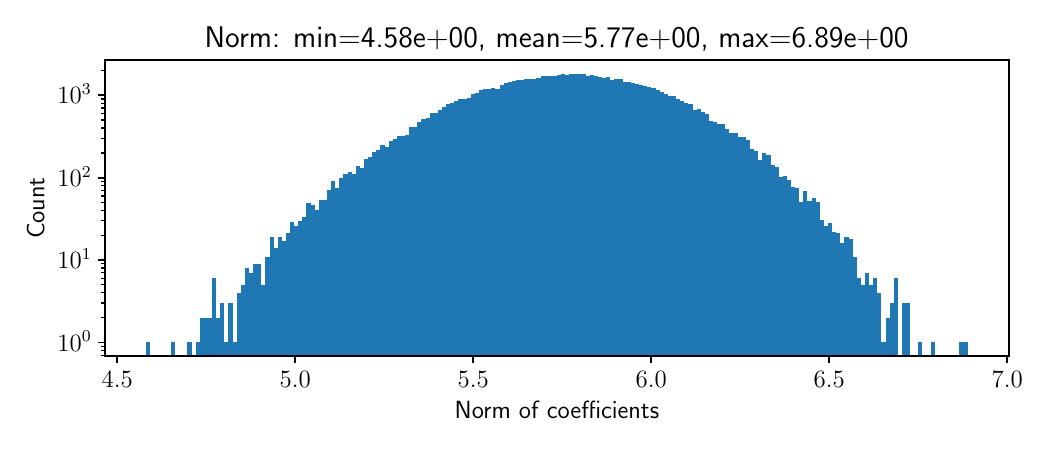}}
    \caption{\textbf{Histogram of $L^2$ norms of 100-dimensional coefficients in test set for the FTE.}
    }
    \label{fig: fig_hist_NormOfCoeffs_TestSet_FTE_deg100_w1024_it500k_stat_FTE}
    \end{center}
    \vskip 0.1in
\end{figure}

\begin{figure}[htbp]
\vskip 0.1in
    \begin{center}
    \centerline{\includegraphics[width=\columnwidth]{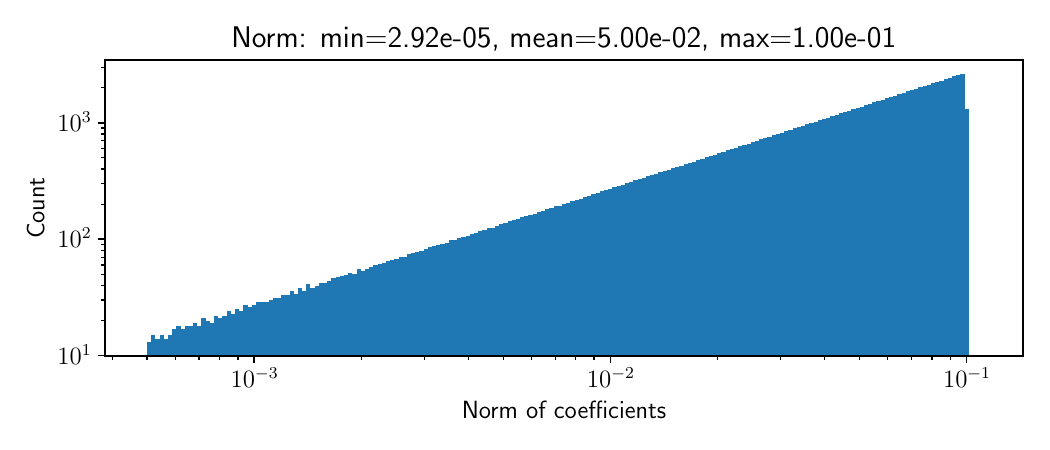}}
    \caption{\textbf{Histogram of $L^2$ norms of 4-dimensional coefficients in test set for the BHE.}
    }
    \label{fig: fig_hist_NormOfCoeffs_TestSet_BHE_init1_deg2_sig1e1_stat_BHE}
    \end{center}
    \vskip 0.1in
\end{figure}
\begin{figure}[htbp]
\vskip 0.1in
    \begin{center}
    \centerline{\includegraphics[width=\columnwidth]{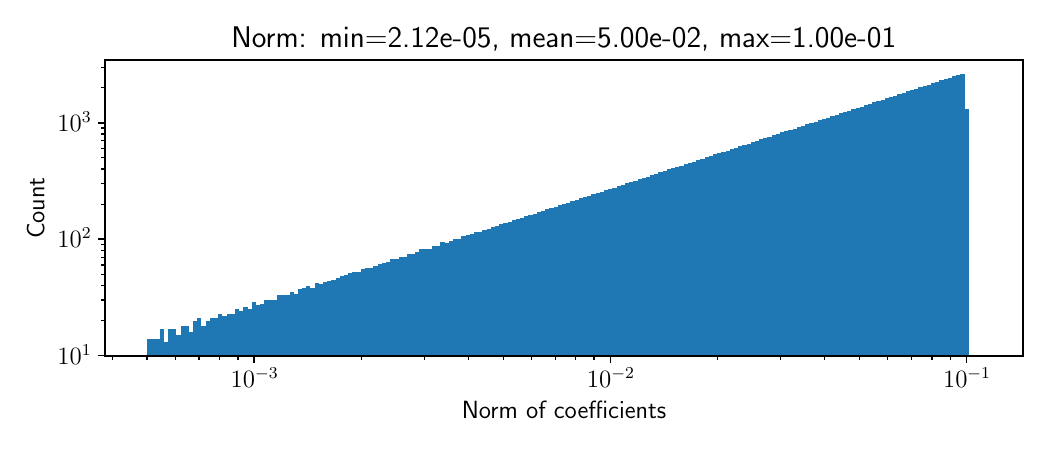}}
    \caption{\textbf{Histogram of $L^2$ norms of 20-dimensional coefficients in test set for the BHE.}
    }
    \label{fig: fig_hist_NormOfCoeffs_TestSet_BHE_init1_deg10_sig1e1_stat_BHE}
    \end{center}
    \vskip 0.1in
\end{figure}
\begin{figure}[htbp]
\vskip 0.1in
    \begin{center}
    \centerline{\includegraphics[width=\columnwidth]{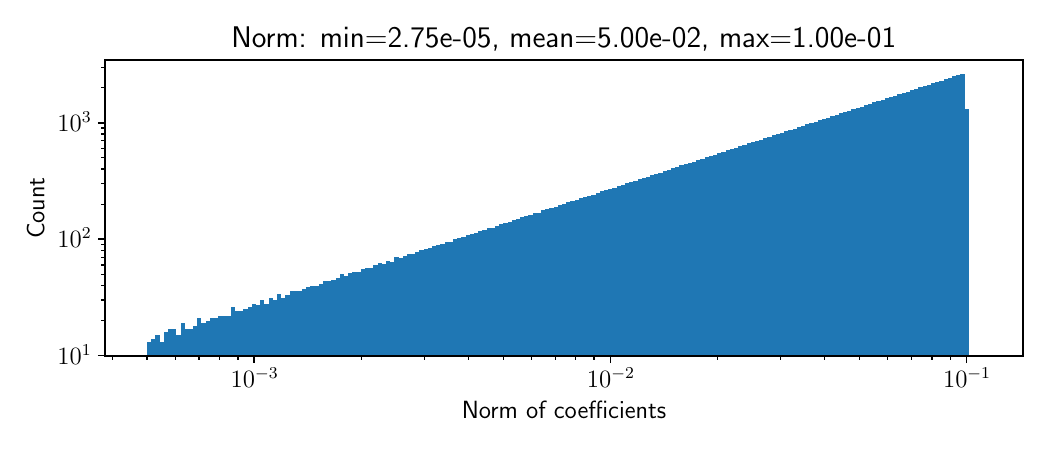}}
    \caption{\textbf{Histogram of $L^2$ norms of 100-dimensional coefficients in test set for the BHE.}
    }
    \label{fig: fig_hist_NormOfCoeffs_TestSet_BHE_init1_deg50_sig1e1_stat_BHE}
    \end{center}
    \vskip 0.1in
\end{figure}

\clearpage
\subsection{Comparison with Baseline: Advection-reaction Functional Equation (ARE)} 
\label{app: Advection-reaction Equation}
We consider the advection-reaction functional equation (ARE).
\begin{equation}
    \frac{\partial F([\theta], t)}{\partial t} = - \int_0^{2\pi} dx \theta(x) \frac{\partial}{\partial x}(\frac{\delta F([\theta], t)}{\delta \theta  (x)})
\end{equation}
We follow the experimental setting described in \cite{venturi2018numerical}.

In \cite{venturi2018numerical}, the reported time-dependent relative error of the CP-ALS and HT fluctuates from $\sim 10^{-8}$ to $\sim 10^{-1}$, and the temporally-averaged relative error is estimated to be $\lesssim 10^{-2}$ (the bottom two panels in Fig.~38 in \cite{venturi2018numerical}), where we roughly approximated the time-dependent error curves as an exponential function of the form $y = 10^{-6+10t}$. 

Our preliminary results are shown in Tab.~\ref{tab: ARE result}. 
We observe that the learned functional is almost constant in time, a failure mode of PINNs, causing large errors. 
We anticipate the result can improve with further hyperparameter optimization \cite{krishnapriyan2021characterizing}.

\begin{table}[htbp] 
\caption{\textbf{Relative and absolute errors.} The models are trained on the ARE. 
The error bars are the standard deviation over 10 training runs with different random seeds.
Relative and absolute error represent the averaged relative and absolute error over all collocation points and all runs.
Best relative and absolute error represent the relative and absolute error of the best collocation point averaged over all runs.
The worst relative and absolute errors represent the averaged relative and absolute errors of the worst collocation point averaged over all runs. 
} 
\label{tab: ARE result}
    \centering
    \begin{center}
        \begin{subtable}{0.7\textwidth}
        \begin{small}
        \begin{sc}
            \begin{tabular}{@{}rcc@{}} 
                \toprule
                Degree & Relative error & Absolute error \\ 
                \midrule
                6  & $ (0.954940 \pm 2.86337) \times 10^{-1}$ & $(6.17040 \pm 1.83985) \times 10^{-2}$ \\ 
            \end{tabular}
        \end{sc}
        \end{small}
        \end{subtable}
        \begin{subtable}{0.7\textwidth}
        \begin{small}
        \begin{sc}
            \begin{tabular}{@{}rcc@{}} 
                \toprule
                Degree & Best relative error & Best absolute error \\ 
                \midrule
                6  & $ (0.12016 \pm 2.49220) \times 10^{-1}$ & $ (0.58046 \pm 1.73938) \times 10^{-2}$ \\ 
            \end{tabular}
        \end{sc}
        \end{small}
        \end{subtable}
        \begin{subtable}{0.7\textwidth}
        \begin{small}
        \begin{sc}
            \begin{tabular}{@{}rcc@{}} 
                \toprule
                Degree & Worst relative error & Worst absolute error \\ 
                \midrule
                6  & $(9.29019 \pm 2.85457) \times 10^{-1}$ & $ (8.49163 \pm 2.0718) \times 10^{-2}$ \\ 
            \end{tabular}
        \end{sc}
        \end{small}
        \end{subtable}
    \end{center}
\end{table}

\clearpage
\subsection{Relative and Absolute Errors When $W([0], t) = 0$ Is Included In Loss} 
\label{app: Relative and Absolute Error When $W([0], t) = 0$ Is Included In Loss}
We show the histograms of the relative and absolute errors for the models trained on the BHE with the loss term corresponding to the identity $W([0], t) = 0$.
They are almost the same as those for the models trained without $W([0], t) = 0$, but the errors of the first-order derivative are $10^{-1}$ times smaller, as shown in Sec.~\ref{sec: Experiment}.

\begin{figure}[htbp]
\vskip 0.1in
    \begin{center}
    \centerline{\includegraphics[width=\columnwidth]{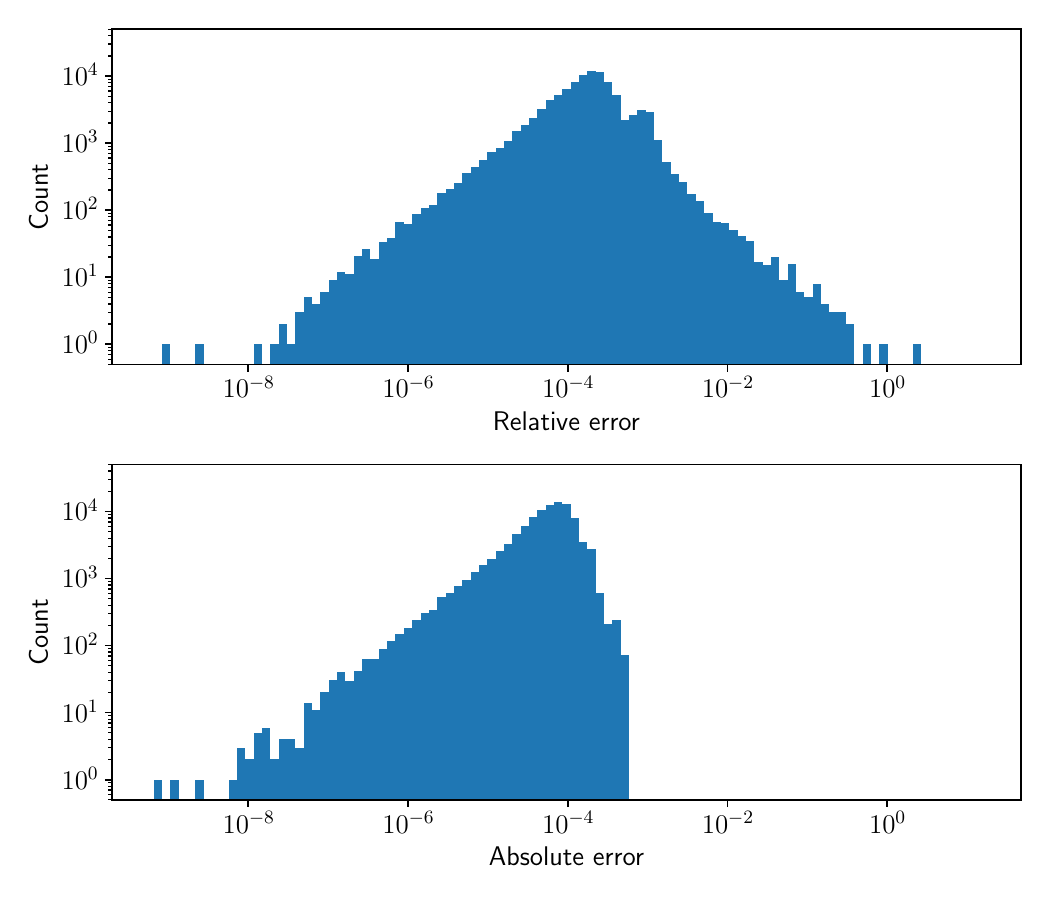}}
    \caption{\textbf{Histogram of relative and absolute error.} The model is trained on the BHE of degree 4 with the delta initial condition. A single random seed is used for the training.
    }
    \label{fig: fig_hist_RelAbs_BHE_init0_deg2_sig1e1_bcW0_it500k_stat_model0_BHE}
    \end{center}
    \vskip 0.1in
\end{figure}

\begin{figure}[htbp]
\vskip 0.1in
    \begin{center}
    \centerline{\includegraphics[width=\columnwidth]{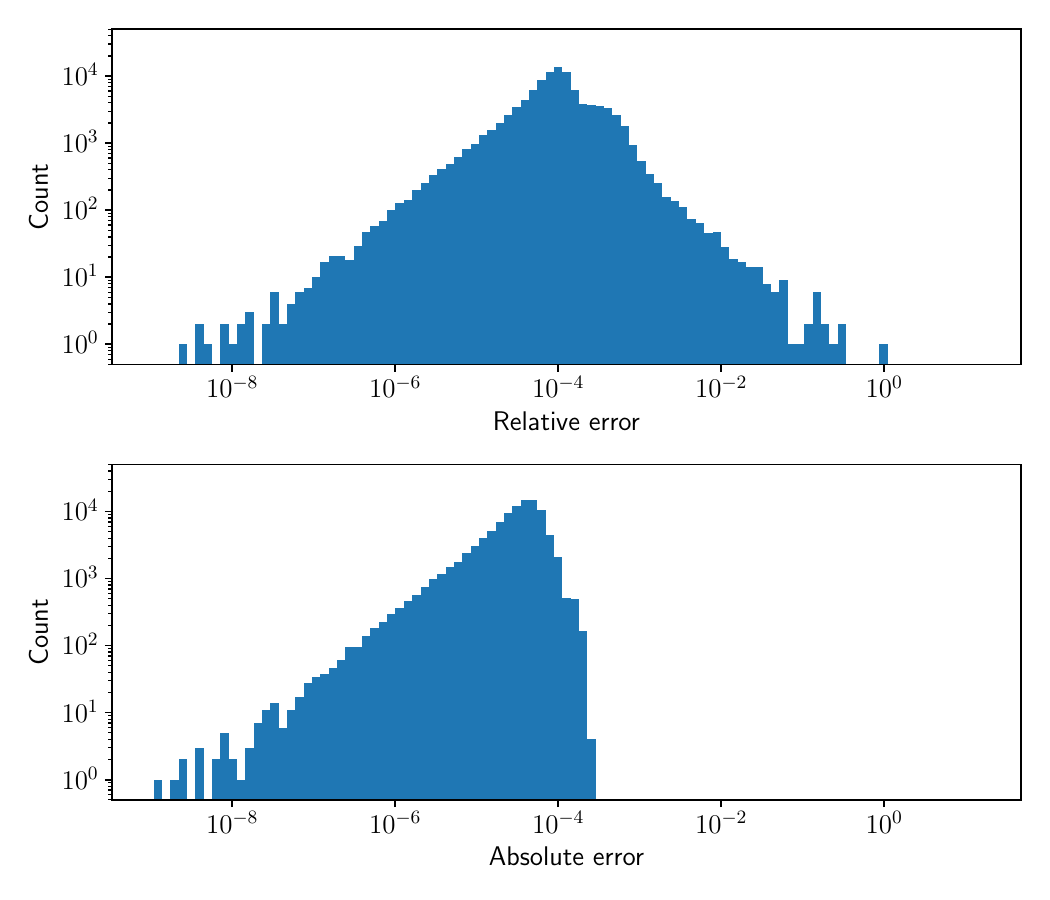}}
    \caption{\textbf{Histogram of relative and absolute error.} The model is trained on the BHE of degree 100 with the delta initial condition. A single random seed is used for the training.
    }
    \label{fig: fig_hist_RelAbs_BHE_init0_deg50_sig1e1_bcW0_it500k_stat_model0_BHE}
    \end{center}
    \vskip 0.1in
\end{figure}

\begin{figure}[htbp]
\vskip 0.1in
    \begin{center}
    \centerline{\includegraphics[width=\columnwidth]{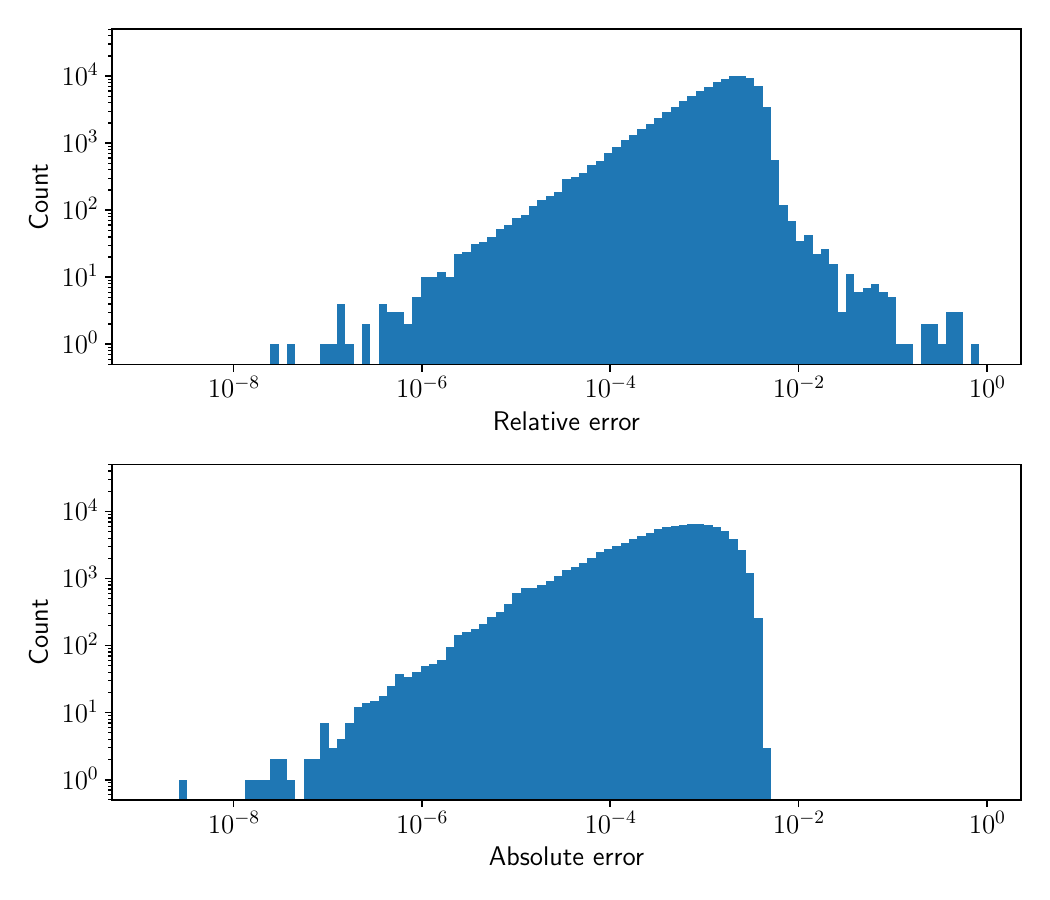}}
    \caption{\textbf{Histogram of relative and absolute error.} The model is trained on the BHE of degree 4 with the constant initial condition. A single random seed is used for the training.
    }
    \label{fig: fig_hist_RelAbs_BHE_init2_deg2_sig1e1_bcW0_it500k_stat_model0_BHE}
    \end{center}
    \vskip 0.1in
\end{figure}

\begin{figure}[htbp]
\vskip 0.1in
    \begin{center}
    \centerline{\includegraphics[width=\columnwidth]{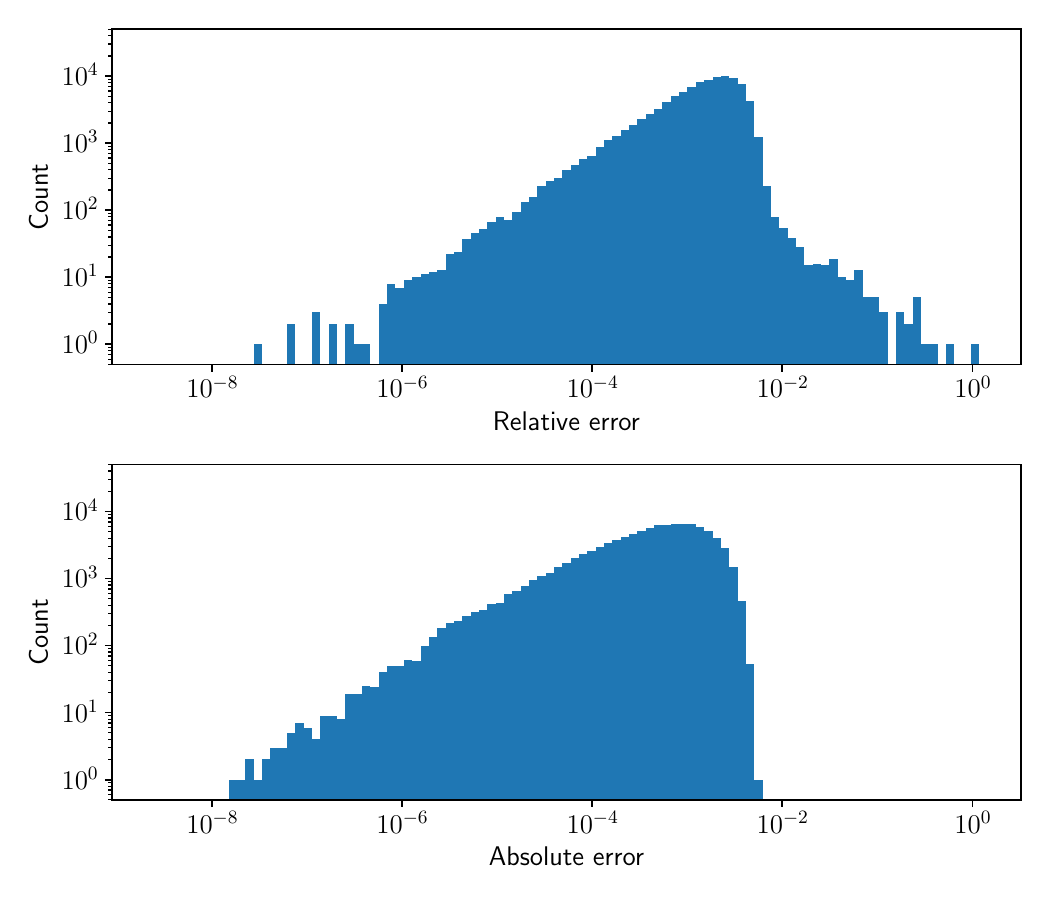}}
    \caption{\textbf{Histogram of relative and absolute error.} The model is trained on the BHE of degree 100 with the constant initial condition. A single random seed is used for the training.
    }
    \label{fig: fig_hist_RelAbs_BHE_init2_deg50_sig1e1_bcW0_it500k_stat_model0_BHE}
    \end{center}
    \vskip 0.1in
\end{figure}

\clearpage
\subsection{Cross-degree Evaluation} \label{app: Cross-degree Evaluation}
In this section, we attempt to empirically investigate the theoretical convergence of the cylindrical approximation. 
The convergence is illustrated in Fig.~\ref{fig: fitting}, where PINN's training is not included, while our focus here is the convergence of the solutions estimated via PINN's training.
According to the convergence theorem of the cylindrical approximation and the universal approximability of PINNs, the models trained on $D_m = \mathrm{span}\{\phi_0,\ldots,\phi_{m-1}\}$ with a large $m$ have a high capability of expression, and they can cover $D_n$ with $n \leq m$. 
On the other hand, the models trained on $D_m$ with a small $m$ have a high capability of expression only for small $m$s, and they \textit{cannot} cover $D_n$ for $n \geq m$.
In theory, this phenomenon can be observed by performing a \textit{cross-degree evaluation}, in which we use training data sampled from $D_m$ and then evaluate the trained model on $D_{m^\prime}$, where $m$ may not equal to $m^\prime$.

The results are shown in Tabs.~\ref{tab: AREmeanFTE_deg4_20_100_w1024_it500k.tex}--\ref{tab: AAEmeanBHE_init2_deg2_10_50_sig1e1.tex} below.
Contrary to the aforementioned intuition, most of the cross-degree evaluations seem to contradict the convergence theorem of FDE solutions under the cylindrical approximation. 
Possible reasons are:
\begin{enumerate}
    \item This is because of the curse of dimensionality, i.e., the optimization error of PINNs in high dimensions.
        Higher-dimensional PDEs require a much larger number of iterations to train PINNs; however, our experiments are performed with a fixed number of iterations (500,000 and 300,000 for the FTE and the BHE, respectively).
        Changing the number of iterations depending on the degree and applying other techniques for training PINNs \cite{zeng2022competitive, hu2023tackling, wang2023multi, hu2023bias, daw2023mitigating, yao2023multadam, gao2023failure} would decrease the optimization error, and we can focus on the approximation error and thus will observe the decay of error w.r.t. increasing $m$.
    \item In addition, note that FDEs in our experiments, except for the FTE with the nonlinear initialization, are dominated by $a_0$ and/or $a_1$ only. Therefore, they show a tiny effect on relative and absolute errors by definition. Nevertheless, the FTE with the nonlinear initialization (Tabs.~\ref{tab: AREmeanFTE_deg4_10_20_w1024_it500k_ICspectrum15.tex} and \ref{tab: AAEmeanFTE_deg4_10_20_w1024_it500k_ICspectrum15.tex}) does not show the theoretical convergence either, which means the optimization error of PINNs, rather than the cylindrical approximation error, dominates the errors anyways.
    \item One can see a strong dependence of the errors on the experimental setups for PINNs.
        The absolute errors of the FTE under the nonlinear initial condition are reduced by a factor of $2$ by simply changing the width of the PINN and the number of training iterations (1024 to 2048 and $5\times10^{5}$ to $8\times10^{5}$, respectively): 
        from
        \begin{itemize}
            \item $(2.46499 \pm 0.26492) \times 10^{-3}$ for $m=4$,  
            \item $(20.3384 \pm 1.0449)  \times 10^{-3}$ for $m=10$, and 
            \item $(98.8854 \pm 2.7068)  \times 10^{-3}$ for $m=20$,
        \end{itemize}
        to
        \begin{itemize}
            \item $(1.43591 \pm 0.1832) \times 10^{-3}$ for $m=4$, 
            \item $(11.0609 \pm 0.3671) \times 10^{-3}$ for $m=10$, and 
            \item $(55.2715 \pm 2.3179) \times 10^{-3}$ for $m=20$,
        \end{itemize}
        where the error bars are the standard deviation over 10 runs. Note that the hyperparameters are different from those used in the main text.
        Moreover, as shown in Fig.~\ref{fig: fig_edit_plot_FirstMomentRel_wERRBARinit050_sig1e1_WandWObcW0_stat_model}, simply adding a regularization term $\|W([0], t)\|$ to the loss function reduces relative errors by an order of magnitude. 
\end{enumerate}

Therefore, to observe the convergence of the cylindrical approximation after PINN's training, careful hyperparameter tuning and optimization are needed. 
Note that optimizing PINNs for high-dimensional PDEs is an active field of research that has been developing (see Apps.~\ref{app: Supplementary Related Work} and \ref{app: Supplementary Discussion}) and is of independent interest.
 
\clearpage
\subsubsection{Relative and Absolute Errors of Functional Transport Equation}
\begin{table}[htbp] 
\vskip 0.1in
\caption{\textbf{Mean relative error $\times 10^{3}$ (functional transport equation with linear initialization) on test set.} The model is trained on the functional transport equation with the linear initialization condition. The error bars represent the standard error of mean with sample size 10 corresponding to different training seeds.} \label{tab: AREmeanFTE_deg4_20_100_w1024_it500k.tex}
    \vskip 0.1in
    \begin{center}
    \begin{small}
    \begin{sc}
        \begin{tabular}{@{}rlll@{}} \toprule
            & \multicolumn{3}{c}{Training set degree} \\ \cmidrule(l){2-4}
            Test set degree&4&20&100 \\ \midrule
            4&  $1.26821	\pm 0.09936$	&	$2.56697	\pm0.20431$ &$	14.6402	 \pm2.1739	$ \\ 
            20& $0.559035	\pm0.026630$ &	$2.01716	\pm0.06875$ &$	5.82534	 \pm 0.49499$\\ 
            100&$0.483585	\pm0.010872$ &	$1.88511	\pm0.05292$ &$	6.24740	 \pm0.10591 $\\ 
        \end{tabular}
    \end{sc}
    \end{small}
    \end{center}
    \vskip -0.1in
\end{table}
\begin{table}[htbp] 
\vskip 0.1in
\caption{\textbf{Mean absolute error $\times 10^{4}$ (functional transport equation with linear initialization)  on test set.} The model is trained on the functional transport equation with the linear initialization condition. The error bars represent the standard error of mean with sample size 10 corresponding to different training seeds.} \label{tab: AAEmeanFTE_deg4_20_100_w1024_it500k.tex}
    \vskip 0.1in
    \begin{center}
    \begin{small}
    \begin{sc}
        \begin{tabular}{@{}rlll@{}} \toprule
            & \multicolumn{3}{c}{Training set degree} \\ \cmidrule(l){2-4}
            Test set degree&4&20&100 \\ \midrule
            4&$1.32203	\pm 0.13933$&$	3.06281	\pm 0.28009$&$	15.2530	 \pm 0.7296$\\ 
            20&$1.34579	\pm 0.14329$&$	2.29632	\pm 0.05205$&$	13.4665	 \pm 0.7117$\\ 
            100&$1.32709	\pm 0.14107$&$	2.27854	\pm 0.05028$&$	12.3312	 \pm 0.5986$\\ 
        \end{tabular}
    \end{sc}
    \end{small}
    \end{center}
    \vskip -0.1in
\end{table}

\begin{table}[htbp] 
\vskip 0.1in
\caption{\textbf{Mean relative error $\times 10^{3}$ (functional transport equation with nonlinear initialization) on test set.} The model is trained on the functional transport equation with the nonlinear initialization condition. The error bars represent the standard error of mean with sample size 10 corresponding to different training seeds.}\label{tab: AREmeanFTE_deg4_10_20_w1024_it500k_ICspectrum15.tex}
    \vskip 0.1in
    \begin{center}
    \begin{small}
    \begin{sc}
        \begin{tabular}{@{}rlll@{}} \toprule
            & \multicolumn{3}{c}{Training set degree} \\ \cmidrule(l){2-4}
            Test set degree&4&10&20 \\ \midrule
             4&$1.09131  \pm 0.02945$&$	   2.06847 \pm 0.06531$&$	   7.51237	\pm 0.25359 $\\ 
            10&$1.14029	\pm 0.03166$&$	3.59902	\pm 0.05139$&$	8.61269	 \pm 0.10351$\\ 
            20&$1.11229	\pm 0.02789$&$	3.62669	\pm 0.06044$&$	11.7414	 \pm 0.1112$\\ 
        \end{tabular}
    \end{sc}
    \end{small}
    \end{center}
    \vskip -0.1in
\end{table}
\begin{table}[htbp] 
\vskip 0.1in
\caption{\textbf{Mean absolute error $\times 10^{3}$ (functional transport equation with nonlinear initialization) on test set.} The model is trained on the functional transport equation with the nonlinear initialization condition. Note that the hyperparameters are different from those used in the main text. The error bars represent the standard error of mean with sample size 10 corresponding to different training seeds.} \label{tab: AAEmeanFTE_deg4_10_20_w1024_it500k_ICspectrum15.tex}
    \vskip 0.1in
    \begin{center}
    \begin{small}
    \begin{sc}
        \begin{tabular}{@{}rlll@{}} \toprule
            & \multicolumn{3}{c}{Training set degree} \\ \cmidrule(l){2-4}
            Test set degree&4&10&20 \\ \midrule
            4&$2.46499	\pm 0.08377$&$	9.00259	\pm 0.21308$&$47.5449	 \pm 0.6410$\\ 
            10&$2.46885	\pm 0.08409$&$	20.3384	\pm 0.3304$&$	69.6191	 \pm 0.7526$\\ 
            20&$2.48688	\pm 0.08436$&$	20.5866	\pm 0.3346$&$	98.8854	 \pm 0.8560$\\ 
        \end{tabular}
    \end{sc}
    \end{small}
    \end{center}
    \vskip -0.1in
\end{table}

\clearpage
\subsubsection{Relative and Absolute Errors of Burgers-Hopf Equation}
\begin{table}[htbp] 
\vskip 0.1in
\caption{\textbf{Mean relative error $\times 10^{4}$ (BHE with delta initialization) on test set.} The model is trained on the BHE with the delta initialization condition. The error bars represent the standard error of mean with sample size 10 corresponding to different training seeds.} \label{tab: AREmeanBHE_init0_deg2_10_50_sig1e1.tex}
    \vskip 0.1in
    \begin{center}
    \begin{small}
    \begin{sc}
        \begin{tabular}{@{}rlll@{}} \toprule
            & \multicolumn{3}{c}{Training set degree} \\ \cmidrule(l){2-4}
            Test set degree&4&20&100 \\ \midrule
            4&$2.93905	\pm 0.05503$&$	2.97886	\pm 0.03187$&$	2.99835	 \pm 0.03851$\\ 
            20&$2.24827	\pm 0.07876$&$	2.20842	\pm 0.09022$&$	2.23898	 \pm 0.05530$\\ 
            100&$4.11689	\pm 0.34740$&$	2.98177	\pm 0.20264$&$	2.41667	 \pm 0.07989$\\ 
        \end{tabular}
    \end{sc}
    \end{small}
    \end{center}
    \vskip -0.1in
\end{table}
\begin{table}[htbp] 
\vskip 0.1in
\caption{\textbf{Mean absolute error $\times 10^{5}$ (BHE with delta initialization) on test set.} The model is trained on the BHE with the delta initialization condition. The error bars represent the standard error of mean with sample size 10 corresponding to different training seeds.} \label{tab: AAEmeanBHE_init0_deg2_10_50_sig1e1.tex}
    \vskip 0.1in
    \begin{center}
    \begin{small}
    \begin{sc}
        \begin{tabular}{@{}rlll@{}} \toprule
            & \multicolumn{3}{c}{Training set degree} \\ \cmidrule(l){2-4}
            Test set degree&4&20&100 \\ \midrule
            4&$1.66451	\pm0.15036$&$	1.35700	\pm0.03771$&$	1.62045 \pm	0.02767$\\ 
            20&$1.66216	\pm0.15005$&$	1.34640	\pm0.03828$&$	1.60869	\pm 0.02892$\\ 
            100&$1.66510\pm0.15030$&$	1.34980	\pm0.03779$&$	1.60980	\pm 0.02831$\\ 
        \end{tabular}
    \end{sc}
    \end{small}
    \end{center}
    \vskip -0.1in
\end{table}

\begin{table}[htbp] 
\vskip 0.1in
\caption{\textbf{Mean relative error $\times 10^{5}$ (BHE with constant initialization) on test set.} The model is trained on the BHE with the constant initialization condition. The error bars represent the standard error of mean with sample size 10 corresponding to different training seeds.} \label{tab: AREmeanBHE_init1_deg2_10_50_sig1e1.tex}
    \vskip 0.1in
    \begin{center}
    \begin{small}
    \begin{sc}
        \begin{tabular}{@{}rlll@{}} \toprule
            & \multicolumn{3}{c}{Training set degree} \\ \cmidrule(l){2-4}
            Test set degree&4&20&100 \\ \midrule
            4&$12.1782	\pm2.70298 $&$	8.32134	\pm 0.72992$&$	8.44530	 \pm 1.4148$\\ 
            20&$9.07441	\pm1.50640 $&$	6.14352	\pm 0.40680$&$	5.99939	 \pm 0.68357$\\ 
            100&$8.26214	\pm1.17503 $&$	5.75086	\pm 0.35933$&$	5.50375	 \pm 0.55500 $\\ 
        \end{tabular}
    \end{sc}
    \end{small}
    \end{center}
    \vskip -0.1in
\end{table}
\begin{table}[htbp] 
\vskip 0.1in
\caption{\textbf{Mean absolute error $\times 10^{5}$ (BHE with constant initialization) on test set.} The model is trained on the BHE with the constant initialization condition. The error bars represent the standard error of mean with sample size 10 corresponding to different training seeds.} \label{tab: AAEmeanBHE_init1_deg2_10_50_sig1e1.tex}
    \vskip 0.1in
    \begin{center}
    \begin{small}
    \begin{sc}
        \begin{tabular}{@{}rlll@{}} \toprule
            & \multicolumn{3}{c}{Training set degree} \\ \cmidrule(l){2-4}
            Test set degree&4&20&100 \\ \midrule
            4&$ 1.62849	\pm 0.21154$&$	1.12380	\pm 0.10234$&$	1.05562	 \pm 0.05272$\\ 
            20&$1.62892	\pm 0.21162$&$	1.12371	\pm 0.10235$&$	1.05556	 \pm 0.05272$\\ 
            100&$1.62931	\pm 0.21148$&$	1.12369	\pm 0.10231$&$	1.05558	 \pm 0.05273$\\ 
        \end{tabular}
    \end{sc}
    \end{small}
    \end{center}
    \vskip -0.1in
\end{table}

\begin{table}[htbp] 
\vskip 0.1in
\caption{\textbf{Mean absolute error $\times 10^{3}$ (BHE with moderate initialization) on test set.} The model is trained on the BHE with the moderate initialization condition. The error bars represent the standard error of mean with sample size 10 corresponding to different training seeds.} \label{tab: AREmeanBHE_init2_deg2_10_50_sig1e1.tex}
    \vskip 0.1in
    \begin{center}
    \begin{small}
    \begin{sc}
        \begin{tabular}{@{}rlll@{}} \toprule
            & \multicolumn{3}{c}{Training set degree} \\ \cmidrule(l){2-4}
            Test set degree&4&20&100 \\ \midrule
            4&$1.32699	\pm 0.06872$&$	1.60725	\pm0.03521 $&$	1.84851	 \pm 0.01689$\\ 
            20&$1.24102	\pm0.04793 $&$	1.63232	\pm0.02236 $&$	1.92692	 \pm0.00862 $\\ 
            100&$1.18515	\pm0.04836 $&$	1.55581	\pm 0.02152$&$	1.82844	 \pm 0.00684$\\ 
        \end{tabular}
    \end{sc}
    \end{small}
    \end{center}
    \vskip -0.1in
\end{table}
\begin{table}[htbp] 
\vskip 0.1in
\caption{\textbf{Mean absolute error $\times 10^{4}$ (BHE with moderate initialization) on test set.} The model is trained on the BHE with the moderate initialization condition. The error bars represent the standard error of mean with sample size 10 corresponding to different training seeds.} \label{tab: AAEmeanBHE_init2_deg2_10_50_sig1e1.tex}
    \vskip 0.1in
    \begin{center}
    \begin{small}
    \begin{sc}
        \begin{tabular}{@{}rlll@{}} \toprule
            & \multicolumn{3}{c}{Training set degree} \\ \cmidrule(l){2-4}
            Test set degree&4&20&100 \\ \midrule
            4&$3.91187	\pm 0.10867$&$	5.49339	\pm 0.03494$&$	6.73537	 \pm 0.01891$\\ 
            20&$3.93392	\pm 0.11011$&$	5.63933	\pm 0.03427$&$	6.87334	 \pm 0.01884$\\ 
            100&$3.90431	\pm 0.10875$&$	5.60266	\pm 0.03403$&$	6.82414	 \pm 0.01860$\\ 
        \end{tabular}
    \end{sc}
    \end{small}
    \end{center}
    \vskip -0.1in
\end{table}

\clearpage
\subsection{Loss Curves and Learning Rate Scheduling} 
\label{app: Loss Curves}
Loss curves and learning rates are provided in Figs.~\ref{fig: LC1}--\ref{fig: LC2}.

\begin{figure}[htbp]
    \centering
    \includegraphics{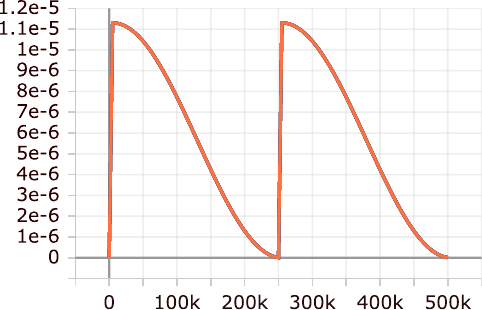}
    \caption{\textbf{Learning rate scheduling.} FTE with nonlinear initial condition with degree $1000$.}
    \label{fig: LC1}
\end{figure}

\begin{figure}[htbp]
    \centering
    \includegraphics{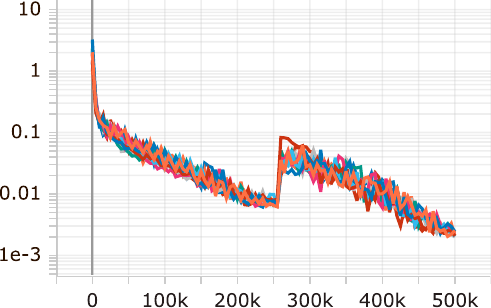}
    \caption{\textbf{Training loss.} FTE with nonlinear initial condition with degree $1000$.}
\end{figure}

\begin{figure}[htbp]
    \centering
    \includegraphics{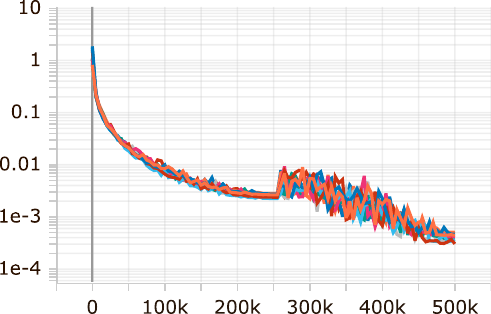}
    \caption{\textbf{Validation loss.} FTE with nonlinear initial condition with degree $1000$.}
\end{figure}

\begin{figure}[htbp]
    \centering
    \includegraphics{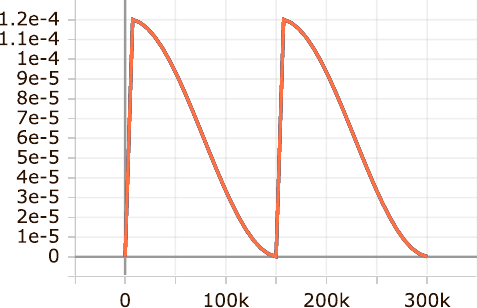}
    \caption{\textbf{Learning rate scheduling.} BHE with delta initial condition with degree $100$.}
\end{figure}

\begin{figure}[htbp]
    \centering
    \includegraphics{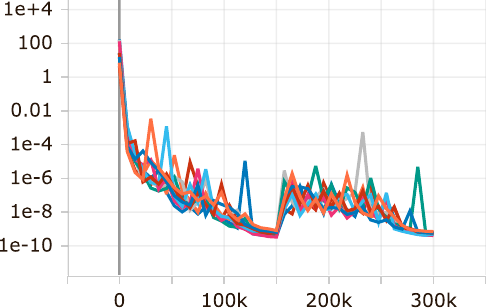}
    \caption{\textbf{Training loss.} BHE with delta initial condition with degree $100$.}
\end{figure}

\begin{figure}[htbp]
    \centering
    \includegraphics{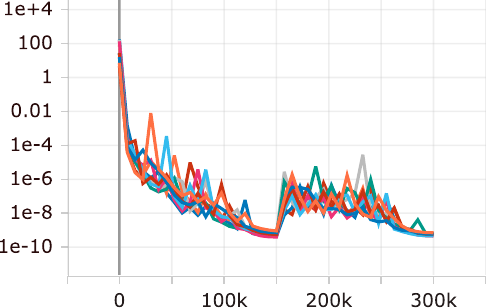}
        \caption{\textbf{Validation loss.} BHE with delta initial condition with degree $100$.}
\end{figure}

\begin{figure}[htbp]
    \centering
    \includegraphics{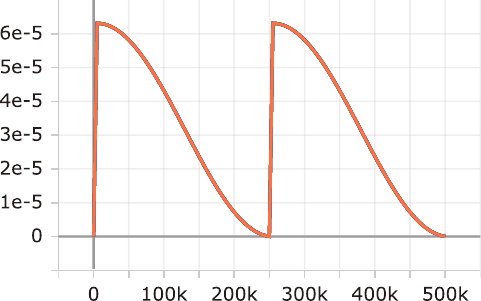}
    \caption{\textbf{Learning rate scheduling.} ARE with degree $6$.}
\end{figure}

\begin{figure}[htbp]
    \centering
    \includegraphics{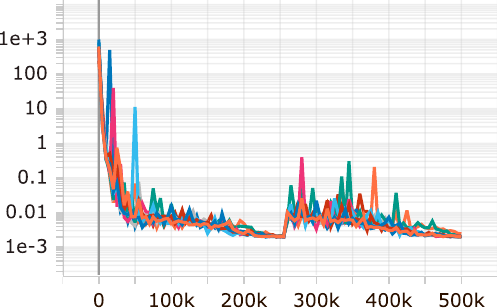}
    \caption{\textbf{Training loss.} ARE with degree $6$.}
\end{figure}

\begin{figure}[htbp]
    \centering
    \includegraphics{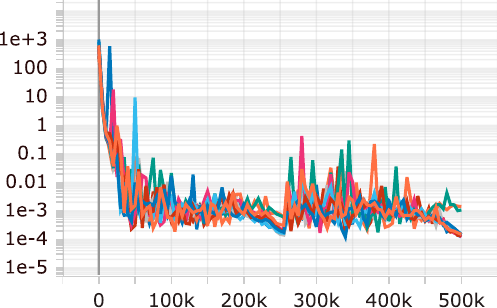}
        \caption{\textbf{Validation loss.} ARE with degree $6$.}
    \label{fig: LC2}
\end{figure}

\clearpage
\subsection{Additional Figures: Functional Transport Equation} 
\label{app: Figures with Error Bars: Functional Transport Equation}
{\color{blue} The contents are omitted due to arXiv's storage limit. Please refer to our full paper at OpenReview (NeurIPS 2024) or \url{https://github.com/TaikiMiyagawa/FunctionalPINN}.}

\clearpage
\subsection{Additional Figures: Burgers-Hopf Equation}
{\color{blue} The contents are omitted due to arXiv's storage limit. Please refer to our full paper at OpenReview (NeurIPS 2024) or \url{https://github.com/TaikiMiyagawa/FunctionalPINN}.}

\clearpage
\section*{NeurIPS Paper Checklist}

\begin{enumerate}

\item {\bf Claims}
    \item[] Question: Do the main claims made in the abstract and introduction accurately reflect the paper's contributions and scope?
    \item[] Answer: \answerYes{} 
    \item[] Justification: Our contribution is summarized at the end of Introduction. How much the results can be expected to generalize to other settings is also discussed in App.~\ref{app: Supplementary Discussion}.
    \item[] Guidelines:
    \begin{itemize}
        \item The answer NA means that the abstract and introduction do not include the claims made in the paper.
        \item The abstract and/or introduction should clearly state the claims made, including the contributions made in the paper and important assumptions and limitations. A No or NA answer to this question will not be perceived well by the reviewers. 
        \item The claims made should match theoretical and experimental results, and reflect how much the results can be expected to generalize to other settings. 
        \item It is fine to include aspirational goals as motivation as long as it is clear that these goals are not attained by the paper. 
    \end{itemize}

\item {\bf Limitations}
    \item[] Question: Does the paper discuss the limitations of the work performed by the authors?
    \item[] Answer: \answerYes{} 
    \item[] Justification: Limitations are discussed in Sec.~\ref{sec: Conclusion and Limitations} and App.~\ref{app: Limitations}.
    \item[] Guidelines:
    \begin{itemize}
        \item The answer NA means that the paper has no limitation while the answer No means that the paper has limitations, but those are not discussed in the paper. 
        \item The authors are encouraged to create a separate "Limitations" section in their paper.
        \item The paper should point out any strong assumptions and how robust the results are to violations of these assumptions (e.g., independence assumptions, noiseless settings, model well-specification, asymptotic approximations only holding locally). The authors should reflect on how these assumptions might be violated in practice and what the implications would be.
        \item The authors should reflect on the scope of the claims made, e.g., if the approach was only tested on a few datasets or with a few runs. In general, empirical results often depend on implicit assumptions, which should be articulated.
        \item The authors should reflect on the factors that influence the performance of the approach. For example, a facial recognition algorithm may perform poorly when image resolution is low or images are taken in low lighting. Or a speech-to-text system might not be used reliably to provide closed captions for online lectures because it fails to handle technical jargon.
        \item The authors should discuss the computational efficiency of the proposed algorithms and how they scale with dataset size.
        \item If applicable, the authors should discuss possible limitations of their approach to address problems of privacy and fairness.
        \item While the authors might fear that complete honesty about limitations might be used by reviewers as grounds for rejection, a worse outcome might be that reviewers discover limitations that aren't acknowledged in the paper. The authors should use their best judgment and recognize that individual actions in favor of transparency play an important role in developing norms that preserve the integrity of the community. Reviewers will be specifically instructed to not penalize honesty concerning limitations.
    \end{itemize}

\item {\bf Theory Assumptions and Proofs}
    \item[] Question: For each theoretical result, does the paper provide the full set of assumptions and a complete (and correct) proof?
    \item[] Answer: \answerYes{} 
    \item[] Justification: The full set of assumptions and complete proofs for each theoretical result is given in Sec.~\ref{sec: Theory: Cylindrical Approximation}, Apps.~\ref{app: Theorems Related to Cylindrical Approximation and Convergence}, \ref{app: Proofs and Derivations}, and \ref{app: Details about our FDEs}. In particular, App.~\ref{app: Theorems Related to Cylindrical Approximation and Convergence} is dedicated to the mathematical background for non-experts of functional analysis, making our paper as self-contained as possible.
    \item[] Guidelines:
    \begin{itemize}
        \item The answer NA means that the paper does not include theoretical results. 
        \item All the theorems, formulas, and proofs in the paper should be numbered and cross-referenced.
        \item All assumptions should be clearly stated or referenced in the statement of any theorems.
        \item The proofs can either appear in the main paper or the supplemental material, but if they appear in the supplemental material, the authors are encouraged to provide a short proof sketch to provide intuition. 
        \item Inversely, any informal proof provided in the core of the paper should be complemented by formal proofs provided in appendix or supplemental material.
        \item Theorems and Lemmas that the proof relies upon should be properly referenced. 
    \end{itemize}

    \item {\bf Experimental Result Reproducibility}
    \item[] Question: Does the paper fully disclose all the information needed to reproduce the main experimental results of the paper to the extent that it affects the main claims and/or conclusions of the paper (regardless of whether the code and data are provided or not)?
    \item[] Answer: \answerYes{}
    \item[] Justification: All the information for reproduction is given in Sec.~\ref{sec: Experiment} and App.~\ref{app: Detailed Experimental Settings}. We also provide our code in the supplementary materials. See also our GitHub \url{https://github.com/TaikiMiyagawa/FunctionalPINN}.
    \item[] Guidelines:
    \begin{itemize}
        \item The answer NA means that the paper does not include experiments.
        \item If the paper includes experiments, a No answer to this question will not be perceived well by the reviewers: Making the paper reproducible is important, regardless of whether the code and data are provided or not.
        \item If the contribution is a dataset and/or model, the authors should describe the steps taken to make their results reproducible or verifiable. 
        \item Depending on the contribution, reproducibility can be accomplished in various ways. For example, if the contribution is a novel architecture, describing the architecture fully might suffice, or if the contribution is a specific model and empirical evaluation, it may be necessary to either make it possible for others to replicate the model with the same dataset, or provide access to the model. In general. releasing code and data is often one good way to accomplish this, but reproducibility can also be provided via detailed instructions for how to replicate the results, access to a hosted model (e.g., in the case of a large language model), releasing of a model checkpoint, or other means that are appropriate to the research performed.
        \item While NeurIPS does not require releasing code, the conference does require all submissions to provide some reasonable avenue for reproducibility, which may depend on the nature of the contribution. For example
        \begin{enumerate}
            \item If the contribution is primarily a new algorithm, the paper should make it clear how to reproduce that algorithm.
            \item If the contribution is primarily a new model architecture, the paper should describe the architecture clearly and fully.
            \item If the contribution is a new model (e.g., a large language model), then there should either be a way to access this model for reproducing the results or a way to reproduce the model (e.g., with an open-source dataset or instructions for how to construct the dataset).
            \item We recognize that reproducibility may be tricky in some cases, in which case authors are welcome to describe the particular way they provide for reproducibility. In the case of closed-source models, it may be that access to the model is limited in some way (e.g., to registered users), but it should be possible for other researchers to have some path to reproducing or verifying the results.
        \end{enumerate}
    \end{itemize}

\item {\bf Open access to data and code}
    \item[] Question: Does the paper provide open access to the data and code, with sufficient instructions to faithfully reproduce the main experimental results, as described in supplemental material?
    \item[] Answer: \answerYes{} 
    \item[] Justification: Our code is given in the supplementary materials and GitHub (\url{https://github.com/TaikiMiyagawa/FunctionalPINN}).
    \item[] Guidelines:
    \begin{itemize}
        \item The answer NA means that paper does not include experiments requiring code.
        \item Please see the NeurIPS code and data submission guidelines (\url{https://nips.cc/public/guides/CodeSubmissionPolicy}) for more details.
        \item While we encourage the release of code and data, we understand that this might not be possible, so “No” is an acceptable answer. Papers cannot be rejected simply for not including code, unless this is central to the contribution (e.g., for a new open-source benchmark).
        \item The instructions should contain the exact command and environment needed to run to reproduce the results. See the NeurIPS code and data submission guidelines (\url{https://nips.cc/public/guides/CodeSubmissionPolicy}) for more details.
        \item The authors should provide instructions on data access and preparation, including how to access the raw data, preprocessed data, intermediate data, and generated data, etc.
        \item The authors should provide scripts to reproduce all experimental results for the new proposed method and baselines. If only a subset of experiments are reproducible, they should state which ones are omitted from the script and why.
        \item At submission time, to preserve anonymity, the authors should release anonymized versions (if applicable).
        \item Providing as much information as possible in supplemental material (appended to the paper) is recommended, but including URLs to data and code is permitted.
    \end{itemize}

\item {\bf Experimental Setting/Details}
    \item[] Question: Does the paper specify all the training and test details (e.g., data splits, hyperparameters, how they were chosen, type of optimizer, etc.) necessary to understand the results?
    \item[] Answer: \answerYes{}
    \item[] Justification: All the information for reproduction is given in Sec.~\ref{sec: Experiment} and App.~\ref{app: Detailed Experimental Settings}. We also provide our code in the supplementary materials.
    \item[] Guidelines:
    \begin{itemize}
        \item The answer NA means that the paper does not include experiments.
        \item The experimental setting should be presented in the core of the paper to a level of detail that is necessary to appreciate the results and make sense of them.
        \item The full details can be provided either with the code, in appendix, or as supplemental material.
    \end{itemize}

\item {\bf Experiment Statistical Significance}
    \item[] Question: Does the paper report error bars suitably and correctly defined or other appropriate information about the statistical significance of the experiments?
    \item[] Answer: \answerYes{}
    \item[] Justification: Error bars (standard deviations) are reported in Sec.~\ref{sec: Experiment} and App.~\ref{app: Additional Results I}.
    \item[] Guidelines:
    \begin{itemize}
        \item The answer NA means that the paper does not include experiments.
        \item The authors should answer "Yes" if the results are accompanied by error bars, confidence intervals, or statistical significance tests, at least for the experiments that support the main claims of the paper.
        \item The factors of variability that the error bars are capturing should be clearly stated (for example, train/test split, initialization, random drawing of some parameter, or overall run with given experimental conditions).
        \item The method for calculating the error bars should be explained (closed form formula, call to a library function, bootstrap, etc.)
        \item The assumptions made should be given (e.g., Normally distributed errors).
        \item It should be clear whether the error bar is the standard deviation or the standard error of the mean.
        \item It is OK to report 1-sigma error bars, but one should state it. The authors should preferably report a 2-sigma error bar than state that they have a 96\% CI, if the hypothesis of Normality of errors is not verified.
        \item For asymmetric distributions, the authors should be careful not to show in tables or figures symmetric error bars that would yield results that are out of range (e.g. negative error rates).
        \item If error bars are reported in tables or plots, The authors should explain in the text how they were calculated and reference the corresponding figures or tables in the text.
    \end{itemize}

\item {\bf Experiments Compute Resources}
    \item[] Question: For each experiment, does the paper provide sufficient information on the computer resources (type of compute workers, memory, time of execution) needed to reproduce the experiments?
    \item[] Answer: \answerYes{}
    \item[] Justification: Compute resources are reported in App.~\ref{app: Detailed Experimental Settings}.
    \item[] Guidelines:
    \begin{itemize}
        \item The answer NA means that the paper does not include experiments.
        \item The paper should indicate the type of compute workers CPU or GPU, internal cluster, or cloud provider, including relevant memory and storage.
        \item The paper should provide the amount of compute required for each of the individual experimental runs as well as estimate the total compute. 
        \item The paper should disclose whether the full research project required more compute than the experiments reported in the paper (e.g., preliminary or failed experiments that didn't make it into the paper). 
    \end{itemize}
    
\item {\bf Code Of Ethics}
    \item[] Question: Does the research conducted in the paper conform, in every respect, with the NeurIPS Code of Ethics \url{https://neurips.cc/public/EthicsGuidelines}?
    \item[] Answer: \answerYes{}
    \item[] Justification: We carefully read the NeurIPS Code of Ethics and checked our work complies it.
    \item[] Guidelines:
    \begin{itemize}
        \item The answer NA means that the authors have not reviewed the NeurIPS Code of Ethics.
        \item If the authors answer No, they should explain the special circumstances that require a deviation from the Code of Ethics.
        \item The authors should make sure to preserve anonymity (e.g., if there is a special consideration due to laws or regulations in their jurisdiction).
    \end{itemize}

\item {\bf Broader Impacts}
    \item[] Question: Does the paper discuss both potential positive societal impacts and negative societal impacts of the work performed?
    \item[] Answer: \answerNA{} 
    \item[] Justification: This paper presents work whose goal is to advance the field of functional differential equations, and there is no societal impact which we feel must be specifically highlighted here.
    \item[] Guidelines:
    \begin{itemize}
        \item The answer NA means that there is no societal impact of the work performed.
        \item If the authors answer NA or No, they should explain why their work has no societal impact or why the paper does not address societal impact.
        \item Examples of negative societal impacts include potential malicious or unintended uses (e.g., disinformation, generating fake profiles, surveillance), fairness considerations (e.g., deployment of technologies that could make decisions that unfairly impact specific groups), privacy considerations, and security considerations.
        \item The conference expects that many papers will be foundational research and not tied to particular applications, let alone deployments. However, if there is a direct path to any negative applications, the authors should point it out. For example, it is legitimate to point out that an improvement in the quality of generative models could be used to generate deepfakes for disinformation. On the other hand, it is not needed to point out that a generic algorithm for optimizing neural networks could enable people to train models that generate Deepfakes faster.
        \item The authors should consider possible harms that could arise when the technology is being used as intended and functioning correctly, harms that could arise when the technology is being used as intended but gives incorrect results, and harms following from (intentional or unintentional) misuse of the technology.
        \item If there are negative societal impacts, the authors could also discuss possible mitigation strategies (e.g., gated release of models, providing defenses in addition to attacks, mechanisms for monitoring misuse, mechanisms to monitor how a system learns from feedback over time, improving the efficiency and accessibility of ML).
    \end{itemize}
    
\item {\bf Safeguards}
    \item[] Question: Does the paper describe safeguards that have been put in place for responsible release of data or models that have a high risk for misuse (e.g., pretrained language models, image generators, or scraped datasets)?
    \item[] Answer: \answerNA{} 
    \item[] Justification: Our paper poses no such risks.
    \item[] Guidelines:
    \begin{itemize}
        \item The answer NA means that the paper poses no such risks.
        \item Released models that have a high risk for misuse or dual-use should be released with necessary safeguards to allow for controlled use of the model, for example by requiring that users adhere to usage guidelines or restrictions to access the model or implementing safety filters. 
        \item Datasets that have been scraped from the Internet could pose safety risks. The authors should describe how they avoided releasing unsafe images.
        \item We recognize that providing effective safeguards is challenging, and many papers do not require this, but we encourage authors to take this into account and make a best faith effort.
    \end{itemize}

\item {\bf Licenses for existing assets}
    \item[] Question: Are the creators or original owners of assets (e.g., code, data, models), used in the paper, properly credited and are the license and terms of use explicitly mentioned and properly respected?
    \item[] Answer: \answerYes{}
    \item[] Justification: We appropriately cite relevant assets. The license and terms of use are mentioned in our code.
    \item[] Guidelines:
    \begin{itemize}
        \item The answer NA means that the paper does not use existing assets.
        \item The authors should cite the original paper that produced the code package or dataset.
        \item The authors should state which version of the asset is used and, if possible, include a URL.
        \item The name of the license (e.g., CC-BY 4.0) should be included for each asset.
        \item For scraped data from a particular source (e.g., website), the copyright and terms of service of that source should be provided.
        \item If assets are released, the license, copyright information, and terms of use in the package should be provided. For popular datasets, \url{paperswithcode.com/datasets} has curated licenses for some datasets. Their licensing guide can help determine the license of a dataset.
        \item For existing datasets that are re-packaged, both the original license and the license of the derived asset (if it has changed) should be provided.
        \item If this information is not available online, the authors are encouraged to reach out to the asset's creators.
    \end{itemize}

\item {\bf New Assets}
    \item[] Question: Are new assets introduced in the paper well documented and is the documentation provided alongside the assets?
    \item[] Answer: \answerYes{} 
    \item[] Justification: We carefully read and follow the NeurIPS Code and Data Submission Guidelines.
    \item[] Guidelines:
    \begin{itemize}
        \item The answer NA means that the paper does not release new assets.
        \item Researchers should communicate the details of the dataset/code/model as part of their submissions via structured templates. This includes details about training, license, limitations, etc. 
        \item The paper should discuss whether and how consent was obtained from people whose asset is used.
        \item At submission time, remember to anonymize your assets (if applicable). You can either create an anonymized URL or include an anonymized zip file.
    \end{itemize}

\item {\bf Crowdsourcing and Research with Human Subjects}
    \item[] Question: For crowdsourcing experiments and research with human subjects, does the paper include the full text of instructions given to participants and screenshots, if applicable, as well as details about compensation (if any)? 
    \item[] Answer: \answerNA
    \item[] Justification: Our paper does not involve crowdsourcing nor research with human subjects.
    \item[] Guidelines:
    \begin{itemize}
        \item The answer NA means that the paper does not involve crowdsourcing nor research with human subjects.
        \item Including this information in the supplemental material is fine, but if the main contribution of the paper involves human subjects, then as much detail as possible should be included in the main paper. 
        \item According to the NeurIPS Code of Ethics, workers involved in data collection, curation, or other labor should be paid at least the minimum wage in the country of the data collector. 
    \end{itemize}

\item {\bf Institutional Review Board (IRB) Approvals or Equivalent for Research with Human Subjects}
    \item[] Question: Does the paper describe potential risks incurred by study participants, whether such risks were disclosed to the subjects, and whether Institutional Review Board (IRB) approvals (or an equivalent approval/review based on the requirements of your country or institution) were obtained?
    \item[] Answer: \answerNA
    \item[] Justification: Our paper does not involve crowdsourcing nor research with human subjects.
    \item[] Guidelines:
    \begin{itemize}
        \item The answer NA means that the paper does not involve crowdsourcing nor research with human subjects.
        \item Depending on the country in which research is conducted, IRB approval (or equivalent) may be required for any human subjects research. If you obtained IRB approval, you should clearly state this in the paper. 
        \item We recognize that the procedures for this may vary significantly between institutions and locations, and we expect authors to adhere to the NeurIPS Code of Ethics and the guidelines for their institution. 
        \item For initial submissions, do not include any information that would break anonymity (if applicable), such as the institution conducting the review.
    \end{itemize}

\end{enumerate}

\end{document}